\documentclass[a4paper]{article}

\usepackage{amsmath,amssymb,amsthm}
\usepackage{graphicx}
\usepackage[english]{babel}
\usepackage[utf8]{inputenc}
\usepackage[T1]{fontenc}
\usepackage{mathtools}
\usepackage[margin=1cm]{caption}
\usepackage{bbm}
\usepackage[margin=3cm]{geometry}
\usepackage{xcolor}
\usepackage{hyperref}
\usepackage{tikz}
\usetikzlibrary{patterns}
\usetikzlibrary{intersections}
\allowdisplaybreaks

\newcommand{\N}{\mathbb{N}}
\newcommand{\Z}{\mathbb{Z}}
\newcommand{\Q}{\mathbb{Q}}
\newcommand{\R}{\mathbb{R}}

\newcommand{\E}{\mathbb{E}}
\newcommand{\ud}{\,\mathrm{d}}
\newcommand{\dd}{\mathsf{d}}

\newcommand{\PP}{\mathbb{P}}

\newcommand{\A}{\mathcal{A}}
\newcommand{\I}{\mathbbm{1}}
\newcommand{\e}{\varepsilon}
\newcommand{\Pow}{\mathfrak{P}}
\newcommand{\Sphere}{\mathbb{S}}
\DeclareMathOperator{\Rea}{Re}

\newcommand{\lmic}{\lambda_{\mathrm{mic}}}
\newcommand{\lmac}{\lambda_{\mathrm{mac}}}
\newcommand{\Qbox}{\mathcal{Q}}
\newcommand{\Pbox}{\mathcal{P}}
\newcommand{\Qmic}{\Qbox_{\lmic}}

\DeclareMathOperator{\supp}{supp}

\DeclareMathOperator{\diam}{diam}
\DeclareMathOperator{\dist}{dist}

\numberwithin{equation}{section}
\numberwithin{figure}{section}

\renewcommand{\notin}{\not\in}
\newcommand{\notni}{\not\ni}

\newtheorem{theorem}{Theorem}[section]
\newtheorem{corollary}[theorem]{Corollary}

\newtheorem{lemma}[theorem]{Lemma}
\theoremstyle{remark}
\newtheorem{remark}[theorem]{Remark}
\theoremstyle{definition}
\newtheorem{definition}[theorem]{Definition}

\begin{document}

\title{Pinning for the critical and supercritical membrane model}
\author{Florian Schweiger\footnote{Institut für angewandte Mathematik, Universität Bonn, Endenicher Allee 60, 53115 Bonn, Germany, E-Mail: \texttt{schweiger@iam.uni-bonn.de}}}
\date{\today}
\maketitle
\begin{abstract}
The membrane model is a Gaussian interface model with a Hamiltonian involving second derivatives of the interface height. We consider the model in dimension $\dd\ge4$ under the influence of $\delta$-pinning of strength $\e$. It is known that this pinning potential manages to localize the interface for any $\e>0$. We refine this result by establishing the $\e$-dependence of the variance and of the exponential decay rate of the covariances for small $\e$ (similar to the corresponding results for the discrete Gaussian free field by Bolthausen-Velenik). We also show the existence of a thermodynamic limit of the field. These conclusions improve upon earlier works by Bolthausen-Cipriani-Kurt and by Sakagawa.

The problem has similarities to the homogenization of elliptic operators in randomly perforated domains, and our proof takes inspiration from this connection. The main new ideas are a correlation inequality for the set of pinned points, and a probabilistic Widman hole filler argument which relies on a discrete multipolar Hardy-Rellich inequality and on a multi-scale argument to construct suitable test functions.
\end{abstract}

\section{Introduction}
\subsection{Setting and overview}

In this work we focus on discrete stochastic interface models describing random height functions. Given a Hamiltonian $H$ describing the energy of the interface and a finite subset $\Lambda$ of $\Z^\dd$, one obtains such an interface model by considering the probability measure
\[\PP_H(\mathrm{d}\psi)=\frac{1}{Z_{\beta,H,\Lambda}}\exp(-\beta H(\psi))\prod_{x\in \Lambda}\mathrm{d}\psi_x\prod_{x\in\Z^\dd\setminus \Lambda}\delta_0(\mathrm{d}\psi_x)\]
where $\beta>0$ is the inverse temperature. The choices $H(\psi)=\frac12\sum_{x\in\Z^\dd}|\nabla_1\psi_x|^2$ (where $\nabla_1\psi_x:=(\psi_{x+e_i}-\psi_x)_{i=1}^\dd$ is the discrete gradient) and $\beta=1$ lead to the discrete Gaussian free field (or gradient model). This is the most important example of a discrete stochastic interface model. It describes a Gaussian measure with zero mean, and covariance given as the Green's function of the discrete Laplacian. Its critical dimension (where covariances decay logarithmically) is given by $\dd=2$. This model is by now very well understood. We cannot survey the literature here, but refer the reader to \cite{Funaki2005,Velenik2006} for reviews on the gradient model and discrete stochastic interface models in general.

The present work, however, concerns a different model, the so-called membrane model, first studied in \cite{Sakagawa2003,Kurt2007}. Here one chooses $H(\psi)=\frac12\sum_{x\in\Z^\dd}|\Delta_1\psi_x|^2$ where $\Delta_1 \psi_x:=\sum_{i=1}^\dd\psi_{x+e_i}-2\psi_x+\psi_{x-e_i}$ is the discrete Laplacian, and $\beta=1$. This leads to the probability measure 
\begin{equation}\label{e:mem_mod}
\PP_\Lambda(\mathrm{d}\psi)=\frac1{Z_\Lambda}\exp\left(-\frac12\sum_{x\in \Z^\dd}|\Delta_1\psi_x|^2\right)\prod_{x\in \Lambda}\mathrm{d}\psi_x\prod_{x\in\Z^\dd\setminus \Lambda}\delta_0(\mathrm{d}\psi_x)\,.
\end{equation}
This is again a centred Gaussian measure, with covariance given by the Green's function $G_\Lambda$ of the discrete Bilaplacian. Its critical dimension is $\dd=4$, and one expects the $\dd$-dimensional membrane model to behave similarly to the $(\dd-2)$-dimensional gradient model. Making this intuition rigorous, however, is often a challenging problem. This is due to the fact that some of the most useful tools used in the study of the gradient model cannot be applied to the membrane model. In particular, there is no random walk representation of the field, and the field does not satisfy the FKG inequality.

Nonetheless, in recent years several results that were already known for the gradient model could be established also for the membrane model. Let us mention the scaling limit of the membrane model \cite{Cipriani2019}, the maximum of the field \cite{Chiarini2016,Cipriani2019,Schweiger2020}, and its behaviour under entropic repulsion \cite{Sakagawa2003,Kurt2007,Kurt2009,Buchholz2019}.

In this work we shall continue this line of research by investigating the problem of pinning. This means that one adds a small attractive potential which rewards the field for being equal to (or close to) 0, thereby breaking the continuous symmetry of the field. The physical motivation for this is mainly that it serves as a stepping stone for understanding the phenomenon of wetting, where one considers the competition between pinning and entropic repulsion, i.e. between the attractive potential near 0 and a hard wall that forces the field to be nonnegative. Depending on the strength of the pinning one or the other factor might win. In fact, for the gradient model it is known that for $\dd\le2$ there is a phase transition, while for $\dd\ge3$ pinning always wins \cite{Bolthausen2000,Caputo2000}. For the membrane model there are no rigorous results on that problem beyond the case $\dd=1$ \cite{Caravenna2008,Caravenna2009}.

From now on we focus on pinning. We consider the pinning potential $\e\delta_0$ with $\delta_0$ a point-mass at 0. That is, instead of \eqref{e:mem_mod} we consider the probability measure 
\begin{equation}\label{e:mem_mod_pinned}
\PP_\Lambda^\e(\mathrm{d}\psi)=\frac1{Z_\Lambda^\e}\exp\left(-\frac12\sum_{x\in \Z^\dd}|\Delta_1\psi_x|^2\right)\prod_{x\in \Lambda}(\mathrm{d}\psi_x+\e\delta_0(\mathrm{d}\psi_x))\prod_{x\in\Z^\dd\setminus \Lambda}\delta_0(\mathrm{d}\psi_x)\,.
\end{equation}
Each sample of $\PP_\Lambda^\e$ will contain some $x\in\Lambda$ where $\psi_x=0$. We call those $x$ the pinned points.
As for wetting, one can ask whether there is a phase transition as the pinning strength $\e$ varies. It turns out that the answer is yes in dimension 1, and no in $\dd\ge2$. More precisely, if $\dd\ge2$ or $\dd=1$ and $\e>\e_c$ for some $\e_c>0$ the expected fraction of points in $\Lambda$ that are pinned is bounded below by a constant (uniformly in $\Lambda$) \cite{Caravenna2008,Sakagawa2012,Sakagawa2018}.

Thus, the pinning effect manages to localize the field in the sense that it touches the 0-plane on a positive fraction of $\Lambda$. It is natural to ask whether this localization also manifests itself in some other ways. One of the main results of this paper is that this is the case, at least in the critical and supercritical cases $\dd\ge4$: for any $\e>0$ the variance of the field is bounded, and the covariance decays exponentially in the distance (i.e. a mass is generated). Physically speaking, this corresponds to a finite transverse and longitudinal correlation length, respectively.

Actually, our results go beyond this: We study the field at the onset of pinning (i.e. for $\e>0$ small), and establish the dependence of the variance on $\e$. We also prove lower bounds on the dependence of the mass on $\e$ that we believe to be optimal. This is not the first result in that direction: in \cite{Bolthausen2016} Bolthausen, Cipriani and Kurt proved stretched-exponential decay of the covariance in $\dd\ge4$, and in \cite{Bolthausen2017} they improved this to exponential decay, if $\dd\ge5$. Their rate of decay is far from optimal, though. Thus, our results are novel in two regards: We can prove the exponential decay also in the critical case $\dd=4$, and we can establish the (presumably) optimal lower bound on that rate of decay in the critical and the supercritical dimensions. As a corollary of these bounds we can also show that a thermodynamic limit of the field exists if $\dd\ge4$.

Our results leave open the question whether similar results on mass and variance hold in dimension $\dd\le3$. For $\dd=1$ and $\e>\e_c$ these should follow from the renewal theory methods in \cite{Caravenna2008,Caravenna2009}. For $\dd=2,3$, however, there are no known results. In fact, in \cite{Bolthausen2017} the authors wrote "it is well possible that exponential decay of correlations is true also for lower dimensions $\dd=2,3$, but we do not know of a method which could successfully be applied", and we have nothing to add to this statement.

For the gradient model it is known that there is no phase transition for any $\dd\ge1$. This and much finer results were established in a series of papers \cite{Dunlop1992,Bolthausen1999,Deuschel2000,Ioffe2000} culminating in \cite{Bolthausen2001}. There, Bolthausen and Velenik actually studied a more general class of gradient models in dimension $\dd\ge2$, and proved bounds for variance and mass similar to ours.

Our proof follows closely \cite{Bolthausen2001} to prove estimates on the set of pinned points, and uses some important ideas from \cite{Bolthausen2017} to prove exponential decay of the covariances. The main novelties are the following: we notice that the Gaussian correlation inequality \cite{Royen2014} implies a FKG inequality for the set of pinned points. This allows us to prove the existence of the thermodynamic limit of the set of pinned points (and later also the existence of the thermodynamic limit of the field). We also use a certain monotonicity property of the variances that allows us to adapt the proofs from \cite{Bolthausen2001} to show estimates on the set of pinned points. Regarding the exponential decay of the covariances, we take from \cite{Bolthausen2017} the idea to use a Widman hole filler argument \cite{Widman1971} on random annuli. The details, however, are rather different. We use a multipolar Hardy-Rellich inequality for second derivatives (inspired by similar inequalities for first derivatives as e.g. in \cite{Cazacu2013}) to estimate the local effect of the pinned points. We also use a rather subtle multiscale construction to construct the required cut-off functions, and we prove that this construction can be done with sufficiently high probability. This is the most technical part of the present paper, and it is novel to the best of our knowledge.

\subsection{Main results}
Let us describe our results in detail. First of all, expanding the bracket in \eqref{e:mem_mod_pinned}, we see that for $f\colon\R^{\Z^d}\to\R$ we have
\begin{equation}\label{e:pinnedmeasureexpansion}
\begin{aligned}
	\E_\Lambda^\e(f)&=\frac1{Z_\Lambda^\e}\int\exp\left(-\frac12\sum_{x\in \Z^\dd}|\Delta_1\psi_x|^2\right)f(\psi)\prod_{x\in \Lambda}(\mathrm{d}\psi_x+\e\delta_0(\mathrm{d}\psi_x))\prod_{x\in\Z^\dd\setminus \Lambda}\delta_0(\mathrm{d}\psi_x)\\
	&=\frac1{Z_\Lambda^\e}\sum_{A\subset\Lambda}\int\exp\left(-\frac12\sum_{v\in \Z^\dd}|\Delta_1\psi_x|^2\right)f(\psi)\e^{|A|}\prod_{x\in \Lambda\setminus A}\mathrm{d}\psi_v\prod_{x\in\Z^\dd\setminus (\Lambda\setminus A)}\delta_0(\mathrm{d}\psi_x)\\
	&=\sum_{A\subset\Lambda}\frac{\e^{|A|}Z_{\Lambda\setminus A}}{Z_\Lambda^\e}\E_{\Lambda\setminus A}(f)
\end{aligned}
\end{equation}
where $\E_\Lambda^\e$ and $\E_{\Lambda\setminus A}$ denote the expectation with respect to $\PP_\Lambda^\e$ and $\PP_{\Lambda\setminus A}$, respectively. Thus, we have
\begin{equation}\label{e:pinnedmeasureasamixture}
	\PP_\Lambda^\e(\mathrm{d}\psi)=\sum_{A\subset\Lambda}\zeta^\e_\Lambda(A)\PP_{\Lambda\setminus A}(\mathrm{d}\psi)
\end{equation}
where
\[\zeta^\e_\Lambda(A)=\frac{\e^{|A|}Z_{\Lambda\setminus A}}{Z_\Lambda^\e}\,,\]
so that $\zeta^\e_\Lambda$ is a probability measure on $\Pow(\Lambda)$, the powerset of $\Lambda$. It describes the set of pinned points. In fact, one easily sees that for any $A\subset\Lambda$ we have
\[A=\{x\in\Lambda\colon \psi_x=0\}\quad\text{$\PP_{\Lambda\setminus A}$-almost surely}\,.\]

By \eqref{e:pinnedmeasureasamixture}, $\PP_\Lambda^\e$ is a mixture of the Gaussian measures $\PP_{\Lambda\setminus A}$ for $A\subset\Lambda$. Our first goal will therefore be to understand the weights of this mixture, i.e. the measure $\zeta^\e_\Lambda$. We write $\zeta^\e_\Lambda(f)$ for $\sum_{A\subset\Lambda}f(A)\zeta^\e_\Lambda(A)$.  The first result is that the measure $\zeta^\e_\Lambda$ satisfies a correlation inequality.
\begin{theorem}\label{t:fkg}
The measure $\zeta^\e_\Lambda$ satisfies the FKG inequality, i.e.
\[\zeta^\e_\Lambda(fg)\ge\zeta^\e_\Lambda(f)\zeta^\e_\Lambda(g)\]
for any pair of increasing functions $f,g\colon \Pow(\Lambda)\to\R$.
\end{theorem}
This FKG inequality allows us to prove directly that a thermodynamic limit of the $\zeta^\e_\Lambda$ exists. We can also prove that a thermodynamic limit of the $\PP^\e_\Lambda$ exists. That result relies on the estimates for the Green's function which we state in Theorem \ref{t:estimatesfield} below.
\begin{theorem}\label{t:thermolimit}
If $\dd\ge4$, the thermodynamic limits
\begin{align*}
	\zeta^\e&:=\lim_{\Lambda\nearrow\Z^\dd}\zeta^\e_\Lambda\,,\\
	\PP^\e&:=\lim_{\Lambda\nearrow\Z^\dd}\PP_\Lambda^\e
\end{align*}
exist and are translation invariant.
\end{theorem}
The convergence here is meant as weak convergence of measures on $\R^{\Z^\dd}$ equipped with the cylinder $\sigma$-algebra, i.e. we claim that the measures integrated against any bounded local function converge.

It is easy to see that $\PP^\e$ is an infinite volume Gibbs measure for the interaction \eqref{e:mem_mod_pinned} (with appropriate boundary conditions). We write $\E^\e$ for the expectation with respect to $\PP^\e$.

We will now state a few results on $\zeta^\e_\Lambda$ and $\PP^\e_\Lambda$ that hold uniformly in $\Lambda$. Theorem \ref{t:thermolimit} then implies that they hold for $\zeta^\e$ and $\PP^\e$ as well.

We begin with precise estimates on the pinned set. The heuristic is that this set behaves like a Bernoulli point process with density $p_\dd$ depending on $\e$. It turns out that this is true in a rather strong sense if $\dd\ge5$. In $\dd=4$ this no longer holds, but fortunately we can still compare the probabilities that large sets are free of pinned points, and this is sufficient to continue with our argument. The precise result is the following. For the definition of strong stochastic domination see Definition \ref{d:stoch_dom}. We denote by $\A$ a random variable distributed according to $\zeta^\e_\Lambda$.
\begin{theorem}\label{t:estimatespinnedset} Let $\dd\ge4$. There are constants $c_\dd,C_\dd$, $\e_{\dd,*}$ depending on $\dd$ only with the following property.
\begin{itemize}
	\item[a)] If $\dd\ge5$ and $p_{\dd,-}=c_\dd\e$, then for any $\Lambda\Subset\Z^d$ and any $\e<\e_{\dd,*}$ the measure $\zeta^\e_\Lambda$ strongly dominates the Bernoulli measure on $\Pow(\Lambda)$ with parameter $p_{\dd,-}$. In particular for any $E\subset\Lambda$
\begin{equation}\label{e:lowerestpinnedset_dim5}
	(1-p_{\dd,-})^{|E|}\ge\zeta^\e_\Lambda(\A\cap E=\varnothing)\,.
\end{equation}
	\item[b)] If $\dd\ge5$ and $p_{\dd,+}=C_\dd\e$, then for any $\Lambda\Subset\Z^d$ and any $\e<\e_{\dd,*}$ the measure $\zeta^\e_\Lambda$ is strongly dominated by the Bernoulli measure on $\Pow(\Lambda)$ with parameter $p_{\dd,+}$. In particular for any $E\subset\Lambda$
\begin{equation}\label{e:upperestpinnedset_dim5}
	(1-p_{\dd,+})^{|E|}\le\zeta^\e_\Lambda(\A\cap E=\varnothing)\,.
\end{equation}
	\item[c)] If $\dd=4$ and $p_{4,-}=c_4\frac{\e}{|\log\e|^{1/2}}$, then for any $E\subset\Lambda$ and any $\e<\e_{4,*}$ we have
\begin{equation}\label{e:lowerestpinnedset_dim4}
(1-p_{4,-})^{|E|}\ge\zeta^\e_\Lambda(\A\cap E=\varnothing)\,.
\end{equation}
	\item[d)] If $\dd=4$, then there is for any $\alpha>0$ a constant $C_{4,\alpha}$ depending on $\dd$ and $\alpha$ such that with $p_{4,+,\alpha}=C_{4,\alpha}\frac{\e}{|\log\e|^{1/2}}$ for any $E\subset\Lambda$ with $d(E,\Z^\dd\setminus\Lambda)\ge\e^{-\alpha}$ and any $\e<\e_{4,*}$ we have
\begin{equation}\label{e:upperestpinnedset_dim4}
(1-p_{4,+,\alpha})^{|E|}\le\zeta^\e_\Lambda(\A\cap E=\varnothing)\,.
\end{equation}
\end{itemize}
All estimates also hold with $\zeta^\e$ in place of $\zeta^\e_\Lambda$.
\end{theorem}
Let us warn the reader that we use the notation $p_{\dd,\pm}$ in the opposite way as in \cite{Bolthausen2001}. Our convention here follows \cite{Bolthausen2017}.

Note carefully that we do not claim any domination result in case $\dd=4$. In fact, the same argument as in \cite[Section 2]{Bolthausen2001} shows that neither $\zeta^\e_\Lambda$ is strongly dominated by the Bernoulli measure on $\Pow(\Lambda)$ with parameter $p_{4,+}$, nor that $\zeta^\e_\Lambda$ strongly dominates the Bernoulli measure on $\Pow(\Lambda)$ with parameter $p_{4,-}$.

In the subcritical dimensions $\dd<4$ the set of pinned points is too correlated for any meaningful comparison with a Bernoulli measure. This is the reason why new techniques would be necessary to study the pinned membrane model in dimensions 2 and 3.

From Theorems \ref{t:thermolimit} and \ref{t:estimatespinnedset} one immediately obtains the following corollary, which strengthens Sakagawa's result \cite{Sakagawa2012} that the density of pinned points is positive for any $\e>0$.
\begin{corollary}\label{c:density_pinned_points}
Let $\dd\ge4$. Consider the density of pinned points
\[\rho_\e=\liminf_{\Lambda\nearrow\Z^\dd}\frac1{|\Lambda|}\zeta^\e_\Lambda(|\A|)=\liminf_{\Lambda\nearrow\Z^\dd}\frac1{|\Lambda|}\sum_{A\subset\Lambda}|A|\zeta^\e_\Lambda(A)\,.\]
For each $\e>0$ we have $\rho_\e\ge c_\dd p_{\dd,-}>0$.
\end{corollary}
It is unclear whether the limit here exists in general. However, using Theorem \ref{t:fkg} and a subadditivity argument one can show that it exists along the sequence $\Lambda_n=[-n,n]^\dd\cap\Z^\dd$, say.

\bigskip

Using the knowledge about $\zeta^\e_\Lambda$ from Theorem \ref{t:estimatespinnedset} we can now establish some more precise results on $\PP^\e$.  For any vector $\theta\in \Sphere^{\dd-1}$ (where $\Sphere^{\dd-1}$ is the unit sphere in $\R^\dd$) define the mass
\begin{equation}\label{e:def_mass}
m_\e(\theta):=-\limsup_{k\to\infty}\frac1k\log |\E^\e(\psi_0\psi_{\lfloor k\theta\rfloor})|\,.
\end{equation}
where we set $\log0=-\infty$. Note that we take the absolute value of $\E^\e(\psi_0\psi_{\lfloor k\theta\rfloor})$ in this definition. This is because for the membrane model correlations can be negative. In fact, the heuristic in Section \ref{s:heuristics} below suggests that $\E^\e(\psi_0\psi_{\lfloor k\theta\rfloor})$ behaves like an underdamped harmonic oscillator. In particular, we expect that the limit in \eqref{e:def_mass} does not exist. In contrast, for the gradient model the limit in \eqref{e:def_mass} exists, even without the absolute values, cf. \cite[Appendix A]{Bolthausen2001}.

We can show the following results on the variance, covariance and mass. Here $d(x,E)$ denotes the distance from $x$ to the set $E$.
\begin{theorem}\label{t:estimatesfield} Let $\dd\ge4$, and $\Lambda\Subset\Z^\dd$. There are constants $c_\dd,C_\dd$, $\e_{\dd,**}$ depending on $\dd$ only with the following property.
\begin{itemize}
	\item[a)] Let $x\in\Lambda$. Then for $\e<\e_{\dd,**}$ we have the following estimates on the variance:
	if $\dd\ge5$, then
	\begin{equation}\label{e:est_var_d>4}
	c_\dd\le\E^\e_\Lambda(\psi_x^2)\le C_\dd\,,
	\end{equation}
	while if $\dd=4$ and $\alpha>0$, then
	\begin{equation}\label{e:est_var_d=4}
	\frac{|\log\e|}{32\pi^2}-C_{4,\alpha}\log|\log\e|\le\E^\e_\Lambda(\psi_x^2)\le \frac{|\log\e|}{16\pi^2}+C_4\log|\log\e|\,,
	\end{equation}
	where the lower bound only holds if $d(x,\Z^\dd\setminus\Lambda)\ge \e^{-\alpha}+\e^{-1/4}$.
	The same estimates hold for $\E^\e$ instead of $\E^\e_\Lambda$ (with the condition on $d(x,\Z^d\setminus\Lambda)$ becoming vacuous). 
	\item[b)] Let $x,y\in\Lambda$. Then for $\e<\e_{\dd,**}$ we have the following estimates on the variance:
	if $\dd\ge5$, then
	\begin{equation}\label{e:est_cov_d>4}
	|\E^\e_\Lambda(\psi_x\psi_y)|\le \frac{C_\dd}{\e^{1/2}}\exp\left(-c_\dd\e^{1/4}|x-y|\right)\,,
	\end{equation}
	while if $\dd=4$, then
	\begin{equation}\label{e:est_cov_d=4}
	|\E^\e_\Lambda(\psi_x\psi_y)|\le \frac{C_4|\log\e|^{5/4}}{\e^{1/2}}\exp\left(-\frac{c_4\e^{1/4}|x-y|}{|\log\e|^{3/8}}\right)\,.
	\end{equation}
	The same estimates hold for $\E^\e$ instead of $\E^\e_\Lambda$.
	
	In particular, we have the following estimates on the mass: 	
	if $\dd\ge5$, then
	\begin{equation}\label{e:est_mass_d>4}
	c_\dd\e^{1/4}\le m_\e(\theta)\qquad\forall \theta\in \Sphere^{\dd-1}\,,
	\end{equation}
	while if $\dd=4$, then
	\begin{equation}\label{e:est_mass_d=4}
	c_4\frac{\e^{1/4}}{|\log\e|^{3/8}}\le m_\e(\theta) \qquad\forall \theta\in \Sphere^3\,.
	\end{equation}
\end{itemize}
\end{theorem}
The estimates in Theorem \ref{t:estimatesfield} are only valid for sufficiently small $\e>0$. However, a calculation similar to \eqref{e:pinnedmeasureexpansion} reveals that for $\e'<\e$ the measure $\PP^{(\e)}_{\Lambda}$ is a mixture of the measures $\PP^{(\e')}_{\Lambda\setminus A}$, and so the theorem also implies that for any $\e>0$ the measure $\PP^{(\e)}_{\Lambda}$ has bounded variances and exponentially decaying covariances.

In the next section we describe some heuristics for the exponential decay of the correlations. These heuristics suggest that the exponential rates in \eqref{e:est_cov_d>4} and \eqref{e:est_cov_d=4} are optimal, but the prefactors are not. In fact, we have made no real effort to optimise these prefactors, as this would require further technicalities. Nonetheless, as we believe that the exponential rates in \eqref{e:est_cov_d>4} and \eqref{e:est_cov_d=4} are optimal, the same holds for the rates in \eqref{e:est_mass_d>4} and \eqref{e:est_mass_d=4}.

The results of Theorem \ref{t:estimatesfield} are a significant improvement of the results in \cite{Bolthausen2016,Bolthausen2017}. In \cite{Bolthausen2017} it is shown that for $\dd\ge5$ the mass is positive for each fixed $\e>0$. No explicit lower bound on the mass is given, but if one keeps track of the constants in their argument one can check that their proof gives the estimate
\[c_\dd\e^{\frac{(2\dd+4)(2\dd+1)}{\dd}}\le m_\e(\theta)\,.\]
For $\dd=4$ in \cite{Bolthausen2016} stretched-exponential decay of the covariances is shown, and it was unknown whether the decay is actually exponential.

\subsection{Heuristics: The continuous Bilaplace equation in a perforated domain}\label{s:heuristics}
Before we describe the proofs of our results in more detail, let us discuss a related problem that provides some heuristics. Namely we consider the continuous Bilaplace equation in a domain perforated by small holes. This is a well-studied problem, and the analogous problem for the Laplacian even more so, cf. \cite{Cioranescu1997,Marchenko2006}\footnote{Note that the original French and Russian works date back to the 70s and 80s.}. If one lets the size of the holes tend to zero while keeping their capacity density constant, the problem converges (in an appropriate sense) to a Bilaplace equation with a mass term on the whole domain. The Green's function of the associated operator decays exponentially, and so it is unsurprising that the same holds true already for the Green's function in the perforated domain.

In our context, this connection gives a hint how to deduce Theorem \ref{t:estimatesfield} if one assumes Theorem \ref{t:estimatespinnedset}. Let us explain this in detail: fix some $0<r<\frac12$. Let $\e>0$, and let $\Omega\subset\R^\dd$ be a bounded domain with smooth boundary. Let $N$ be a large parameter to be chosen later. We perforate the domain $N\Omega$ with small holes of radius $r>0$, centred at a subset of $\Z^\dd$. Theorem \ref{t:estimatespinnedset} suggests that we choose a fraction $p_{\dd,-}$ of the points in $\Z^\dd$ as the centres of these holes. For now we consider the simplest case of equally-spaced holes, i.e. we place them at $(\lmic\Z)^\dd$, where $\lmic\approx (p_{\dd,-})^{-\frac1\dd}$ is an integer. That is, we consider the equation
\begin{alignat}{2}\label{e:heuristics1}
\begin{aligned}
\Delta^2 u&=f && \text{in }N\Omega\setminus\bigcup_{x\in(\lmic\Z)^\dd}\overline{B_r(x)}\,,\\
u&=0 && \text{else}\,.
\end{aligned}
\end{alignat}
The Green's function $G$ of this problem should predict the behaviour of the covariances in Theorem \ref{t:estimatesfield}.

We can rescale \eqref{e:heuristics1} back to a unit domain by letting $\hat f(y)=\frac{1}{N^{\dd-4}}f(Ny)$, $\hat u(y)=\frac{1}{N^\dd}u(Ny)$, so that $\hat u$ and $\hat f$ solve
\begin{alignat}{2}\label{e:heuristics2}
\begin{aligned}
\Delta^2 \hat u&=\hat f && \text{in }\Omega\setminus\bigcup_{x\in((\lmic/N)\Z)^\dd}\overline{B_{r/N}(x)}\,,\\
\hat u&=0 && \text{else}\,.
\end{aligned}
\end{alignat}
In order to apply now results from \cite{Cioranescu1997}, we need to treat $\dd\ge5$ and $\dd=4$ separately. We begin with the former case. The collection of balls $\bigcup_{x\in((\lmic/N)\Z)^\dd}B_{r/N}(x)$ has capacity density
\[\mu=C_\dd\left(\frac{N}{\lmic}\right)^\dd\left(\frac{r}{N}\right)^{\dd-4}=C_\dd N^4r^{\dd-4}\lmic^{-\dd}\approx C_\dd N^4r^{\dd-4}p_{\dd,-}=C_\dd N^4r^{\dd-4}\e\,.\]
We want to consider a limit of dense small holes where $\mu$ is constant, and so we choose $N=\e^{-1/4}$. Then, according to \cite[Example 2.14]{Cioranescu1997}, the solution of \eqref{e:heuristics2} in the limit $\e\to0$ behaves like the solution of
\begin{alignat}{2}\label{e:heuristics3}
\begin{aligned}
\Delta^2 \hat u+\mu\hat u&=\hat f && \text{in }\Omega\,,\\
\hat u&=0 && \text{else}\,.
\end{aligned}
\end{alignat}
This is a Bilaplace equation with a mass term. Its Green's function $\hat G$ behaves like the Green's function $\hat G_{\R^\dd}$ of the same equation in the full space $\R^\dd$ (at least when we stay away from the boundary of $\Omega$). The latter Green's function can be computed quite explicitly using separation of variables in spherical coordinates. One finds that $\hat G_{\R^\dd}(\hat x,\hat y)=F(\mu^{1/4}|x-y|)$, where $F(t)$ is a linear combination of $\Rea\left((\zeta_8 t)^{-(\dd-2)/2}H^{(1)}_{(\dd-2)/2}(\zeta_8 t)\right)$. Here $H^{(1)}_\nu$ is the Hankel function of the first kind, and $\zeta_8$ runs through the primitive eighth roots of unity. A short calculation using the asymptotic expansion for these functions (cf. \cite[Equation 9.7.2]{Abramowitz1964}) and the fact that $\hat G_{\R^\dd}$ needs to decay at infinity reveals that
\begin{align*}
	&\hat G_{\R^\dd}(\hat x,\hat y)\\*
	&=C_\dd\left(\mu^{1/4}|\hat x-\hat y|\right)^{-(\dd-1)/2}\left(\sin\left(\frac{\mu^{1/4}|\hat x-\hat y|}{2^{1/2}}-\omega_\dd\right)+O(\mu^{-1/4}|\hat x-\hat y|^{-1})\right)\exp\left(-\frac{\mu^{1/4}|\hat x-\hat y|}{2^{1/2}}\right)
\end{align*}
where $\omega_\dd$ is a phase shift depending only on $\dd$, and we used the standard Landau notation. Neglecting the error term altogether, we thus expect
\begin{equation}\label{e:heuristics4}
\hat G(\hat x,\hat y)\approx C_\dd\left(\mu^{1/4}|\hat x-\hat y|\right)^{-(\dd-1)/2}\sin\left(\frac{\mu^{1/4}|\hat x-\hat y|}{2^{1/2}}-\omega_\dd\right)\exp\left(-\frac{\mu^{1/4}|\hat x-\hat y|}{2^{1/2}}\right)
\end{equation}
when $|\hat x-\hat y|\gg\mu^{-1/4}$, and the Green's function of \eqref{e:heuristics2} should behave similarly (at least if $\Omega$ is large enough, i.e. $\diam \Omega\gg\mu^{-1/4}$).
Rescaling back, we thus expect for the Green's function $G$ of \eqref{e:heuristics1} that
\begin{align*}
&G(x,y)\\
&\quad=\frac{1}{N^{\dd-4}}\hat G\left(\frac{x}{N},\frac{y}{N}\right)\\
&\quad\approx \frac{1}{N^{\dd-4}}C_\dd\left(\frac{\mu^{1/4}|x-y|}{2^{1/2}N}\right)^{-(\dd-1)/2}\sin\left(-\frac{\mu^{1/4}|x-y|}{2^{1/2}N}-\omega_\dd\right)\exp\left(-\frac{\mu^{1/4}|x-y|}{N}\right)\\
&\quad\approx\frac{C_\dd\e^{(\dd-7)/8}}{r^{(\dd-1)(\dd-4)/8}|x-y|^{(\dd-1)/2}}\sin\left(-C_\dd\e^{1/4}r^{(\dd-4)/4}|x-y|-\omega_\dd\right)\exp\left(-C_\dd\e^{1/4}r^{(\dd-4)/4}|x-y|\right)
\end{align*}
when $|x-y|\gg N\mu^{-1/4}=\e^{-1/4}r^{-(\dd-4)/4}$. Thus, $G$ decays exponentially, with polynomial corrections and an oscillatory term that makes $G$ change sign. While the polynomial corrections and the oscillatory term are not captured in \eqref{e:est_cov_d>4}, the exponential decay rates in both estimates are the same (up to constant factors).

If $\dd=4$, the argument is in principle the same, but we need to use extra care when defining the capacity density. Following \cite[Example 2.14]{Cioranescu1997} we define
\[\mu=C\left(\frac{N}{\lmic}\right)^4\frac{1}{|\log\frac{r}{N}|}\approx CN^4p_{4,-}\frac{1}{\log N-\log r}\approx C\frac{N^4\e}{(\log N-\log r)|\log\e|^\frac12}\,.\]
We again want $\e\to0$ while $\mu$ is constant, and so we choose $N=\frac{|\log\e|^{3/8}}{\e^{1/4}}+O_r\left(\frac{|\log\e|^{1/8}\log|\log\e|}{\e^{1/4}}\right)=\frac{|\log\e|^{3/8}+o_r(|\log\e|^{3/8})}{\e^{1/4}}$ accordingly. Then we conclude from \cite[Example 2.14]{Cioranescu1997} that the solutions of \eqref{e:heuristics2} and \eqref{e:heuristics3} are close. The Green's function of \eqref{e:heuristics3} in $\dd=4$ still behaves like \eqref{e:heuristics4}, and so, rescaling back, we find again
\begin{align*}
&G(x,y)\\
&\quad=\hat G\left(\frac{x}{N},\frac{y}{N}\right)\\
&\quad\approx C\left(\frac{\mu^{1/4}|x-y|}{N}\right)^{-3/2}\sin\left(-\frac{\mu^{1/4}|x-y|}{2^{1/2}N}-\omega_4\right)\exp\left(-\frac{\mu^{1/4}|x-y|}{2^{1/2}N}\right)\\
&\quad\approx \frac{C|\log\e|^{9/16}}{\e^{3/8}|x-y|^{3/2}}\sin\left(-\frac{C\e^{1/4}|x-y|}{|\log\e|^{3/8}+o_r(|\log\e|^{3/8}))}-\omega_4\right)\exp\left(-\frac{C\e^{1/4}|x-y|}{|\log\e|^{3/8}+o_r(|\log\e|^{3/8}))}\right)\,.
\end{align*}
when $|x-y|\gg_r N\mu^{-1/4}\approx\e^{-1/4}|\log\e|^{3/8}$.
This is again exponential decay with polynomial corrections and as oscillatory term. The exponential decay rate is again the same as in \eqref{e:est_cov_d=4} (up to constants).

In summary, our heuristic predicts the same exponential decay rates as in Theorem \ref{t:estimatesfield}. The heuristic we used is rather simplistic, though. One problem is that in the context of the membrane model a single pinned point forces the field to be zero there, but does not pose any restrictions on the gradient of the field. In contrast, in \eqref{e:heuristics1} we force the field and all its derivatives to be zero at the pinned balls. One way to improve the heuristics would thus be to only prescribe that $u$ has average zero over each $B_r(x)$ for $x\in(\lmic\Z)^\dd$, instead of it being identically zero there. This is not a serious change, though, as a modification of \cite[Example 2.14]{Cioranescu1997} or an application of the general framework in \cite{Marchenko2006} show that the convergence of \eqref{e:heuristics2} to \eqref{e:heuristics3} still holds, albeit with a different constant prefactor in $\mu$.

A more serious problem is that the pinned points are not distributed on a lattice, but following the probability distribution $\zeta^\e_\Lambda$. If this distribution were, say, a Poisson point process, then the framework from \cite{Marchenko2006} would still apply.
Our actual $\zeta^\e_\Lambda$ is possibly quite correlated (at least if $\dd=4$), though, and so it is not clear that the heuristic still applies. On the other hand, we are not actually interested in "quenched" estimates that hold for all realizations of the sets of pinned points, but rather in "annealed" estimates where we average over the randomness of the pinned points. So there is hope to retain the heuristic.

A further question is how to rigorously show that the convergence of the boundary value problem \eqref{e:heuristics2} to the boundary value problem \eqref{e:heuristics3} implies that the Green's function of \eqref{e:heuristics2} already has the predicted behaviour. There are very few results on this in the literature. One exception is \cite{Niethammer2006}, where this is proved rigorously for the case of the Laplace equation in $\dd=3$. However, that approach relies on the maximum principle, and so one cannot extend it to our situation. Instead, in \cite{Hoefer2018} a more robust approach is used: There (in another context) exponential decay of the $L^2$-norm of harmonic functions on perforated large annuli is shown, using Widman's hole filler technique \cite{Widman1971} in combination with the fact that one has a local Poincaré inequality. A similar argument is also used in \cite{Bolthausen2017}, and the authors describe that they learned it from Vladimir Maz'ya. The decay rates in \cite{Hoefer2018} are not optimal, but a small modification of their argument leads to the optimal decay rate. These arguments are the inspiration for our proof of Theorem \ref{t:estimatesfield} from Theorem \ref{t:estimatespinnedset}. We shall explain this in more detail in the next subsection, where we outline the proofs of our results.

\subsection{Main ideas of the proofs}\label{s:main_ideas}

This paper consists of two main parts. We first establish the various results on the pinned set, and then deduce from them the results on the variances and covariances. We will discuss these parts separately. Before doing so, let us remark that there are two natural lengthscales occuring. There are the average distance between pinned points
\[\lmic\approx\begin{cases}\frac{1}{\e^{1/\dd}}&\dd\ge5\\\frac{|\log\e|^{1/8}}{\e^{1/4}}&\dd=4\end{cases}\,,\]
and the lengthscale on which the correlations decay
\[\lmac\approx\begin{cases} \frac{1}{\e^{1/4}}&\dd\ge5\\\frac{|\log\e|^{3/8}}{\e^{1/4}}&\dd=4\end{cases}\,.\]
Note that we have $1\ll\lmic\ll\lmac$ as $\e\to0$ for any $\dd\ge4$.

\subsubsection{Estimates on the pinned set}
The first novel result of this paper is the FKG inequality for the pinned set, Theorem \ref{t:fkg}. As already mentioned, it follows rather directly from the Gaussian correlation inequality \cite{Royen2014}, and it is standard to deduce from the FKG inequality the existence of the thermodynamic limit of the set of pinned points, i.e. the first part of Theorem \ref{t:thermolimit}. We give these proofs in Section \ref{s:corr_ineq}. Note, however, that our proof of Theorem \ref{t:fkg} is specific to the case of $\delta$-pinning, and we conjecture that the result is not true for other pinning potentials such as a square-well potential. The point is that conditioning a Gaussian vector on being 0 at some coordinates yields another Gaussian vector, but that is no longer true if we condition on some coordinates being small instead. 
We give a more detailed explanation in Remark \ref{r:fkg_counterex}. 

For the proof of Theorem \ref{t:estimatespinnedset} in Section \ref {s:est_pinned_set} we follow \cite{Bolthausen2001} rather closely. The domination results in Theorem \ref{t:estimatespinnedset} a) and b) are actually already in \cite{Bolthausen2017}. They follow via a short calculation from the boundedness of the Green's function in $\dd\ge5$. Part d) is a little more difficult. It could be proven as in \cite[Section 3.2]{Bolthausen2001}, but we give a slightly simpler proof. The idea is that if $x\in E$ is quite far from the pinned points we have already found, the fluctuations of $\psi_x$ are quite large, and so the chance that $x$ is pinned is low.

By far the most difficult part of Theorem \ref{t:estimatespinnedset} is part c), where we again mostly follow \cite{Bolthausen2001}. There we want to control the probability that $E\subset\Lambda$ is free of pinned points from above. To do so, we need to find for any configuration of pinned points that avoids $E$ many others that intersect $E$. This is done using a two-scale argument. We first consider the case that $E$ is a union of boxes of sidelength $C\lmic$, and prove \eqref{e:lowerestpinnedset_dim4} in this case by carefully tracking how a pinned point in one of these boxes makes it likely that there are pinned points in the neighbouring boxes. Next, we pass to the larger lengthscale $C\lmac$ and deduce from the first step that for an arbitrary $E$, most points of $E$ are at a distance $\le C\lmac$ from a pinned point. Finally we use this knowledge together with an argument similar to the first step to construct many configurations of pinned points that intersect $E$.

The main difference to \cite{Bolthausen2001} is that one cannot use random walk estimates to see how pinning at some $x\in\Lambda$ influences the variance at $y\neq x$. Instead we use an explicit variance estimate (Lemma \ref{l:estvariance}) that follows from the monotonicity of the variance in the set of pinned points. We also streamline the argument from \cite{Bolthausen2001} at some points and correct a minor mistake there.

\subsubsection{Asymptotics for the variances and covariances}
The remainder of the paper is then concerned with proving Theorem \ref{t:estimatesfield} and the second part of Theorem \ref{t:thermolimit}. In \cite{Bolthausen2001} the random walk representation of the Green's function of the Laplacian is used for that purpose. In our case there is no such representation, so we need a completely new argument.

It turns out that the estimates for the variance follow quite easily from Theorem \ref{t:estimatespinnedset} and the variance estimate in Lemma \ref{l:estvariance}. We give details in Sections \ref{s:quenched} and \ref{s:annealed}.

The estimates on the covariance in Theorem \ref{t:estimatespinnedset} are much more difficult. The general strategy is the same as in \cite{Bolthausen2017}. That is, we show that the $L^2$-norm of the second derivative of the "quenched" Green's function decays exponentially on large annuli. These annuli have to be chosen adapted to the set of pinned points, and so we do not get an estimate valid for all realizations of $\A$. But our estimates hold up to an exponentially small probability, so that we control $G_{\Lambda\setminus A}$ for all but exponentially few $A$. For these we can use a rather crude estimate. Finally, we can average these quenched estimates for the Green's function over $A$ to deduce "annealed" bounds for the covariance.

The existence of the thermodynamic limit of the field in Theorem \ref{t:thermolimit} follows then from the existence of the thermodynamic limit of the set of pinned points and the quenched decay estimates on the Green's function. The somewhat technical proof is given in Section \ref{s:thermo_limit_field}.

\bigskip

For the remainder of this section, let us describe in more detail how we prove the quenched estimates on the covariance. Our main technical result used for that purpose is, roughly speaking, the following (see Theorem \ref{t:decay_high_prob} for the precise statement):
there is a constant $\hat N_\dd$ such that if $k\in\N$ and $\e$ is sufficiently small there is an event $\Omega_{U,k}$ with $\zeta^\e_\Lambda\left(\Omega_{U,k}\right)\ge1-\frac{C_U}{2^k}$ such that if $A\in\Omega_{U,k}$, and if $u\colon\Z^\dd\to\R$ is a function such that $u=0$ on $A\setminus U$ and $u\Delta_1^2u=0$ on $\Z^\dd\setminus U$, we have the estimate
\begin{equation}\label{e:outline_ext_decay}
\left\|\nabla_1^2u\right\|^2_{L^2(\Z^\dd\setminus(U+Q_{2k\hat N_\dd\lmac}(0)))}\le\frac{1}{2^k}\left\|\nabla_1^2u\right\|^2_{L^2((U+Q_{2k\hat N_\dd\lmac}(0)\setminus U)}\,.
\end{equation}
Here $Q_{2k\hat N_\dd\lmac}(0)$ denotes a cube of halfdiameter $2k\hat N_\dd\lmac$ centred at $0$, and $U+Q_{2k\hat N_\dd\lmac}(0)$ is the Minkowski sum of two sets. This is an exterior decay estimate for biharmonic functions that holds up to exponentially small probability (and we state and prove in Theorem \ref{t:decay_high_prob} also the analogous interior decay estimate). Applying \eqref{e:outline_ext_decay} with $u=G_{\Lambda\setminus\A}(\cdot,y)$ it is a bit tedious but not difficult to deduce the aforementioned quenched estimates on the Green's function, and we do so in Sections \ref{s:quenched} and \ref{s:annealed}.

Let us describe how to prove \eqref{e:outline_ext_decay}. We first outline the basic strategy that was used in \cite{Bolthausen2017} and (in another context) in \cite{Hoefer2018}, and then describe our novel ideas. For convenience we pretend in the following that $u$ is a continuous function on $\R^\dd$. Adapting the argument to the discrete setting will be somewhat technical but not hard.

We try to iterate a Widman hole filler argument \cite{Widman1971} (see, e.g., \cite[Section 4.4]{Giaquinta2012} for a modern presentation). That is, given $U\subset\R^\dd$, we want to find $U'\supset U$, so that the $L^2$-norm of $\nabla^2u$ on $\R^\dd\setminus U$ is controlled by a constant less than 1 times the $L^2$-norm of $\nabla^2u$ on $U'\setminus U$. We also want $\dist(U,\R^\dd\setminus V)\le C\lmac$. Once we have such an estimate, we can iterate it to deduce exponential decay at rate $\frac{1}{C\lmac}$, at least on the $L^2$-level. 

So suppose that $U\subset U'$ are open sets and $\eta$ is a smooth cut-off function such that
\[\{\Delta^2u\neq0\}\subset U\subset\{\eta=0\}\subset\{\eta\neq1\}\subset U'\,.\]
Then we have
\begin{equation}\label{e:intbyparts1}
0=(\Delta^2u,\eta u)=(\nabla^2u,\nabla^2(\eta u))=\int\eta|\nabla^2u|^2+2\int\nabla^2u:\nabla u\otimes\nabla\eta+\int u\nabla^2u:\nabla^2\eta
\end{equation}
and one can rewrite this using the Cauchy-Schwarz inequality as
\begin{equation}\label{e:intbyparts4}
\begin{aligned}
\int_{\R^\dd\setminus U'}|\nabla^2u|^2&\le\int\eta|\nabla^2u|^2\\
&=-2\int\nabla^2u:\nabla u\otimes\nabla\eta-\int u\nabla^2u:\nabla^2\eta\\
&\le\frac15\int_{U'\setminus U}|\nabla^2u|^2+5\int_{U'\setminus U}|\nabla u|^2|\nabla\eta|^2+\frac15\int_{U'\setminus U}|\nabla^2u|^2+\frac54\int_{U'\setminus U}|u|^2|\nabla^2\eta|^2\,.
\end{aligned}
\end{equation}
Now if we could choose $\eta$ in such a way that the second and fourth summand here are both bounded by $\frac15\int_{U'\setminus U}|\nabla^2u|^2$ we would obtain the desired decay estimate. In fact, this is what was done in \cite{Bolthausen2017}. However, in order to bound both the second and fourth summand, one needs to impose strong pointwise conditions on $\nabla \eta$ and $\nabla^2\eta$, and, in particular, both need to be near zero on mesoscopic holes in the pinned set. These conditions do not allow growth of $\eta$ at the optimal rate, and so using this argument one cannot obtain the optimal estimate for the decay rate (but is it comparably easy to construct an $\eta$ that satisfies these conditions and grows at a non-optimal rate, cf. \cite{Bolthausen2017}).

To solve this problem we first rewrite the right hand side of \eqref{e:intbyparts1} so that there are no longer any terms containing $\nabla\eta$. An integration by parts shows that
\[\int\nabla^2u:\nabla u\otimes\nabla\eta=-\int\nabla u\cdot\left(\nabla\cdot\left(\nabla u\otimes\nabla\eta\right)\right)=-\int\nabla^2u:\nabla u\otimes\nabla\eta-\int|\nabla u|^2\Delta\eta\]
and hence
\[\int\nabla^2u:\nabla u\otimes\nabla\eta=-\frac12\int|\nabla u|^2\Delta\eta\,.\]
Plugging this into \eqref{e:intbyparts1} we see that
\begin{equation}\label{e:intbyparts2}
\int\eta|\nabla^2u|^2=\frac12\int|\nabla u|^2\Delta\eta-\int u\nabla^2u:\nabla^2\eta\,.
\end{equation}
Using the assumptions on $\eta$ and the Cauchy-Schwarz inequality we can now estimate
\begin{equation}\label{e:intbyparts3}
\begin{aligned}
\int_{\R^\dd\setminus U'}|\nabla^2u|^2&\le\int\eta|\nabla^2u|^2\\
&=\frac12\int|\nabla u|^2\Delta\eta-\int_{\R^\dd\setminus U} u\nabla^2u:\nabla^2\eta\\
&\le\frac12\int|\nabla u|^2\Delta\eta+\int|u|^2|\nabla^2\eta|^2+\frac14\int_{U'\setminus U}|\nabla^2u|^2\,.
\end{aligned}
\end{equation}
If we can now arrange things in such a way that the first two summands here are each bounded by $\frac14\int_{U'\setminus U}|\nabla^2u|^2$, we see that
\begin{equation}\label{e:intbyparts5}
\int_{\R^\dd\setminus U'}|\nabla^2u|^2\le\frac34\int_{U'\setminus U}|\nabla^2u|^2
\end{equation}
and now we can try to iterate this estimate to obtain exponential decay of the $L^2$-norm of $\nabla^2u$. Note that unlike \eqref{e:intbyparts4} we now only need to impose conditions on $\nabla^2\eta$. 

As it turns out, the first summand in \eqref{e:intbyparts3} can be controlled by the second and third summand using an interpolation inequality on lengthscale $\lmic$ that we discuss in \ref{s:inter_ineq}.

\bigskip

The remaining task is thus to choose $\eta$ in such a way that it grows fast enough, but we nonetheless can bound the term $\int|u|^2|\nabla^2\eta|^2$. For that purpose we need some sort of local Poincaré inequality on scale $\lmic$. Of course, such an estimate can only hold if there are enough pinned points close to the point of interest. In \cite[Lemma 4.1]{Bolthausen2017} this was done provided there is a nearby cube of $3^\dd$ points, which are all pinned. On that small cube we then have $u=\nabla u=0$, and some version of the Hardy-Rellich inequality forces $u$ to be small near that cube as well. However, this is not optimal, as cubes of $3^\dd$ points that are all pinned are very rare. We show that is sufficient if there are $\dd+1$ pinned points somewhere nearby that are well-spread out. The number $\dd+1$ arises from the fact that we need to eliminate nonzero affine functions on $\R^\dd$. Thus, in some sense we use a multipolar Hardy-Rellich inequality instead of a unipolar one. For multipolar Hardy inequalities cf. e.g. \cite{Cazacu2013}; we could not find a detailed discussion of multipolar Hardy-Rellich inequalities in the literature.

The local Poincaré inequality result is, roughly speaking, the following (see Theorem \ref{t:localpoincare} for the precise result): 
Let $A\subset \Lambda$, and $V\subset\Lambda$ be an arbitrary subset. Let $R\in\N$, $R\ge2$ be a parameter. Then
\[\left\|u\I_{\cdot\in X_R}\right\|^2_{L^2(V)}\le C_\dd R^\dd(1+\I_{\dd=4}\log R)\left\|\nabla^2u\right\|^2_{L^2(V+Q_R(0))}
\]
where $X_R$ is the set of those points that have $\dd+1$ well-separated pinned points at distance $\le R\lmic$ around them, and we write $\I_{\cdot\in X_R}$ for the indicator function of that set.

This result makes it clear what we need to require of $\eta$. Namely we want $|\nabla^2\eta|\le C\I_{\cdot\in X_R}$ for some $R$. If we have this relation, then our multipolar Hardy-Rellich inequality allows us to control the second term on the right hand side in \eqref{e:intbyparts3}, and we can close the argument for the exponential decay estimate.

It thus remains to choose $R$ and construct $\eta$ such that $|\nabla^2\eta|\le C\I_{\cdot\in X_R}$ in such a way that $\eta$ grows fast enough, and the construction should work up to an exponentially small probability. This is the content of Sections \ref{s:const_cutoff} and \ref{s:est_prob_cutoff}. This is the technical heart of the present paper, and the arguments are novel. We can think of $X_R$ as the good set, and its complement as the bad set, and we need to construct $\eta$ such that it is locally affine on the bad set, but still grows quadratically. To execute this construction, we start with one $\eta_*$ that grows quadratically, and then try to modify it so that it becomes affine on the bad set. For such modifications it is necessary that the components of the bad set are well-separated from each other. In general this will not be the case, but we make a multiscale composition of the bad set into parts that live on lengthscale $\ell_j$ and are well-separated on that lengthscale, and then we change $\eta_*$ to be affine on those parts separately. The correct choice of the lengthscales $\ell_j$ turns out to be the rather strange looking $\ell_j=CM^{j^3}\lmic$, where $M$ is some large integer. The construction can be carried out provided the multiscale decomposition vanishes beyond some large lengthscale. Using the results from Theorem \ref{t:estimatespinnedset} we show that this is the case up to a probability that can be made arbitrarily small. 

Unfortunately "arbitrarily small" is not quite good enough, as that means that there are still exceptional pairs $(U,U')$ on which we cannot deduce \eqref{e:intbyparts5}. But such exceptional pairs are rare, and when we iterate \eqref{e:intbyparts5} to conclude \eqref{e:outline_ext_decay} it is sufficient if we can apply \eqref{e:intbyparts5} on at least half of the possible $(U,U')$, which is possible up to an exponentially small probability.

This completes the construction of $\eta$. Once we have $\eta$ at our disposal, we can complete the proof of \eqref{e:outline_ext_decay}. We refer to Section \ref{s:decay_L2} for a more detailed exposition of the argument.

\subsection{Notation and preliminaries}
Our notation mostly follows \cite{Mueller2019,Schweiger2020}, but let us review the important points.

We use the convention that $c$ and $C$ denote generic constants whose precise value can change from occurrence to occurrence. Constants that are denoted by any other latin or greek letter have some fixed value and keep it. By adding subscripts to a constant we emphasize that the precise value of that constant may depend on the variables in the subscript (and typically on no others).

We let $e_1,\ldots,e_\dd$ be the standard unit vectors in $\R^\dd$. We define the forward difference quotient $D^1_iu(x)=u(x+e_i)-u(x)$ and the backward difference quotient $D^1_{-i}u(x)=u(x)-u(x-e_i)$. The discrete gradient is the vector $\nabla_1u(x):=(D^1_iu(x))_{i=1}^\dd$, the discrete Hessian is the tuple $\nabla^2_1u(x):=(D^1_iD^1_{-j}u(x))_{i,j=1}^\dd$, the discrete Laplacian is $\Delta_1u(x):=\sum_{i=1}^\dd D^1_iD^1_{-i}u(x)$, and the discrete Bilaplacian is $\Delta_1^2:=\Delta_1\circ\Delta_1$. For a multi-index $\underline{\alpha}\in\N^\dd$ we write $D^1_{\underline{\alpha}} u(x)=(D^1_1)^{\alpha_1}\ldots(D^1_\dd)^{\alpha_\dd}u(x)$, and for $k\in\N$ with $k\ge3$ we let $\nabla^ku(x)$ be the the collection of all $D^1_{\underline{\alpha}} u(x)$ with $|\underline{\alpha}|=k$.

We also use the translation operators $\tau^1_{\pm i}$ defined by $\tau^1_{\pm i}u(x)=u(x\pm e_i)$.

We will freely use various summation by part identities. These are all shown in detail in \cite{Mueller2019}.

For $r>0$ and $x\in\Z^4$ we let $Q_r(x)=(x+[-r,r]^4)\cap\Z^4$ be the discrete cube of diameter $2r$ around $x$. We will frequently use the Minkowski-sum of sets $E,E'$ defined by $E+E'=\{e+e'\colon e\in E,e'\in E'\}$. In particular, $E+Q_r(0)$ is the set of all points at distance $\le r$ from $E$.

Given $E\subset \Z^4$ and $u,v\colon E\to\R$, we define the discrete $L^2$-scalar product as $(u,v)_{L^2(E)}^2=\sum_{x\in E}u(x)v(x)$, the discrete $L^2$-norm $\|u\|_{L^2(E)}^2=\sum_{x\in E}|u(x)|^2$ as well as the discrete $L^\infty$-norm $\|u\|_{L^\infty(E)}=\sup_{x\in E}|u(x)|$. We extend these definitions to vector-valued functions by taking the Euclidean norm of the norms of the components.

For measures $\mu$ on $\Pow(\Lambda)$ we write $\mu(f)$ for $\int f\ud\mu=\sum_{A\subset\Lambda}f(A)\mu(A)$. We denote a sample from $\zeta^\e_\Lambda$ by $\A$. We define $\tilde A=A\cup(\Z^\dd\setminus\Lambda)$ for $A\subset\Lambda$ and analogously $\tilde\A=\A\cup(\Z^\dd\setminus\Lambda)$.
We let $G_{\Lambda\setminus A}$ be the discrete Green's function of $\Delta_1^2$ on $\Lambda\setminus A$, i.e. $G_{\Lambda\setminus A}(x,y):=\E_{\Lambda\setminus A}(\psi_x\psi_y)$. 

We use the standard Euclidean norm $|\cdot|$ as well as the $l^1$-norm $|\cdot|_1$ and the maximum norm $|\cdot|_\infty$ on $\R^\dd$. For $x\in\R^\dd$ and $E,E'\subset\R^\dd$ we let $d(x,E)=\inf_{y\in E}|x-y|$ be the distance of $x$ to $E$, and $d(E,E')=\inf_{y\in E,y'\in E'}|x-y|$ be the distance of $E$ and $E'$. Here we use the convention that the infimum of the empty set is $+\infty$. We also use these distances with respect to $|\cdot|_1$ and $|\cdot|_\infty$ instead $|\cdot|$, and in that case we write $d_1$ or $d_\infty$ instead of $d$.

We will use two different length scales $\lmic$ and $\lmac$. The former describes the typical distance between two pinned points which according to Theorem \ref{t:estimatespinnedset} is of the order $\frac{1}{\e^{1/\dd}}$ if $\dd\ge5$ and $\frac{|\log\e|^{1/8}}{\e^{1/4}}$ if $\dd=4$. We hence define
\[\lmic=\begin{cases}\frac{1}{\e^{1/\dd}}+\alpha_{\textrm{mic},5}(\e)&\dd\ge5\\\frac{|\log\e|^{1/8}}{\e^{1/4}}+\alpha_{\textrm{mic},4}(\e)&\dd=4\end{cases}\,.\]
Here $\alpha_{\textrm{mic},\dd}(\e)\in[0,2)$ is chosen in such a way that $\lmic$ is an odd integer.

The latter corresponds to the length scale on which correlations decay, and so, in line with Theorem \ref{t:estimatesfield} we set
\[\lmac=\begin{cases} \frac{1}{\e^{1/4}}+\alpha_{\textrm{mac},5}(\e)&\dd\ge5\\\frac{|\log\e|^{3/8}}{\e^{1/4}}+\alpha_{\textrm{mac},4}(\e)&\dd=4\end{cases}\]
where $\alpha_{\textrm{mac},\dd}(\e)\in[0,2\lmic)$ is chosen such that $\lmac$ is an odd multiple of $\lmic$. Note that for any $\dd\ge4$ we have $1\ll\lmic\ll\lmac$ as $\e\to0$. 

Given an odd integer $l>0$, we consider 
the set of $l$-boxes (or $l$-cubes)
\[\Qbox_l=\left\{Q_{l/2}(x)\colon x\in(l\Z)^\dd\right\}=\left\{\left(x+\left[-\frac{l}{2},\frac{l}{2}\right]^\dd\right)\cap\Z^\dd\colon x\in (l\Z)^\dd\right\}\]
and the set \[\Pbox_l=\left\{\bigcup_{Q\in I}Q\colon I\subset\Qbox_l,|I|<\infty\right\}\] of $l$-polymers. We identify each box with the polymer consisting just of that box. We call polymers connected if they are connected as subgraphs of $\Z^\dd$ with nearest-neighbour edges. We say that two polymers touch if they are disjoint but their union is connected.

The boxes in $\Qbox_l$ form a partition of $\Z^\dd$. Later on we will also need boxes with some overlap. Thus, if $l>0$ is an odd multiple of 3, we define
\[\Qbox^\#_l=\left\{Q_{l/2}(x)\colon x\in \left(\frac{l}{3}\Z\right)^\dd\right\}=\left\{x+\left[-\frac{l}{2},\frac{l}{2}\right]^\dd\cap\Z^\dd\colon x\in \left(\frac{l}{3}\Z\right)^\dd\right\}\,.\]
Then every point of $\Z^\dd$ is contained in precisely $3^\dd$ boxes in $\Qbox^\#_l$.

For some statement $s$ we let $\I_s$ be the indicator function of $s$, that is $\I_s=1$ if $s$ is true, and $\I_s=0$ else.

\section{Structure of the pinned set}

In this section we prove our results on the distribution of the pinned set, i.e. Theorem \ref{t:fkg}, the first part of Theorem \ref{t:thermolimit}, as well as Theorem \ref{t:estimatespinnedset}.

\subsection{Correlation inequalities}\label{s:corr_ineq}
We want to establish the FKG inequality for the set of pinned points in Theorem \ref{t:fkg}. We begin with a useful calculation. Let $A\subset A'\subset\Lambda\Subset\Z^\dd$. Then, using that $\delta_0(\mathrm{d}\psi)$ is a weak limit of the measures $\frac1{2t}\I_{\psi\in(-t,t)}\mathrm{d}\psi$ as $t\to0$, we have
\begin{equation}
\begin{aligned}\label{e:ratiopartsums}
	&\frac{Z_{\Lambda\setminus A'}}{Z_{\Lambda\setminus A}}\\
	&\quad=\frac{1}{Z_{\Lambda\setminus A}}\int\exp\left(-\frac12\sum_{x\in \Z^\dd}|\Delta_1\psi_x|^2\right)\prod_{x\in \Lambda\setminus A'}\mathrm{d}\psi_x\prod_{x\in\Z^\dd\setminus (\Lambda\setminus A')}\delta_0(\mathrm{d}\psi_x)\\
	&\quad=\frac1{Z_{\Lambda\setminus A}}\lim_{t\to0}\int\exp\left(-\frac12\sum_{x\in \Z^\dd}|\Delta_1\psi_v|^2\right)\prod_{x\in \Lambda\setminus A'}\mathrm{d}\psi_x\prod_{x\in A'\setminus A}\frac1{2t}\I_{\psi_x\in(-t,t)}\mathrm{d}\psi_x\prod_{x\in\Z^\dd\setminus (\Lambda\setminus A)}\delta_0(\mathrm{d}\psi_x)\\
	&\quad=\lim_{t\to0}\frac1{(2t)^{|A'\setminus A|}}\PP_{\Lambda\setminus A}(|\psi_x|<t\ \forall x\in A'\setminus A)\,.
\end{aligned}
\end{equation}
We can also interpret the right hand side as the density at zero of the Gaussian vector $(\psi_x)_{x\in A'\setminus A}$ under $\PP_{\Lambda\setminus A}$ (this observation was essentially already made in \cite[p. 143]{Velenik2006}). If $A'\setminus A=\{x\}$ is a singleton, the density of $\psi_x$ at 0 is equal to $\frac{1}{\sqrt{2\pi}}$ times the inverse of its standard deviation. We thus obtain the formula
\begin{equation}\label{e:ratiopartsums2}
\frac{Z_{\Lambda\setminus (A\cup\{x\})}}{Z_{\Lambda\setminus A}}=\frac{1}{\sqrt{2\pi G_{\Lambda\setminus A}(x,x)}}\,.
\end{equation}
\begin{proof}[Proof of Theorem \ref{t:fkg}]
We will prove the FKG lattice condition 
\begin{equation}\label{e:fkglattice}
\zeta^\e_\Lambda(A\cup A')\zeta^\e_\Lambda(A\cap A')\ge\zeta^\e_\Lambda(A)\zeta^\e_\Lambda(A')\quad\forall A,A'\subset\Lambda\,.
\end{equation}
It is well-known that this is a sufficient condition for the validity of the FKG inequality.

Now \eqref{e:fkglattice} is an easy consequence of the Gaussian correlation inequality \cite{Royen2014,Latala2017}. Indeed, note that by definition of $\zeta^\e_\Lambda$ the estimate \eqref{e:fkglattice} is equivalent to
\[Z_{\Lambda\setminus(A\cup A')}Z_{\Lambda\setminus(A\cap A')}\ge Z_{\Lambda\setminus A}Z_{\Lambda\setminus A'}\]
(here we used $|A\cup A'|+|A\cap A'|=|A|+|A'|$). Dividing both sides by $(Z_{\Lambda\setminus(A\cap A')})^2$ and using \eqref{e:ratiopartsums} we only have to verify
\begin{align*}
	&\lim_{t\to0}\PP_{\Lambda\setminus (A\cap A')}(|\psi_x|<t\ \forall x\in (A'\setminus A)\cup(A\setminus A'))\\
	&\quad\ge\lim_{t\to0}\PP_{\Lambda\setminus (A\cap A')}(|\psi_x|<t\ \forall x\in A'\setminus A)\PP_{\Lambda\setminus (A\cap A')}(|\psi_x|<t\ \forall x\in A\setminus A')\,.
\end{align*}
The sets $\{\psi\colon|\psi_x|<t\ \forall x\in A'\setminus A\}$ and $\{\psi\colon|\psi_x|<t\ \forall x\in A\setminus A'\}$ are convex and symmetric around the origin, and the measure $\PP_{\Lambda\setminus (A\cap A')}$ is Gaussian. Thus, the claim follows from the Gaussian correlation inequality, applied for each $t>0$.
\end{proof}
\begin{remark}\label{r:fkg_counterex}
In \cite{Bolthausen2001} it is shown that the set of pinned points for the gradient model satisfies a FKG inequality not only in the case of $\delta$-pinning, but also in the case of pinning by a square-well potential $b\I_{|\cdot|<a}$. The proof in \cite{Bolthausen2001} uses Griffiths' inequalities (as described in detail e.g. in \cite[Appendix A]{Dunlop1992}), and thus requires that the measure describing the field is an even ferromagnetic measure. This is certainly not the case for the membrane model, and so that proof cannot be applied in our setting.

Our proof of Theorem \ref{t:fkg} only used that $\PP_\Lambda$ is a non-degenerate Gaussian measure. However, this proof would not work if we considered pinning by a square-well potential $b\I_{|\psi|<a}$ instead of $\delta$-pinning. Namely, in this case we would need to consider $\PP_\Lambda(\cdot\mid |\psi_x|<a\ \forall x\in A\cap A')$ instead of $\PP_{\Lambda\setminus (A\cap A')}$, and the former measure is not Gaussian, so that we cannot apply the Gaussian correlation inequality.

This is not a shortcoming of our proof. Namely, we conjecture that the analogue of \eqref{e:fkglattice} in the case of pinning by a square-well potential is false. We do not have a counterexample for the case of the membrane model. However, we can give an example of a Gaussian measure where the set of pinned points with respect to a square-well potential does not satisfy \eqref{e:fkglattice}.

For this example, let $X_1,X_2$ be independent standard Gaussians, and $N>0$ a large parameter, and define
\[\begin{pmatrix}Y_1\\Y_2\\Y_3\\Y_4\\Y_5\\Y_6\end{pmatrix}=\begin{pmatrix}1&0\\0&1\\N&0\\0&N\\1&1\\1&-1\end{pmatrix}\begin{pmatrix}X_1\\X_2\end{pmatrix}\,.\]
Then $Y$ is a multivariate Gaussian vector. It is degenerate, but one can fix this later by adding some small Gaussian noise to it, so we will ignore that point. Let also $A=\{1,3,5,6\}$ and $A'=\{2,4,5,6\}$.

In this setting \eqref{e:fkglattice} would correspond to
\begin{equation}\label{e:fkg_example}
\PP\left(|Y_i|\le t\ \forall t\in A\cup A'\right)\PP\left(|Y_i|\le t\ \forall t\in A\cap A'\right)\ge\PP\left(|Y_i|\le t\ \forall t\in A\right)\PP\left(|Y_i|\le t\ \forall t\in A'\right)
\end{equation}
for any $t>0$. The probabilities here are equal to the Gaussian measure of certain sets in $\R^2$ (cf. Figure \ref{f:sets}). As $t\to0$, we can approximate this Gaussian measure by the Lebesgue measure, and thereby compute that
\begin{align*}
\lim_{t\to0}\frac{2\pi}{t^2}\PP\left(|Y_i|\le t\ \forall t\in A\cup A'\right)&=\frac{4}{N^2}\,,\\
\lim_{t\to0}\frac{2\pi}{t^2}\PP\left(|Y_i|\le t\ \forall t\in A\cap A'\right)&=2\,,\\
\lim_{t\to0}\frac{2\pi}{t^2}\PP\left(|Y_i|\le t\ \forall t\in A\right)&=\frac{4}{N}-\frac{2}{N^2}\,,\\
\lim_{t\to0}\frac{2\pi}{t^2}\PP\left(|Y_i|\le t\ \forall t\in A'\right)&=\frac{4}{N}-\frac{2}{N^2}\,.
\end{align*}
In particular, for $N$ large and $t$ small \eqref{e:fkg_example} is wrong by a factor arbitrarily close to 2.

\begin{figure}[ht]
\centering
\begin{tikzpicture}
\draw[->,dashed] (-2.2,0)--(2.3,0);
\draw[->,dashed] (0,-2.2)--(0,2.3);

\draw[-] (-0.1,-2.1)--(2.1,0.1);
\draw[-] (-0.1,2.1)--(2.1,-0.1);
\draw[-] (-2.1,-0.1)--(0.1,2.1);
\draw[-] (-2.1,0.1)--(0.1,-2.1);

\draw[-] (-2.1,0.3)--(2.1,0.3);
\draw[-] (-2.1,-0.3)--(2.1,-0.3);
\draw[-] (0.3,-2.1)--(0.3,2.1);
\draw[-] (-0.3,-2.1)--(-0.3,2.1);

\draw[-] (-2.1,2)--(2.1,2);
\draw[-] (-2.1,-2)--(2.1,-2);
\draw[-] (2,-2.1)--(2,2.1);
\draw[-] (-2,-2.1)--(-2,2.1);

\draw[pattern=north west lines,draw=none] (0,2)--(2,0)--(0,-2)--(-2,0)--cycle; 
\draw[pattern=north east lines,draw=none] (0.3,0.3)--(0.3,-0.3)--(-0.3,-0.3)--(-0.3,0.3)--cycle;
\draw[pattern=horizontal lines,draw=none] (0.3,2)--(0.3,-2)--(-0.3,-2)--(-0.3,2)--cycle;
\draw[pattern=vertical lines,draw=none] (2,0.3)--(-2,0.3)--(-2,-0.3)--(2,-0.3)--cycle;

\draw[|<->|] (2.4,-0.3)--(2.4,0.3) node[midway, right]{$\frac{2t}{N}$};
\draw[|<->|] (-0.3,2.4)--(0.3,2.4) node[midway, above]{$\frac{2t}{N}$};
\draw[|<->|] (-2.4,-2)--(-2.4,2) node[midway, left]{$2t$};
\draw[|<->|] (-2,-2.4)--(2,-2.4) node[midway, below]{$2t$};

\end{tikzpicture}
\caption{The sets associated to the probabilities in \eqref{e:fkg_example}. The product of the areas of the large and small square is about half the product of the two areas of the thin rectangles.}\label{f:sets}
\end{figure}
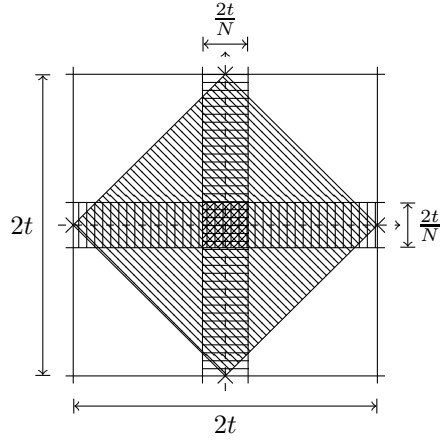

\end{remark}

Theorem \ref{t:fkg} directly implies the existence of a thermodynamic limit of the $\zeta^\e_\Lambda$:

\begin{proof}[Proof of Theorem \ref{t:thermolimit}, first part]
It suffices to check that $\lim_{\Lambda\nearrow\Z^\dd}\zeta^\e_\Lambda(f)$ exists for each bounded $f\colon \Z^\dd\to\R$ that is a local function (i.e. depends only on finitely many points). Each such $f$ is a linear combination of increasing functions, and so it actually suffices to check that $\lim_{\Lambda\nearrow\Z^\dd}\zeta^\e_\Lambda(f)$ exists for each local increasing $f$.

For that purpose note that Theorem \ref{t:fkg} implies that for any $\Lambda\subset\Lambda'\Subset\Z^\dd$ large enough so that $f$ only depends on the points in $\Lambda$, we have $\zeta^\e_\Lambda(f)\ge\zeta^\e_{\Lambda'}(f)$. Thus, $\lim_{R\to\infty}\zeta^\e_{Q_R(0)}(f)$ exists as a limit of a bounded decreasing sequence. Furthermore, for any $\Lambda\Subset\Z^\dd$ with $Q_r(0)\subset\Lambda\subset Q_R(0)$ we have 
\[\zeta^\e_{Q_r(0)}(f)\ge\zeta^\e_\Lambda(f)\ge\zeta^\e_{Q_R(0)}(f)\ge\lim_{R\to\infty}\zeta^\e_{Q_R(0)}(f)\,.\]
Since $\Lambda\nearrow\Z^\dd$ allows us to take $r\to\infty$, we see from this that $\lim_{\Lambda\nearrow\Z^\dd}\zeta^\e_\Lambda(f)$ exists and is equal to $\lim_{R\to\infty}\zeta^\e_{Q_R(0)}(f)$.

Thus, the unique weak limit $\zeta^\e$ exists. Its translation invariance follows from the fact that
\[\zeta^\e(f)=\lim_{\Lambda\nearrow\Z^\dd}\zeta^\e_\Lambda(f)=\lim_{\Lambda\nearrow\Z^\dd}\zeta^\e_{\Lambda+x}(f(\cdot-x))=\zeta^\e(f(\cdot-x))\]
for any $x\in\Z^\dd$.
\end{proof}

\subsection{Estimates on the pinned set}\label{s:est_pinned_set}
We will prove the various domination results of Theorem \ref{t:estimatespinnedset}. We first show some estimates on the variance of the membrane model. We begin with the straightforward proofs of part a) and b), then show part d), and finally part c). See Section \ref{s:main_ideas} for an outline of the proofs.

Let us first give the precise definition of (strong) domination, as in \cite{Bolthausen2001}.
\begin{definition}\label{d:stoch_dom}
Let $\Lambda$ be a finite set, and let $\nu,\nu'$ be two probability measures on $\Pow(\Lambda)$. We say that $\nu$ dominates $\nu'$ if we have
\[\nu(f)\ge\nu'(f)\]
for all increasing functions $f\colon\Pow(\Lambda)\to\R$. 
 We say that $\nu$ strongly dominates $\nu'$, if for all $x\in\Lambda$ and for all $E\subset\Lambda\setminus \{x\}$ we have
 \[\nu(A\ni x\mid A\setminus\{x\}=E)\ge\nu'(A\ni x\mid A\setminus\{x\}=E)\,.\]
\end{definition}
It is easy to see that strong stochastic domination implies stochastic domination, and the latter implies
\[\nu(A\cap E=\varnothing)\le \nu'(A\cap E=\varnothing)\quad\forall E\subset\Lambda\,.\]

Our proof of Theorem \ref{t:estimatespinnedset} is based on the proof of the corresponding result for the gradient model in \cite{Bolthausen2001}. We begin with some useful estimates on the variance of the membrane model.

The first one states the fact that the variance is non-increasing in the size of the pinned set.
\begin{lemma}\label{l:monotonicityvariance}
Let $A\subset A'\subset\Lambda\Subset\Z^\dd$, and let $x\in\Lambda$. Then $G_{\Lambda\setminus A'}(x,x)\le G_{\Lambda\setminus A}(x,x)$.
\end{lemma}
\begin{proof}
This follows easily from the Markov property of the field. See e.g. \cite[Corollary 3.2]{Bolthausen2017}.
\end{proof}
The preceding lemma allows us to conclude bounds on the variances.
\begin{lemma}\label{l:estvariance}
Let $\varnothing\neq A\subset\Lambda\Subset\Z^\dd$, and let $x\in\Lambda$. If $\dd\ge5$, we have
\begin{equation}\label{e:estvariance1}
c_\dd\le G_{\Lambda\setminus A}(x,x)\le C_\dd\,.
\end{equation}
If $\dd=4$, we have
\begin{equation}\label{e:estvariance2}
\frac{1}{8\pi^2}\log(1+d(x,A))-C\le G_{\Lambda\setminus A}(x,x)\le \frac{1}{4\pi^2}\log(1+d(x,A))+C\,.
\end{equation}
\end{lemma}
\begin{proof}
We begin with the upper bound in \eqref{e:estvariance2}. Let $a\in A$ be such that $|x-a|=d(x,A)$. Let $N\in\N$. For large enough $N$ we have $\Lambda\subset Q_N(x)$. Now Lemma \ref{l:monotonicityvariance} implies that
\begin{equation}\label{e:estvariance3}
G_{\Lambda\setminus A}(x,x)\le G_{Q_N(x)\setminus \{a\}}(x,x)
\end{equation}
for all $N$ large enough. The right hand side can be computed quite explicitly: We have
\[G_{Q_N(x)\setminus \{a\}}(x,x)=G_{Q_N(x)}(x,x)-\frac{G_{Q_N(x)}(a,x)^2}{G_{Q_N(x)}(a,a)}\]
and by \cite[Theorem 1.4]{Schweiger2020} we have
\[\left|G_{Q_N(x)}(y,y')-\frac{1}{8\pi^2}\log\left(\frac{N}{1+|y-y'|}\right)\right|\le C\]
for all $y,y'\in Q_N(x)$ with $d(y,\partial\Lambda_N)\ge cN$, $d(y',\partial\Lambda_N)\ge cN$.

Using this in \eqref{e:estvariance3} we find for $N$ large enough
\begin{align*}
G_{Q_N(x)\setminus \{a\}}(x,x)&\le\frac{1}{8\pi^2}\left(\log N+C-\frac{\left(\log\left(\frac{N}{1+|x-a|}\right)-C\right)^2}{\log N-C}\right)\\
&\le \frac{1}{4\pi^2}\log(1+|x-a|)+C
\end{align*}
and this implies the upper bound in \eqref{e:estvariance2}. The lower bound is similar, only this time we compare $G_{\Lambda\setminus A}(x,x)$ with $G_{Q_{d(x,A)-1}(x)}(x,x)$.

Finally, the proof of \eqref{e:estvariance1} is similar, using that $G_{Q_N(x)}(x,x)$ is bounded above and below if $\dd\ge5$.
\end{proof}

\begin{proof}[Proof of Theorem \ref{t:estimatespinnedset} a) and b)]
The two results are already proven in \cite[Lemma 3.4]{Bolthausen2017}. Nonetheless, we repeat the short argument: For $x\in\Lambda$, $E\subset\Lambda\setminus \{x\}$ we have
\begin{equation}\label{e:strongdomination}
\begin{aligned}
\zeta_\Lambda^\e(\A\ni x\mid \A\setminus\{x\}=E)&=\frac{\zeta_\Lambda^\e(E\cup\{x\})}{\zeta_\Lambda^\e(E\subset\A\subset E\cup\{x\})}\\
&=\frac{\zeta_\Lambda^\e(E\cup\{x\})}{\zeta_\Lambda^\e(E)+\zeta_\Lambda^\e(E\cup\{x\})}\\
&=\left(1+\frac{Z_{\Lambda\setminus E}}{\e Z_{\Lambda\setminus(E\cup\{x\})}}\right)^{-1}\\
&=\left(1+\frac{\sqrt{2\pi G_{\Lambda\setminus E}(x,x)}}{\e}\right)^{-1}
\end{aligned}
\end{equation}
where the last step follows from \eqref{e:ratiopartsums2}. Now in dimension $\dd\ge5$ we have $c_\dd\le G_{\Lambda\setminus E}(x,x)\le C_\dd$ by Lemma \ref{l:estvariance}, and this implies
\[c_\dd\e\le \zeta_\Lambda^\e(\A\ni x\mid \A\setminus\{x\}=E)\le C_\dd\e\]
for all $\e$ small enough. From this we immediately conclude the strong domination results from both sides, and these easily imply \eqref{e:lowerestpinnedset_dim5} and \eqref{e:upperestpinnedset_dim5}
\end{proof}
\begin{remark}\label{r:strongdomdim4}
When $\dd=4$ the calculation \eqref{e:strongdomination} is still valid, but we do no longer have a uniform upper bound on $G_{\Lambda\setminus E}(x,x)$. Let us point out for future use though that \eqref{e:strongdomination} and Lemma \ref{l:estvariance} imply that
\[\zeta_\Lambda^\e(\A\ni x\mid \A\setminus\{x\}=E)\le C\e\]
and thus the measure $\zeta_\Lambda^\e$ is strongly dominated by the Bernoulli measure on $\Pow(\Lambda)$ with parameter $p_{4,+}':=C\e$.
\end{remark}

\begin{proof}[Proof of Theorem \ref{t:estimatespinnedset} d)]
One could prove \eqref{e:upperestpinnedset_dim4} analogously as in \cite[Section 3.2]{Bolthausen2001}. We, however, give a slightly different proof in the following.

The events $\A\ni x$ for $x\in E$ are decreasing, and so by the FKG property of $\zeta^\e_\Lambda$ we have
\[\zeta^\e_\Lambda(\A\cap E=\varnothing)=\zeta^\e_\Lambda\left(\bigcap_{x\in E}\{\A\notni x\}\right)\ge\prod_{x\in E}\zeta^\e_\Lambda(\A\notni x)\,.\]
Thus, to establish \eqref{e:upperestpinnedset_dim4} it suffices to show 
\[\zeta^\e_\Lambda(\A\notni x)\ge 1-C_\alpha\frac{\e}{|\log\e|^{1/2}}\]
or equivalently
\begin{equation}\label{e:upperestpinned_1}
\zeta^\e_\Lambda(\A\ni x)\le C_\alpha\frac{\e}{|\log\e|^{1/2}}
\end{equation}
where the constant $C_\alpha$ depends only on $\alpha$.

For this we consider the box $Q:=Q_{\min(\e^{-\alpha},\e^{-1/5})}(x)$. We can write
\begin{equation}\label{e:upperestpinned_2}
\begin{aligned}
\zeta^\e_\Lambda(\A\ni x)&=\zeta^\e_\Lambda(\A\cap Q=\{x\})+\zeta^\e_\Lambda(\A\cap Q\supsetneq\{x\})\\
&\le\zeta^\e_\Lambda(\A\ni x\mid\A\cap(Q\setminus\{x\})=\varnothing)+\zeta^\e_\Lambda(\A\cap Q\supsetneq\{x\})\,.
\end{aligned}
\end{equation}
By Remark \ref{r:strongdomdim4} the second summand can be estimated as
\begin{equation}\label{e:upperestpinned_3}
\begin{aligned}
\zeta^\e_\Lambda(\A\cap Q\supsetneq\{x\})&\le p'_{4,+}-p'_{4,+}(1-p'_{4,+})^{|Q|}\\
&=C\e(1-(1-C\e)^{|Q|-1})\\
&\le C\e^2|Q|\\
&\le C\e^2\left(\e^{-1/5}\right)^4\\
&=C\e^{6/5}
\end{aligned}
\end{equation}
whenever $\e$ is small enough. For the first summand we can use the FKG property once more and then proceed as in \eqref{e:strongdomination} to see that
\begin{align*}
\zeta^\e_\Lambda(\A\ni x\mid\A\cap(Q\setminus\{x\})=\varnothing)&\le\zeta^\e_\Lambda(\A\ni x\mid\A\cap(Q\setminus\{x\})=\varnothing,\A\supset\Lambda\setminus Q)\\
&=\zeta^\e_Q(\A\ni x\mid\A\subset\{x\})\\
&=\frac{\zeta^\e_Q(\{x\})}{\zeta^\e_Q(\varnothing)+\zeta^\e_Q(\{x\})}\\
&=\left(1+\frac{Z_Q}{\e Z_{Q\setminus\{x\}}}\right)^{-1}\\
&=\left(1+\frac{\sqrt{2\pi G_Q(x,x)}}{\e}\right)^{-1}\,.
\end{align*}
From Lemma \ref{l:estvariance} we know \[G_Q(x,x)\ge \frac{1}{C}\log\left(1+\min(\e^{-\alpha},\e^{-1/5}\right)\ge\frac{1}{C_\alpha}|\log\e|\]
and thus
\[\zeta^\e_\Lambda(\A\ni x\mid\A\cap(Q\setminus\{x\})\le C_\alpha\frac{\e}{|\log\e|^{1/2}}\,.\]
When we combine this with \eqref{e:upperestpinned_2} and \eqref{e:upperestpinned_3} we obtain \eqref{e:upperestpinned_1}. This completes the proof.
\end{proof}
In this proof the choice of $\e^{-1/5}$ for the halfdiameter of $Q$ might seem arbitrary. Indeed, one could also choose $\e^{-1/4}|\log\e|^{-1/8}$ and obtain the same result. This is still smaller than $\lmic$ which is the actual lengthscale that one expects here. However, because we have to use Remark \ref{r:strongdomdim4} instead of a comparison with $p_{4,+}$, we lose some logarithmic factor and hence cannot use the natural length scale for the size of $Q$. Fortunately, this does not affect the proof as the estimate \eqref{e:upperestpinned_3} shows that the second summand in \eqref{e:upperestpinned_2} is of lower order.

\begin{proof}[Proof of Theorem \ref{t:estimatespinnedset} c)] The following proof follows closely the proof in \cite[Section 3.3]{Bolthausen2001}, which itself is based on \cite{Deuschel2000,Ioffe2000}. However, that proof is a bit hard to follow as one has to refer to all three references. Furthermore, there is a small mistake in \cite{Ioffe2000} that needs to be fixed (cf. Remark \ref{r:gap_IV} below). Thus, we give a complete proof for the case at hand.

\emph{Step 1: Growing microscopic polymers}\\
For reasons that will become clear in the next step we need a procedure to grow microscopic polymers in a controlled way. Thus, we begin with the necessary definitions.

Let $K$ be an odd integer to be fixed later (in \eqref{e:estimateforK}). We consider the polymers in $\Pbox_{K\lmic}$. Let $E\in\Pbox_{K\lmic}$ be such a polymer. Suppose that it has $n$ connected components . We want to define for any multiindex $\underline{k}\in\N^n$ an enlarged polymer $E^{(\underline{k})}\in\Pbox_{K\lmic}$ in such a way that we add $k_i$ boxes to the $i$-th connected component.

To be precise, fix some enumeration of the boxes in $\Qbox_{K\lmic}$ by the natural numbers. Let the connected components of $E$ be $E_1,\ldots,E_n$, named in such a way that the minimal label of a box in $E_i$ increases with $i$.

For $i\in\{1,\ldots,n\}$, $j\in\{0,\ldots,k_i\}$ we define inductively a polymer $E^{(i,j)}\supset E$ as follows. If $j=0$, we let $E^{(i,j)}=E^{(i-1,k_{i-1})}$ (and $E^{(1,0)}=E$). If $j>0$, let $\tilde E_j$ be the connected component of $E^{(i,j-1)}$ that contains $E_j$, let $Q^{(i,j)}\in\Qbox_{K\lmic}$ be the box of smallest index that touches $\tilde E_j$, and let $E^{(i,j)}=E^{(i,j-1)}\cup Q^{(i,j)}$. Finally we let $E^{\underline{k}}=E^{(n,k_n)}$.

Let us note some properties of $E^{\underline{k}}$. First of all, it contains precisely $|\underline{k}|_1:=k_1+\ldots+k_n$ boxes of $\Qbox_{K\lmic}$ more than $E$. In other words,
\begin{equation}\label{e:lowerestpinned_5}
\left|E^{\underline{k}}\right|=|E|+|\underline{k}|_1K^4\lmic^4\,.
\end{equation}
Furthermore, $E^{\underline{k}}$ has at most $n$ connected components. Each $E_i$ is contained in one of the connected components of $E^{\underline{k}}$, and the latter has grown by at least $k_i$ boxes. Also each fixed box in $\Qbox_{K\lmic}$ is eventually contained in $E^{\underline{k}}$ whenever $|\underline{k}|_1$ is large enough.
Let us also note that each connected component of $E$ consists of at least one box. Therefore we have the estimate 
\begin{equation}\label{e:lowerestpinned_6}
n\le\frac{|E|}{K^4\lmic^4}\,.
\end{equation}

\emph{Step 2: Estimate for microscopic polymers}\\
We first prove \eqref{e:lowerestpinnedset_dim4} for the special case that $E$ is a polymer in $\Pbox_{K\lmic}$, where $K$ is a constant as in Step 1. That is, we claim that there is an odd integer $K$ such that for any $E\subset\Lambda$ such that $E\in\Pbox_{K\lmic}$, and for any $\e$ small enough we have
\begin{equation}\label{e:lowerestpinned_1}
\left(1-\frac{\e}{C|\log\e|^{1/2}}\right)^{|E|}\ge\zeta^\e_\Lambda(\A\cap E=\varnothing)\,.
\end{equation}

Suppose that $E$ has $n$ connected components, and consider for $\underline{k}\in\N^n$ the polymers $E^{(\underline{k})}$ constructed in the previous section. For $\underline{l}\in\N^n$ we write $\underline{l}>\underline{k}$ to denote $l_i\ge k_i$ for all $i$ and $l_i> k_i$ for at least one $i$. Recall that $\tilde\A=\A\cup(\Z^\dd\setminus\Lambda)$.

For $|\underline{k}|_1$ large enough we have $E^{(\underline{k})}\not\subset\Lambda$ and therefore $\tilde\A\cap E^{(\underline{k})}\neq\varnothing$ almost surely. Thus
\[\zeta^\e_\Lambda(\A\cap E=\varnothing)=\zeta^\e_\Lambda(\tilde\A\cap E=\varnothing)=\zeta^\e_\Lambda\left(\exists\underline{k}\in\N^n\colon\tilde\A\cap E^{(\underline{k})}=\varnothing,\tilde\A\cap E^{(\underline{l})}\neq\varnothing\ \forall\underline{l}>\underline{k}\right)\]
and so in particular
\begin{equation}\label{e:lowerestpinned_2}
\begin{aligned}
\zeta^\e_\Lambda(\A\cap E=\varnothing)&\le\sum_{\underline{k}\in\N^n}\zeta^\e_\Lambda\left(\tilde\A\cap E^{(\underline{k})}=\varnothing,\tilde\A\cap E^{(\underline{l})}\neq\varnothing\ \forall\underline{l}>\underline{k}\right)\\
&\le\sum_{\underline{k}\in\N^n}\zeta^\e_\Lambda\left(\tilde\A\cap E^{(\underline{k})}=\varnothing\middle|\tilde\A\cap E^{(\underline{l})}\neq\varnothing\ \forall\underline{l}>\underline{k}\right)\,.
\end{aligned}
\end{equation}
Note that this sum is actually a finite sum as for large enough $|\underline{k}|_1$ the conditional probability is equal to 0. Let us estimate the summands in \eqref{e:lowerestpinned_2} separately. We have
\begin{equation}\label{e:lowerestpinned_3}
\begin{aligned}
\zeta^\e_\Lambda\left(\tilde\A\cap E^{(\underline{k})}=\varnothing\middle|\tilde\A\cap E^{(\underline{l})}\neq\varnothing\ \forall\underline{l}>\underline{k}\right)&=\frac{\zeta^\e_\Lambda\left(\tilde\A\cap E^{(\underline{k})}=\varnothing,\tilde\A\cap E^{(\underline{l})}\neq\varnothing\ \forall\underline{l}>\underline{k}\right)}{\zeta^\e_\Lambda\left(\tilde\A\cap E^{(\underline{l})}\neq\varnothing\ \forall\underline{l}>\underline{k}\right)}\\
&=\frac{\sum\limits_{\substack{A\subset\Lambda\setminus E^{(\underline{k})}\\\tilde A\cap E^{(\underline{l})}\neq\varnothing\ \forall\underline{l}>\underline{k}}}\e^{|A|}\frac{Z_{\Lambda\setminus A}}{Z^\e_\Lambda}}{\sum\limits_{\substack{A'\subset\Lambda\\\tilde{A'}\cap E^{(\underline{l})}\neq\varnothing\ \forall\underline{l}>\underline{k}}}\e^{|A'|}\frac{Z_{\Lambda\setminus A'}}{Z^\e_\Lambda}}\\
&=\frac{\sum\limits_{\substack{A\subset\Lambda\setminus E^{(\underline{k})}\\\tilde A\cap E^{(\underline{l})}\neq\varnothing\ \forall\underline{l}>\underline{k}}}\e^{|A|}Z_{\Lambda\setminus A}}{\sum\limits_{B\subset E^{(\underline{k})}}\sum\limits_{\substack{A\subset\Lambda\setminus E^{(\underline{k})}\\\tilde{A}\cap E^{(\underline{l})}\neq\varnothing\ \forall\underline{l}>\underline{k}}}\e^{|A|+|B|}Z_{\Lambda\setminus (A\cup B)}}\\
&=\left(\sum_{B\subset E^{(\underline{k})}}\e^{|B|}\frac{\sum\limits_{\substack{A\subset\Lambda\setminus E^{(\underline{k})}\\\tilde{A}\cap E^{(\underline{l})}\neq\varnothing\ \forall\underline{l}>\underline{k}}}\e^{|A|}Z_{\Lambda\setminus (A\cup B)}}{\sum\limits_{\substack{A\subset\Lambda\setminus E^{(\underline{k})}\\\tilde A\cap E^{(\underline{l})}\neq\varnothing\ \forall\underline{l}>\underline{k}}}\e^{|A|}Z_{\Lambda\setminus A}}\right)^{-1}\\
&\le\left(\sum_{B\subset E^{(\underline{k})}}\e^{|B|}\min\limits_{\substack{A\subset\Lambda\setminus E^{(\underline{k})}\\\tilde{A}\cap E^{(\underline{l})}\neq\varnothing\ \forall\underline{l}>\underline{k}}}\frac{Z_{\Lambda\setminus (A\cup B)}}{Z_{\Lambda\setminus A}}\right)^{-1}
\end{aligned}
\end{equation}
where we have used $\frac{\sum_{i\in I} x_i}{\sum_{i\in I} y_i}\ge\min_{i\in I}\frac{x_i}{y_i}$ in the last step.

Next, we estimate this minimum from below, at least for sufficiently many sets $B$. Let $m=\frac{|E^{(\underline{k})}|}{K^4\lmic^4}$ be the number of boxes in $E^{(\underline{k})}$. We will consider the class of sets $B$ that contain exactly one point in each box of $E^{(\underline{k})}$.

Consider some $A\subset\Lambda\setminus E^{(\underline{k})}$ such that $\tilde{A}\cap E^{(\underline{l})}\neq\varnothing$ for all $\underline{l}>\underline{k}$. The properties of $A$ imply that each connected component of $E^{(\underline{k})}$ touches a box that contains a point of $\tilde A$, as otherwise we could still grow one of the components (by choosing a larger multiindex) without intersecting $\tilde A$. Therefore we can enumerate the boxes of $E^{(\underline{k})}$ as $D_1,\ldots,D_m$ in such a way that each $D_i$ touches a box that contains a point of $\tilde A$ or a box $D_j$ with $j<i$. As mentioned, we consider sets $B=\{b_1,\ldots,b_m\}$ that contain one point $b_i$ in each box $D_i$. Let $B_i=\{b_1,\ldots,b_i\}$ (and $B_0=\varnothing$). We have that
\[\frac{Z_{\Lambda\setminus (A\cup B)}}{Z_{\Lambda\setminus A}}=\prod_{i=1}^m\frac{Z_{\Lambda\setminus (A\cup B_i)}}{Z_{\Lambda\setminus (A\cup B_{i-1}})}\,.\]
Pick some $i\in\{1,\ldots,m\}$. Our construction of the $D_j$ ensures that $D_i$ touches a box containing a point of $\tilde A\cup B_{i-1}$. In particular, $b_i\in D_i$ has distance at most $\sqrt{2^2+1^2+1^2+1^2}K\lmic=\sqrt{7}K\lmic$ from a point in $\tilde A\cup B_{i-1}$. Now, \eqref{e:ratiopartsums2} and Lemma \ref{l:estvariance} imply that
\begin{align*}
\frac{Z_{\Lambda\setminus (A\cup B_i)}}{Z_{\Lambda\setminus (A\cup B_{i-1}})}&=\frac{1}{\sqrt{2\pi G_{\Lambda\setminus(A\cup B_{i-1}})(b_i,b_i)}}\\
&\ge \frac{1}{C\sqrt{\log\left(1+\sqrt{7}K\lmic\right)}}\\
&\ge\frac{1}{C|\log\e|^{1/2}}
\end{align*}
as soon as $\e$ is small enough (depending on $K$). Thus,
\[\frac{Z_{\Lambda\setminus (A\cup B)}}{Z_{\Lambda\setminus A}}\ge\left(\frac{1}{C|\log\e|^{1/2}}\right)^m\,.\]
This estimate holds for all $A\subset\Lambda\setminus E^{(\underline{k})}$ such that $\tilde{A}\cap E^{(\underline{l})}\neq\varnothing$ for all $\underline{l}>\underline{k}$, and all $B$ that contain exactly one point in each box of $E^{(\underline{k})}$. The number of such sets $B$ is $(K^4\lmic^4)^m$, and so \eqref{e:lowerestpinned_3} implies that
\begin{equation}\label{e:lowerestpinned_4}
\begin{aligned}
\zeta^\e_\Lambda\left(\tilde\A\cap E^{(\underline{k})}=\varnothing\middle|\tilde\A\cap E^{(\underline{l})}\neq\varnothing\ \forall\underline{l}>\underline{k}\right)&\le\left((K^4\lmic^4)^m\e^m\left(\frac{1}{C|\log\e|^{1/2}}\right)^m\right)^{-1}\\
&\le\left(\left(2K\frac{|\log\e|^{1/8}}{\e^{1/4}}\right)^4\frac{\e}{C|\log\e|^{1/2}}\right)^{-m}\\
&=\left(\frac{K^4}{\gamma}\right)^{-m}
\end{aligned}
\end{equation}
for a certain constant $\gamma$. We can now choose $K$ as an odd integer such that
\begin{equation}\label{e:estimateforK}
K\ge(e\gamma)^{1/4}\,.
\end{equation}
Then \eqref{e:lowerestpinned_4} in combination with \eqref{e:lowerestpinned_5} implies
\begin{align*}
\zeta^\e_\Lambda\left(\tilde\A\cap E^{(\underline{k})}=\varnothing\middle|\tilde\A\cap E^{(\underline{l})}\neq\varnothing\ \forall\underline{l}>\underline{k}\right)&\le\exp(-m)\\
&=\exp\left(-\frac{|E^{(\underline{k})}|}{K^4\lmic^4}\right)\\
&=\exp\left(-\frac{|E|}{K^4\lmic^4}-|\underline{k}|_1\right)\,.
\end{align*}
Now we can use this result in \eqref{e:lowerestpinned_2} and obtain
\begin{align*}
\zeta^\e_\Lambda(\A\cap E=\varnothing)&\le\sum_{\underline{k}\in\N^n}\exp\left(-\frac{|E|}{K^4\lmic^4}-|\underline{k}|_1\right)\\
&=\exp\left(-\frac{|E|}{K^4\lmic^4}\right)\left(\sum_{k_1=0}^\infty\exp(-k_1)\right)\cdot\ldots\cdot\left(\sum_{k_n=0}^\infty\exp(-k_1)\right)\\
&=\exp\left(-\frac{|E|}{K^4\lmic^4}\right)\left(\frac{e}{e-1}\right)^n\\
&=\exp\left(-\frac{|E|}{K^4\lmic^4}+n(1-\log(e-1))\right)\,.
\end{align*}
Finally, we can recall \eqref{e:lowerestpinned_6} and conclude
\begin{align*}
\zeta^\e_\Lambda(\A\cap E=\varnothing)&\le\exp\left(-\frac{|E|}{K^4\lmic^4}+\frac{|E|}{K^4\lmic^4}(1-\log(e-1))\right)\\
&=\exp\left(-\log(e-1)\frac{|E|}{K^4\lmic^4}\right)\\
&\le\exp\left(-\frac{\e}{C|\log\e|^{1/2}}|E|\right)\\
&\le\left(1-\frac{\e}{C|\log\e|^{1/2}}\right)^{|E|}
\end{align*}
whenever $\e$ is small enough, and the $K$ (that is now fixed) has been absorbed into the constant. This completes the proof of \eqref{e:lowerestpinned_1}.

\emph{Step 3: Density of pinned points on macroscopic scales}\\
We now show that on the length scale $\lmac$ most points of a set $E\subset\Lambda$ are close to a point in $\tilde\A$. To make this precise, we need to make a few definitions. Let $L$ be an odd integer to be fixed later (in \eqref{e:estimateforL_1} and \eqref{e:estimateforL_2}). We consider polymers in $\Pbox_{KL\lmac}$. Observe that $KL\lmac$ is an odd multiple of $K\lmic$, the lengthscale from Step 2. For $E\subset\Lambda$ let
\[S_E=\{Q\in\Qbox_{KL\lmac}\colon Q\cap E\neq\varnothing\}\]
and
\[S_{E,\mathrm{bad}}(\A)=\{Q\in S_E\colon Q\cap\tilde A=\varnothing\}\,.\]
We think of the boxes in $S_{E,\mathrm{bad}}(\A)$ as bad boxes, as they contain points of $E$ but no pinned point. We will show that not too many boxes are bad. Note that $|S_E|\ge\frac{|E|}{(KL)^4\lmac^4}$. Our claim now is that with $K$ as before there is an odd integer $L$ such that for any $E\subset\Lambda$ and any $\e$ small enough we have
\begin{equation}\label{e:lowerestpinned_7}
\zeta^\e_\Lambda\left(|S_{E,\mathrm{bad}}(\A)|>\frac{|E|}{2(KL)^4\lmac^4}\right)\le\left(1-\frac{\e}{C|\log\e|^{1/2}}\right)^{|E|}\,.
\end{equation}
To see this, we use the result from the previous step to estimate
\begin{align*}
&\zeta^\e_\Lambda\left(|S_{E,\mathrm{bad}}(\A)|>\frac{|E|}{2(KL)^4\lmac^4}\right)\\
&\quad=\sum_{\substack{T\subset S_E\\|T|\ge|E|/(2(KL)^4\lmac^4)}}\zeta^\e_\Lambda(S_{E,\mathrm{bad}}(\A)=T)\\
&\quad\le\sum_{\substack{T\subset S_E\\|T|\ge|E|/(2(KL)^4\lmac^4)}}\zeta^\e_\Lambda(S_{E,\mathrm{bad}}(\A)\supset T)\\
&\quad=\sum_{\substack{T\subset S_E\\|T|\ge|E|/(2(KL)^4\lmac^4)}}\zeta^\e_\Lambda\left(\A\cap \bigcup_{Q\in T}Q=\varnothing\right)\\
&\quad\le\sum_{\substack{T\subset S_E\\|T|\ge|E|/(2(KL)^4\lmac^4)}}\left(1-\frac{\e}{C|\log\e|^{1/2}}\right)^{|T|(KL)^4\lmac^4}\\
&\quad=\sum_{j=\lceil|E|/(2(KL)^4\lmac^4)\rceil}^{|S_E|}\binom{|S_E|}{j}\left(1-\frac{\e}{C|\log\e|^{1/2}}\right)^{j(KL)^4\lmac^4}\\
&\quad\le\sum_{j=\lceil|E|/(2(KL)^4\lmac^4)\rceil}^{|S_E|}\binom{|S_E|}{j}\exp\left(-\frac{\e}{C|\log\e|^{1/2}}(KL)^4\lmac^4j\right)\\
&\quad\le\sum_{j=\lceil|E|/(2(KL)^4\lmac^4)\rceil}^{|S_E|}\binom{|S_E|}{j}\exp\left(-\frac{\e}{C|\log\e|^{1/2}}(KL)^4\frac{|\log\e|^{3/2}}{\e} j\right)\\
&\quad=\sum_{j=\lceil|E|/(2(KL)^4\lmac^4)\rceil}^{|S_E|}\binom{|S_E|}{j}\e^{(KL)^4j/\gamma' }
\end{align*}
for a certain constant $\gamma'$. We now want to apply the estimate for binomial sums that is stated in Lemma \ref{l:tailbound} below with $N=|S_E|$, $p=\e^{(KL)^4j/\gamma'}$ and $r=\frac{|E|}{2(KL)^4\lmac^4|S_E|}$. To do so, we need $p\le r\le\frac12$. Because $1\le\frac{|E|}{|S_E|}\le(KL)^4\lmac^4$ we always have $r\le\frac12$, and for $p\le r$ it suffices that $\e^{(KL)^4j/\gamma'}\le\frac{\e}{2(KL)^4|\log\e|^{3/2}}$. To ensure the latter we choose $L$ such that
\begin{equation}\label{e:estimateforL_1}
L>\frac{\gamma'^{1/4}}{K}
\end{equation}
and $\e$ is small enough.
Using Lemma \ref{l:tailbound} we then obtain
\begin{equation}\label{e:lowerestpinned_8}
\zeta^\e_\Lambda\left(|S_{E,\mathrm{bad}}(\A)|>\frac{|E|}{2(KL)^4\lmac^4}\right)\le\left(\frac{p}{r^2}\right)^{r|S_E|}\,.
\end{equation}
We can estimate that
\begin{align*}
\frac{p}{r^2}&=\exp\left(-\frac{(KL)^4|\log\e|}{\gamma'}-2\log\frac{|E|}{2(KL)^4\lmac^4|S_E|}\right)\\
&\le\exp\left(-\frac{(KL)^4|\log\e|}{\gamma'}+2\log\frac{1}{2(KL)^4\lmac^4}\right)\\
&\le\exp\left(-\frac{(KL)^4|\log\e|}{\gamma'}+2|\log\e|+2\log(2(KL)^4)+\log(|\log\e|^{3/2})\right)\,.
\end{align*}
Provided that we choose 
\begin{equation}\label{e:estimateforL_2}
L>\frac{(2\gamma')^{1/4}}{K}
\end{equation}
we can estimate this as
\[\frac{p}{r^2}\le\exp\left(-\frac{|\log\e|}{C}\right)\]
whenever $\e$ is small enough (depending on $K$, $L$ that are now fixed).
Returning to \eqref{e:lowerestpinned_8}, we see that
\begin{align*}
\zeta^\e_\Lambda\left(|S_{E,\mathrm{bad}}(\A)|>\frac{|E|}{2(KL)^4\lmac^4}\right)&\le\exp\left(-\frac{|\log\e|}{C}\frac{|E|}{2(KL)^4\lmac^4|S_E|}|S_E|\right)\\
&\le\exp\left(-\frac{\e}{C|\log\e|^{1/2}}|E|\right)
\end{align*}
which implies \eqref{e:lowerestpinned_7}.

\emph{Step 4: Estimate for arbitrary sets}\\
We now can prove the actual result \eqref{e:lowerestpinnedset_dim4}. So let $E\subset\Lambda$. Using the notation from the previous step, we let \[E_{\mathrm{bad}}(\A)=E\cap\bigcup_{Q\in S_{E,\mathrm{bad}}(\A)}Q\]
be the set of bad points (those which are far from a pinned point). We have the estimate $|E_{\mathrm{bad}}(\A)|\le(KL)^4\lmac^4|S_{E,\mathrm{bad}}(\A)|$ and so the previous step implies that
\[\zeta^\e_\Lambda\left(|E_{\mathrm{bad}}(\A)|>\frac{|E|}{2}\right)\le\zeta^\e_\Lambda\left(|S_{E,\mathrm{bad}}(\A)|>\frac{|E|}{2(KL)^4\lmac^4}\right)\le\left(1-\frac{\e}{C|\log\e|^{1/2}}\right)^{|E|}\,.
\]
We can now write
\begin{align*}
\zeta^\e_\Lambda(\A\cap E=\varnothing)&\le\zeta^\e_\Lambda\left(\A\cap E=\varnothing,|E_{\mathrm{bad}}(\A)|\le\frac{|E|}{2}\right)+\zeta^\e_\Lambda\left(|E_{\mathrm{bad}}(\A)|>\frac{|E|}{2}\right)\\
&\le\zeta^\e_\Lambda\left(\A\cap E=\varnothing,|E_{\mathrm{bad}}(\A)|\le\frac{|E|}{2}\right)+\left(1-\frac{\e}{C|\log\e|^{1/2}}\right)^{|E|}
\end{align*}
and hence we only need to estimate the first term to establish \eqref{e:lowerestpinnedset_dim4}. If $\zeta^\e_\Lambda\left(|E_{\mathrm{bad}}(\A)|\le\frac{|E|}{2}\right)=0$, that term is equal to 0 and we are trivially done. So we can assume otherwise, and estimate
\[\zeta^\e_\Lambda\left(\A\cap E=\varnothing,|E_{\mathrm{bad}}(\A)|\le\frac{|E|}{2}\right)\le\zeta^\e_\Lambda\left(\A\cap E=\varnothing\middle||E_{\mathrm{bad}}(\A)|\le\frac{|E|}{2}\right)\,.\]
Next, we can apply a similar argument as in \eqref{e:lowerestpinned_3} to see that
\begin{equation}\label{e:lowerestpinned_9}
\begin{aligned}
\zeta^\e_\Lambda\left(\A\cap E=\varnothing\middle||E_{\mathrm{bad}}(\A)|\le\frac{|E|}{2}\right)&=\frac{\sum\limits_{\substack{A\subset\Lambda\setminus E\\|E_{\mathrm{bad}}(A)|\le|E|/2}}\e^{|A|}Z_{\Lambda\setminus A}}{\sum\limits_{B\subset E}\sum\limits_{\substack{A\subset\Lambda\setminus E\\|E_{\mathrm{bad}}(A)|\le|E|/2}}\e^{|A|+|B|}Z_{\Lambda\setminus (A\cup B)}}\\
&=\left(\frac{\sum\limits_{\substack{A\subset\Lambda\setminus E\\|E_{\mathrm{bad}}(A)|\le|E|/2}}\sum\limits_{B\subset E}\e^{|A|+|B|}Z_{\Lambda\setminus (A\cup B)}}{\sum\limits_{\substack{A\subset\Lambda\setminus E\\|E_{\mathrm{bad}}(A)|\le|E|/2}}\e^{|A|}Z_{\Lambda\setminus A}}\right)^{-1}\\
&\le\left(\min\limits_{\substack{A\subset\Lambda\setminus E\\|E_{\mathrm{bad}}(A)|\le|E|/2}}\sum_{B\subset E}\e^{|B|}\frac{Z_{\Lambda\setminus (A\cup B)}}{Z_{\Lambda\setminus A}}\right)^{-1}\,.
\end{aligned}
\end{equation}
Note that unlike in \eqref{e:lowerestpinned_3} we interchanged the summations over $A$ and $B$ in an intermediate step, which allows us to have $\min_A\sum_B$ instead of $\sum_B\min_A$ in the result of this calculation.

We can estimate this further by only allowing good points for $B$, that is by estimating
\begin{equation}\label{e:lowerestpinned_10}
\zeta^\e_\Lambda\left(\A\cap E=\varnothing\middle||E_{\mathrm{bad}}(\A)|\le\frac{|E|}{2}\right)\le\left(\min\limits_{\substack{A\subset\Lambda\setminus E\\|E_{\mathrm{bad}}(A)|\le|E|/2}}\sum_{B\subset E\setminus E_{\mathrm{bad}}(A)}\e^{|B|}\frac{Z_{\Lambda\setminus (A\cup B)}}{Z_{\Lambda\setminus A}}\right)^{-1}\,.
\end{equation}
Consider some $A\subset\Lambda\setminus E$, and some $B\subset E\setminus E_{\mathrm{bad}}(A)$. By definition of $E_{\mathrm{bad}}(A)$, each point in $B$ is in the same macroscopic box as a point of $\tilde A$. In particular, each point in $B$ has distance at most $\sqrt{7}KL\lmac$ to a point of $\tilde A$. Thus, if we let $B=\{b_1,\ldots,b_{|B|}\}$, and $B_i=\{b_1,\ldots,b_i\}$ we see as in Step 2 that
\begin{align*}
\frac{Z_{\Lambda\setminus (A\cup B)}}{Z_{\Lambda\setminus A}}&=\prod_{i=1}^{|B|}\frac{Z_{\Lambda\setminus (A\cup B_i)}}{Z_{\Lambda\setminus (A\cup B_{i-1}})}\\
&=\prod_{i=1}^{|B|}\frac{1}{\sqrt{2\pi G_{\Lambda\setminus(A\cup B_{i-1}})}(b_i,b_i)}\\
&\ge\prod_{i=1}^{|B|}\frac{1}{C\sqrt{\log\left(1+\sqrt{7}KL\lmac\right)}}\\
&\ge\left(\frac{1}{C|\log\e|^{1/2}}\right)^{|B|}
\end{align*}
where we used \eqref{e:ratiopartsums2} and Lemma \ref{l:estvariance}. Returning to \eqref{e:lowerestpinned_9} and \eqref{e:lowerestpinned_10}, we obtain
\begin{align*}
&\zeta^\e_\Lambda\left(\A\cap E=\varnothing\middle||E_{\mathrm{bad}}(\A)|\le\frac{|E|}{2}\right)\\
&\quad\le\left(\min\limits_{\substack{A\subset\Lambda\setminus E\\|E_{\mathrm{bad}}(A)|\le|E|/2}}\sum_{B\subset E\setminus E_{\mathrm{bad}}(A)}\e^{|B|}\left(\frac{1}{C|\log\e|^{1/2}}\right)^{|B|}\right)^{-1}\\
&\quad=\left(\min\limits_{\substack{A\subset\Lambda\setminus E\\|E_{\mathrm{bad}}(A)|\le|E|/2}}\sum_{j=0}^{|E\setminus E_{\mathrm{bad}}(A)|}\binom{|E\setminus E_{\mathrm{bad}}(A)|}{j}\left(\frac{\e}{C|\log\e|^{1/2}}\right)^j\right)^{-1}\\
&\quad=\left(\min\limits_{\substack{A\subset\Lambda\setminus E\\|E_{\mathrm{bad}}(A)|\le|E|/2}}\left(1+\frac{\e}{C|\log\e|^{1/2}}\right)^{|E\setminus E_{\mathrm{bad}}(A)|}\right)^{-1}\\
&\quad\le\left(\left(1+\frac{\e}{C|\log\e|^{1/2}}\right)^{|E|/2}\right)^{-1}\\
&\quad\le\left(1-\frac{\e}{C|\log\e|^{1/2}}\right)^{|E|}\,.
\end{align*}
This finally completes the proof.
\end{proof}

\begin{remark}\label{r:gap_IV}
In \cite{Bolthausen2001} a similar argument is used. However, for growing the polymers \cite{Bolthausen2001} refers to \cite{Ioffe2000}, where a construction that is different from ours is used. Unfortunately, the argument from \cite{Ioffe2000} contains a small gap.

The problem is as follows:
Take $\dd\ge2$. In \cite{Ioffe2000} the grown polymer $\tilde E^{\underline{k}}$ is only defined for certain admissible $\underline{k}$. Using our notation, one defines $\tilde E^{\underline{k}}$ by adding $k_i$ layers of cubes in $\Qbox_{K\lmic}$ to $E_i$, i.e. one replaces $E$ by
\[\tilde E^{\underline{k}}:=\bigcup_{i=1}^nE_i+Q_{k_iK\lmic}(0)\,.\]
However, this is only done if for each $i\in\{1,\ldots,n\}$ we have that $E_i+Q_{k_iK\lmic}(0)$ and $\bigcup_{j=1}^{i-1}E_i+Q_{k_iK\lmic}$ are disjoint or $k_i=0$ (and the $\underline{k}$ with this property are called admissible). Now in \cite[p. 398]{Ioffe2000} it is claimed that this construction satisfies 
\begin{equation}\label{e:analogue_of_lowerestpinned_5}
|\tilde E^{\underline{k}}|\ge|E|+|\underline{k}|_1K^\dd\lmic^\dd\,,
\end{equation}
or in other words that we have added at least $|\underline{k}|_1$ boxes. This is not true in general: For example if $L$ is a large odd number and
\begin{align*}
E_1&=\left[-\frac{K\lmic}{2},\frac{K\lmic}{2}\right]\cap\Z^\dd\\
E_2&=\left(\left[-\frac{KL\lmic}{2},\frac{KL\lmic}{2}\right]\setminus\left[-\frac{3K\lmic}{2},\frac{3K\lmic}{2}\right]\right)\cap\Z^\dd
\end{align*}
and $E=E_1\cup E_2$, then for any $k_1\in\left\{1,\frac{L-1}{2}\right\}$ the multiindex $\underline{k}=(k_1,0)$ is admissible, but to obtain $\tilde E^{\underline{k}}$ we only add the $3^\dd-1$ cubes that form the gap between $E_1$ and $E_2$. If $L$ is large enough, we can take $k_1\ge3^\dd$, and we arrive at a contradiction to \eqref{e:analogue_of_lowerestpinned_5}.

Note that this problem is not present in the construction that we used in Step 1 of the proof of Theorem \ref{t:estimatespinnedset} c), as our construction directly ensures that \eqref{e:lowerestpinned_5} holds. The same construction could also be used in \cite{Ioffe2000} to fix the gap there.

Alternatively (as pointed out to the author by Yvan Velenik) one can also fix the gap in \cite{Ioffe2000} by first ordering the $E_k$ in such a way that no $E_i$ completely surrounds an $E_j$ with $i<j$.
\end{remark}

In our proof of Theorem \ref{t:estimatespinnedset} c) we used a tail bound for certain binomial sums. We will use this estimate a few more times in Section \ref{s:est_prob_cutoff}, so we state and prove it separately.
\begin{lemma}\label{l:tailbound}
Let $N\in\N$, and $\frac12\ge r\ge p\ge0$. Then
\begin{equation}\label{e:tailbound}
\sum_{j=\lceil rN\rceil}^N\binom{N}{j}p^j\le\left(\frac{p}{r^2}\right)^{rN}\,.
\end{equation}
\end{lemma}
This estimate is very similar to standard Chernoff tail bounds for the binomial distribution. A special case was used in \cite[Section 3.3.2]{Bolthausen2001}. For the proof we will follow the proof of the Chernoff tail bound.

\begin{proof}
For any $t\ge0$ we have the estimate
\begin{align*}
\sum_{j=\lceil rN\rceil}^N\binom{N}{j}p^j&\le e^{-trN}\sum_{j=0}^N\binom{N}{j}e^{tj}p^j\\
&\le e^{-trN}(1+e^tp)^N\,.
\end{align*}
The optimal choice for $t$ is $t=\log\left(\frac{r}{(1-r)p}\right)$, and this yields
\[\sum_{j=\lceil rN\rceil}^N\binom{N}{j}p^j\le\left(\frac{(1-r)^{r-1}}{r^r}\right)^Np^{rN}\,.\]
It remains to observe that for $0< r\le\frac12$ one has
\[\frac{(1-r)^{r-1}}{r^r}\le\frac{1}{r^{2r}}\,.\]
\end{proof}

\section{Some inequalities}
In this section we provide some tools that will be used in the next two sections to establish Theorem \ref{t:estimatesfield}, namely a discrete multipolar Hardy-Rellich inequality as well as an interpolation inequality. We begin with the former.

\subsection{A discrete multipolar Hardy-Rellich inequality}

We want to give a quantitative estimate on the strength of the pinning effect on $x\in\Lambda$. More precisely, consider a function $u\colon\Lambda\to\R$ such that $u=0$ on $\tilde A=A\cup(\Z^\dd\setminus\Lambda)$. We want to control a weighted $L^2$-norm of $u$ by the $L^2$-norm of $\nabla_1^2u$. The weight at $x\in\Lambda$ will have to depend on the location of $x$ with respect to $\tilde A$. If $\nabla_1^2u$ is small, then $u$ is (locally) close to an affine function. We need to ensure that this affine function is close to zero near $x$, and for this purpose we need that $u$ is close to 0 at $\dd+1$ points that are well-spread out, i.e. we need that $x$ is close to $\dd+1$ pinned points.

To state our precise result we need some definitions. First we construct $\dd+1$ cones of directions that are well-spread out:
Let $\theta_1,\ldots,\theta_{\dd+1}\in\Sphere^{\dd-1}$ be such that $\theta_i\cdot\theta_j=-\frac{1}{\dd}$ for $i\neq j$ (e.g. take $\left(\theta_i\right)_{i=1}^{\dd+1}$ to be the vertices of a regular $\dd$-dimensional simplex with circumsphere $\Sphere^{\dd-1}$). 

For $\kappa>0$ let $\Theta_i=B_\kappa\left(\theta_i\right)\cap \Sphere^{\dd-1}$. For $\kappa$ small enough we have $\theta_i'\cdot\theta_j'<0$ for all $\theta_i'\in\Theta_i$, $\theta_j'\in\Theta_j$ for $i\neq j$. Fix one such choice of $\kappa$. Finally let $\Xi_i=\left\{y\in\R^\dd\setminus\{0\}\colon\frac{y}{|y|}\in\Theta_i\right\}$ (cf. Figure \ref{f:Xi_i}).

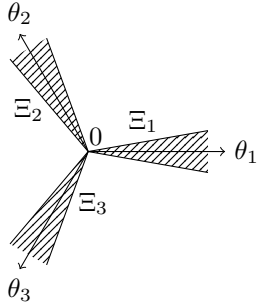
\begin{figure}[h]
\centering
\begin{tikzpicture}
\draw[-] (0,0)--(350:1.6) node (A) {};
\draw[-] (0,0)--(10:1.6) node (B) {};
\draw[-] (0,0)--(110:1.6) node (C) {};
\draw[-] (0,0)--(130:1.6) node (D) {};
\draw[-] (0,0)--(230:1.6) node (E) {};
\draw[-] (0,0)--(250:1.6) node (F) {};

\draw[->] (0,0)--(0:1.8) node[right]{$\theta_1$};
\draw[->] (0,0)--(120:1.8) node[above]{$\theta_2$};
\draw[->] (0,0)--(240:1.8) node[below]{$\theta_3$};

\draw[draw=none] (0,0)--(B) node[midway,above] {$\Xi_1$};
\draw[draw=none] (0,0)--(D) node[midway,left] {$\Xi_2$};
\draw[draw=none] (0,0)--(F) node[midway,right] {$\Xi_3$};

\node at (0.1,0.2) {$0$};
\draw[pattern=north east lines,draw=none] (0,0)--(A.center)--(B.center)--(0,0)--(C.center)--(D.center)--(0,0)--(E.center)--(F.center)--cycle;

\end{tikzpicture}
\caption{The sets $\Xi_i$ for $\dd=2$}\label{f:Xi_i}
\end{figure}

For $x\in\Lambda$ let 
\begin{align*}
	d^{(i)}(x,\tilde A)&=\inf_{a\in \tilde A\cap (x+\Xi_i)}|x-a|_1\,,\\
	d_*(x,\tilde A)&=\max_{i\in\{1,\ldots,\dd+1\}}d^{(i)}(x,\tilde A)
\end{align*}
with the convention that $\inf\varnothing=+\infty$. Thus, for each $x$ there are $\dd+1$ points in $\tilde A$ which are well-spread out around $A$ with distance at most $d_*(x,\tilde A)$.

Then we have the following statement.
\begin{theorem}\label{t:localpoincare}
Let $A\subset \Lambda$ be arbitrary. Let $V\subset\Lambda$ be an arbitrary subset. Let $R\in\N$, $R\ge2$ be a parameter. Suppose that $u\colon\Lambda\to\R$ is such that $u=0$ on $\tilde A$. Then
\begin{equation}\label{e:localpoincare}
\left\|u\I_{d_*(\cdot,\tilde A)\le R}\right\|^2_{L^2(V)}\le C_\dd R^\dd(1+\I_{\dd=4}\log R)\left\|\nabla_1^2u\right\|^2_{L^2(V+Q_R(0))}
\end{equation}
where the constant is independent of $R$, $V$ and $A$.
\end{theorem}

\begin{proof}
We begin with the case $\dd\ge5$. We fix an enumeration of the points in $\Z^\dd$.

We first establish a pointwise bound for $u\I_{d_*(\cdot,\tilde A)\le R}$. Let $x\in V$ such that $d_*(x,\tilde A)\le R$. For $i\in\{1,\ldots,\dd+1\}$ consider the points in $\tilde A\cap(x+\Xi_i)$ of minimal $l^1$-distance to $x$, and let $a^{(i)}_x$ be the one among those that comes first with respect to our fixed enumeration. By assumption $|x-a^{(i)}_x|_1=d^{(i)}(x,\tilde A)\le R$.

We first claim
\begin{equation}\label{eq:localpoincare1}
\begin{aligned}
	|u(x)|&\le\max_{i\in\{1,\ldots,\dd+1\}}\left|u(a^{(i)}_x)-u(x)-\nabla_1u(x)\cdot\left(a^{(i)}_x-x\right)\right|\,.
\end{aligned}
\end{equation}
Indeed, we can assume $u(x)\ge0$ (the other case is analogous). There is an index $i$ such that $\nabla_1u(x)\cdot\left(a^{(i)}_x-x\right)\ge0$, as otherwise the $\dd+2$ vectors \[\nabla_1u(x),a^{(1)}_x-x,\ldots,a^{(\dd+1)}_x-x\] would have pairwise negative scalar products, while it is easy to see that this is possible in $\R^\dd$ for at most $\dd+1$ vectors. In particular, we have
\[\max_{i\in\{1,\ldots,\dd+1\}}\nabla_1u(x)\cdot\left(a^{(i)}_x-x\right)\ge0\,.\]
By assumption $u(a^{(i)}_x)=0$ and so
\begin{align*}
	u(x)&\le u(x)+\max_{i\in\{1,\ldots,\dd+1\}}\nabla_1u(x)\cdot\left(a^{(i)}_x-x\right)\\
	&=\max_{i\in\{1,\ldots,\dd+1\}}u(x)-u(a^{(i)}_x)+\nabla_1u(x)\cdot\left(a^{(i)}_x-x\right)\\
	&\le \max_{i\in\{1,\ldots,\dd+1\}}\left|u(x)-u(a^{(i)}_x)+\nabla_1u(x)\cdot\left(a^{(i)}_x-x\right)\right|
\end{align*}
which is \eqref{eq:localpoincare1}.

We now want to pick a nearest neighbour path $\Psi^{(i)}_x=\left(\Psi^{(i)}_x(0),\ldots,\Psi^{(i)}_x(d^{(i)}(x,\tilde A))\right)$ such that $\Psi^{(i)}_x(0)=x$, $\Psi^{(i)}_x(d^{(i)}(x,\tilde A))=a^{(i)}_x$. We can pick this path in such a way that all of its points have distance at most $\sqrt{\dd}$ from the straight line connecting $x$ and $a^{(i)}_x$, and such that all but possible the first $\tilde\alpha_\dd$ of its vertices lie inside the widening cone $x+\Xi_i$ (cf. Figure \ref{f:Psi_i}). Here $\tilde\alpha_\dd$ is a constant depending only on $\dd$ and the $\Xi_i$. 

\begin{figure}[ht]
\centering
\begin{tikzpicture}
\draw[-] (0,0)--(10:6) node (A) {};
\draw[-] (0,0)--(30:6) node (B) {};
\draw[step=.3,black,very thin,dotted] (-.3,-.3) grid (6,3.3);

\draw[fill] (0,0) circle (1pt);
\draw[fill] (0.3,0) circle (1pt);
\draw[fill] (0.3,0.3) circle (1pt);
\draw[fill] (0.6,0.3) circle (1pt);
\draw[fill] (0.9,0.3) circle (1pt);
\draw[fill] (1.2,0.3) circle (1pt);
\draw[fill] (1.2,0.6) circle (1pt);
\draw[fill] (1.5,0.6) circle (1pt);
\draw[fill] (1.8,0.6) circle (1pt);
\draw[fill] (1.8,0.9) circle (1pt);
\draw[fill] (2.1,0.9) circle (1pt);
\draw[fill] (2.4,0.9) circle (1pt);
\draw[fill] (2.4,1.2) circle (1pt);
\draw[fill] (2.7,1.2) circle (1pt);
\draw[fill] (2.7,1.5) circle (1pt);
\draw[fill] (3.0,1.5) circle (1pt);
\draw[fill] (3.3,1.5) circle (1pt);
\draw[fill] (3.3,1.8) circle (1pt);
\draw[fill] (3.6,1.8) circle (1pt);
\draw[fill] (3.9,1.8) circle (1pt);
\draw[fill] (3.9,2.1) circle (1pt);
\draw[fill] (4.2,2.1) circle (1pt);
\draw[fill] (4.5,2.1) circle (1pt);
\draw[fill] (4.5,2.4) circle (1pt);
\draw[fill] ([xshift=-2pt,yshift=-2pt]4.8,2.4) rectangle ++(4pt,4pt) node[right] {$a^{(i)}_x$};

\draw[-,thick] (0,0)--(0.3,0)--(0.3,0.3)--(0.6,0.3)--(0.9,0.3)--(1.2,0.3)--(1.2,0.6)--(1.5,0.6)--(1.8,0.6)--(1.8,0.9)--(2.1,0.9)--(2.4,0.9)--(2.4,1.2)--(2.7,1.2)--(2.7,1.5)--(3.0,1.5)--(3.3,1.5)--(3.3,1.8)--(3.6,1.8)--(3.9,1.8)--(3.9,2.1)--(4.2,2.1)--(4.5,2.1)--(4.5,2.4)--(4.8,2.4);

\draw[dashed] (0.1897,-0.3795)--(4.9897,2.0205);
\draw[dashed] (-0.1897,0.3795)--(4.6103,2.7795);
\node at (-.1,0.1) {$x$};
\node at (5.5,1.5) {$x+\Xi_i$};

\end{tikzpicture}
\caption{Choice of the path $\Psi^{(i)}_x$. We require all points to have distance at most $\sqrt{\dd}$ from the straight line between $x$ and $a^{(i)}_x$ (i.e. to be in the dashed strip), and all but the first $\tilde\alpha_\dd$ to be inside the cone $x+\Xi_i$.}\label{f:Psi_i}
\end{figure}
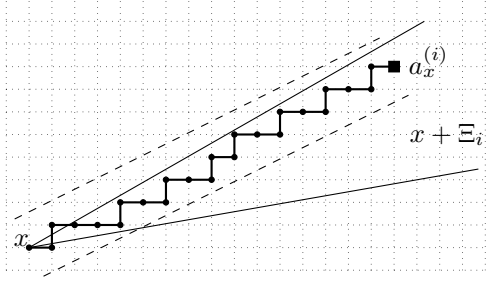

We can now apply a discrete version of the fundamental theorem of calculus along the paths $\Psi^{(i)}_x$ to the function $v:=u(\cdot)-u(x)-\nabla_1u(x)\cdot\left(\cdot-x\right)$. Namely, we know $v(x)=0$ and $\nabla_1v(x)=0$. The point $\Psi^{(i)}_x(1)$ is one of the $2\dd$ neighbours of $x$. If it happens that $\Psi^{(i)}_x(1)\in\{x+e_1,\ldots,x+e_\dd\}$, then $v(\Psi^{(i)}_x(1))=0$ and we can write
\[v(a^{(i)}_x)=\sum_{t=1}^{d^{(i)}(x,A)-1}\sum_{s=1}^t v(\Psi^{(i)}_x(s+1))-2v(\Psi^{(i)}_x(s))+v(\Psi^{(i)}_x(s-1))\,.\]
On the other hand, if $\Psi^{(i)}_x(1)\in\{x-e_1,\ldots,x-e_\dd\}$, we can temporarily add a point $\Psi^{(i)}_x(-1)=2x-\Psi^{(i)}_x(1)$ to our path, so that $v(\Psi^{(i)}_x(-1))=0$, and then write
\[v(a^{(i)}_x)=\sum_{t=0}^{d^{(i)}(x,\tilde A)-1}\sum_{s=0}^{t}v(\Psi^{(i)}_x(s+1))-2v(\Psi^{(i)}_x(s))+v(\Psi^{(i)}_x(s-1))\,.\]
In both cases we can conclude that
\begin{align*}
	\left|u(a^{(i)}_x)-u(x)-\nabla_1u(x)\cdot\left(a^{(i)}_x-x\right)\right|&=\left|v(a^{(i)}_x)\right|\\
	&\le\sum_{t=0}^{d^{(i)}(x,\tilde A)-1}\sum_{s=0}^t \left|\nabla_1^2v(\Psi^{(i)}_x(s))\right|\\
	&=\sum_{t=0}^{d^{(i)}(x,\tilde A)-1}\sum_{s=0}^t\left|\nabla_1^2u(\Psi^{(i)}_x(s))\right|\\
	&=\sum_{s=0}^{d^{(i)}(x,\tilde A)-1}(d^{(i)}(x,\tilde A)-s)\left|\nabla_1^2(\Psi^{(i)}_x(s))\right|
\end{align*}
where we have changed the order of summation in the last step. Thus, \eqref{eq:localpoincare1} implies that
\begin{equation}\label{eq:localpoincare4}
\begin{aligned}
	|u(x)|&\le\max_{i\in\{1,\ldots,\dd+1\}}\left|u(a^{(i)}_x)-u(x)-\nabla_1u(x)\cdot\left(a^{(i)}_x-x\right)\right|\\
	&=\max_{i\in\{1,\ldots,\dd+1\}}\sum_{s=0}^{d^{(i)}(x,\tilde A)-1}(d^{(i)}(x,\tilde A)-s)\left|\nabla_1^2u(\Psi^{(i)}_x(s))\right|\,.
\end{aligned}
\end{equation}

We have this estimate for all $x$ such that $d_*(x,\tilde A)\le R$. Defining $\Psi$ arbitrarily for the other $x$ and summing the square of \eqref{eq:localpoincare4} over $x$, we find
\begin{equation}\label{eq:localpoincare2}
\begin{aligned}
	\sum_{x\in V}|u(x)|^2\I_{d_*(x,\tilde A)\le R}&\le\sum_{x\in V}\I_{d_*(x,\tilde A)\le R}\left(\max_{i\in\{1,\ldots,\dd+1\}}\sum_{s=0}^{d^{(i)}(x,\tilde A)-1}(d^{(i)}(x,\tilde A)-s)\left|\nabla_1^2u(\Psi^{(i)}_x(s))\right|\right)^2\\
	&\le\sum_{i=1}^{\dd+1}\sum_{x\in V}\I_{d_*(x,\tilde A)\le R}\left(\sum_{s=0}^{d^{(i)}(x,\tilde A)-1}(d^{(i)}(x,\tilde A)-s)\left|\nabla_1^2u(\Psi^{(i)}_x(s))\right|\right)^2\,.
\end{aligned}
\end{equation}
Consider a nonzero summand of the outer two sums. Then $d_*(x,\tilde A)\le R$, and (recalling that $\dd\ge5$) we can apply Hölder's inequality to the innermost sum to obtain
\begin{equation}\label{eq:localpoincare3}
\begin{aligned}
&\left(\sum_{s=0}^{d^{(i)}(x,\tilde A)-1}(d^{(i)}(x,\tilde A)-s)\left|\nabla_1^2u(\Psi^{(i)}_x(s))\right|\right)^2\\
&\quad\le\left(\sum_{s=0}^{d^{(i)}(x,\tilde A)-1}\frac{1}{(d^{(i)}(x,\tilde A)-s)^{\dd-3}}\right)\left(\sum_{s=0}^{d^{(i)}(x,\tilde A)-1}(d^{(i)}(x,\tilde A)-s)^{\dd-1}\left|\nabla_1^2u(\Psi^{(i)}_x(s))\right|^2\right)\\
&\quad\le C_\dd\sum_{s=0}^{d^{(i)}(x,\tilde A)-1}(d^{(i)}(x,\tilde A)-s)^{\dd-1}\left|\nabla_1^2u(\Psi^{(i)}_x(s))\right|^2\,.
\end{aligned}
\end{equation}
If $s> \tilde\alpha_\dd$, we know $\Psi^{(i)}_x(s)\in x+\Xi_i$, and hence $\Psi^{(i)}_x(s)+\Xi_i\subset x+\Xi_i$, which implies $d^{(i)}(\Psi^{(i)}_x(s),\tilde A)\ge d^{(i)}(x,\tilde A)-s$. If $s\le\tilde\alpha_\dd$, we can just use the estimate $(d^{(i)}(x,\tilde A)-s)^{\dd-1}\le R^{\dd-1}$. Using this in \eqref{eq:localpoincare3} obtain
\begin{align*}
	&\left(\sum_{s=0}^{d^{(i)}(x,\tilde A)-1}(d^{(i)}(x,\tilde A)-s)\left|\nabla_1^2u(\Psi^{(i)}_x(s))\right|\right)^2\\
	&\quad\le C_\dd\sum_{s=\tilde\alpha_\dd}^{d^{(i)}(x,\tilde A)-1}d^{(i)}(\Psi^{(i)}_x(s),\tilde A)^{\dd-1}\left|\nabla_1^2u(\Psi^{(i)}_x(s))\right|^2+C_\dd\sum_{s=0}^{\tilde\alpha_\dd-1}R^{\dd-1}\left|\nabla_1^2u(\Psi^{(i)}_x(s))\right|^2\\
	&\quad\le C_\dd\sum_{y\in\Psi^{(i)}_x}\left(\I_{|y-x|_1> \tilde\alpha_\dd}d^{(i)}(y,\tilde A)^{\dd-1}+\I_{|y-x|_1\le \tilde\alpha_\dd}R^{\dd-1}\right)\left|\nabla_1^2u(y)\right|^2\,.
\end{align*}

We can insert this into \eqref{eq:localpoincare2} and change the order of summation once more to obtain that
\begin{equation}\label{eq:localpoincare5}
\begin{aligned}
	&\sum_{x\in V}|u(x)|^2\I_{d_*(x,\tilde A)\le R}\\
	&\quad\le C_\dd\sum_{i=1}^{\dd+1}\sum_{x\in V}\I_{d_*(x,\tilde A)\le R}\sum_{y\in\Psi^{(i)}_x}\left(\I_{|y-x|_1> \tilde\alpha_\dd}d^{(i)}(y,\tilde A)^{\dd-1}+\I_{|y-x|_1\le \tilde\alpha_\dd}R^{\dd-1}\right)\left|\nabla_1^2u(y)\right|^2\\
	&\quad\le C_\dd\sum_{i=1}^{\dd+1}\sum_{y\in V+Q_R(0)}\left|\nabla_1^2u(y)\right|^2\sum_{x\colon y\in\Psi^{(i)}_x}\I_{d_*(x,\tilde A)\le R}\left(\I_{|y-x|_1> \tilde\alpha_\dd}d^{(i)}(y,\tilde A)^{\dd-1}+\I_{|y-x|_1\le \tilde\alpha_\dd}R^{\dd-1}\right)\\
	&\quad=C_\dd\sum_{i=1}^{\dd+1}\sum_{y\in V+Q_R(0)}\left|\nabla_1^2u(y)\right|^2\bigg(\left|\left\{x\colon y\in\Psi^{(i)}_x,|y-x|_1> \tilde\alpha_\dd,d_*(x,\tilde A)\le R\right\}\right|d^{(i)}(y,\tilde A)^{\dd-1}\\
	&\qquad\qquad\qquad\qquad\qquad\qquad\qquad\qquad\qquad+\left|\left\{x\colon y\in\Psi^i(x),|y-x|_1\le \tilde\alpha_\dd\right\}\right|R^{\dd-1}\bigg)\,.
\end{aligned}
\end{equation}
The cardinality of the second set here is trivial to estimate and we find $|\{x\colon y\in\Psi^{(i)}(x),|y-x|_1\le \tilde\alpha_\dd\}|\le\tilde\alpha_\dd$. To estimate the cardinality of the first set we need to work a bit. The heuristic here is that the paths $\Psi^{(i)}_x$ are close to straight lines with the same endpoint passing through $y$, so there cannot be too many of them. To make this precise, fix $y$ and consider some $x$ such that $y\in\Psi^{(i)}_x$, $|y-x|_1> \tilde\alpha_\dd$ and $d_*(x,\tilde A)\le R$. Because $|y-x|_1> \tilde\alpha_\dd$, we know that $y\in x+\Xi_i$, and hence $y+\Xi_i\subset x+\Xi_i$. Thus, $a^{(i)}_x$ is one of the candidates for the endpoint of the path $\Psi^{(i)}_y$, and our definition of the paths ensures that we actually have $a^{(i)}_x=a^{(i)}_y$. Because $y\in\Psi^{(i)}_x$, the point $y$ has distance at most $\sqrt{\dd}$ from the straight line connecting $x$ and $a^{(i)}_x=a^{(i)}_y$. This also means that $x$ has distance at most $\sqrt{\dd}$ from the straight line connecting $y$ and $a^{(i)}_y$. Therefore $x$ is contained in some fixed cone with tip $a^{(i)}_y$ and opening angle $\le \frac{C_\dd}{|a^{(i)}_y-y|}=\frac{C_\dd}{d^{(i)}(y,\tilde A)}$. The point $x$ is also contained in the cube around $a^{(i)}_y$ with diameter $2R$, as otherwise $d_*(x,\tilde A)\le d^{(i)}(x,\tilde A)=|x-a^{(i)}_y|_1\ge|x-a^{(i)}_y|_\infty>R$. Thus, $x$ is contained in the intersection of the aforementioned cone with that cube. This intersection contains at most $C_\dd R\left(\frac{R}{d^{(i)}(y,\tilde A)}\right)^{\dd-1}=C_\dd \frac{R^\dd}{(d^{(i)}(y,\tilde A))^{\dd-1}}$ lattice points, and so
\[\left|\left\{x\colon y\in\Psi^{(i)}_x,|y-x|_1> \tilde\alpha_\dd,d_*(x,\tilde A)\le R\right\}\right|\le C_\dd \frac{R^\dd}{(d^{(i)}(y,\tilde A))^{\dd-1}}\,.\]

Returning now to \eqref{eq:localpoincare5}, we find
\begin{align*}
  \sum_{x\in V}|u(x)|^2\I_{d_*(x,\tilde A)\le R}&\le C_\dd\sum_{i=1}^{\dd+1}\sum_{y\in V+Q_R(0)}\left|\nabla_1^2u(y)\right|^2\left(\frac{R^\dd}{d^{(i)}(y,\tilde A)^{\dd-1}}d^{(i)}(y,\tilde A)^{\dd-1}+\tilde\alpha_\dd R^{\dd-1}\right)\\
	&\le C_\dd R^\dd\sum_{i=1}^{\dd+1}\sum_{y\in V+Q_R(0)}\left|\nabla_1^2u(y)\right|^2\\
	&\le C_\dd R^\dd\sum_{y\in V+Q_R(0)}\left|\nabla_1^2u(y)\right|^2\,.
\end{align*}
This completes the proof in the case $\dd\ge5$. The case $\dd=4$ is very similar. The only difference is that in the estimate \eqref{eq:localpoincare3} we no longer have
\[\sum_{s=0}^{d^{(i)}(x,\tilde A)-1}\frac{1}{(d^{(i)}(x,\tilde A)-s)^{\dd-3}}\le C_\dd\]
but instead
\[\sum_{s=0}^{d^{(i)}(x,\tilde A)-1}\frac{1}{d^{(i)}(x,\tilde A)-s}\le C\log(2+d^{(i)}(x,\tilde A))\le C\log R\,.\]
This is the additional factor $\log R$ that appears on the right hand side in \eqref{e:localpoincare}
\end{proof}

Later we will also use a probabilistic quenched version of this estimate.
\begin{lemma}\label{l:localpoincare_prob}
Let $\dd\ge4$. There is an odd integer $N_\dd$ with the following property: Let $\Lambda\Subset\Z^\dd$, and let $x\in\Lambda$, and $k\in\N$. Then if $\e$ is sufficiently small (depending on $\dd$) there is an event $\Omega_{x,k}$ such that $\zeta^\e_\Lambda(\Omega_{x,k})\ge1-\frac{1}{2^{k^\dd}}$ and such that whenever $A\in\Omega_{x,k}$ the following estimate holds: If $u\colon\Z^\dd\to\R$ is a function such that $u=0$ on $\tilde A$, then 
\begin{equation}\label{e:localpoincare_prob}
|u(x)|\le C_\dd\frac{k^{\dd/2}\left(1+\I_{\dd=4}((\log k)^{1/2}+|\log\e|^{3/4})\right)}{\e^{1/2}}\|\nabla_1^2u\|_{L^2(Q_{kN_\dd\lmic}(x))}\,.
\end{equation}

\end{lemma}
\begin{proof}
We want to apply Theorem \ref{t:localpoincare} with $V=\{x\}$ and $R=kN\lmic$. Then $R^\dd(1+\I_{\dd=4}\log R)\le C_{\dd,N}\frac{k^\dd(1+\I_{\dd=4}(\log k+|\log\e|^{3/2})}{\e}$, and so \eqref{e:localpoincare_prob} follows with the choice $N_\dd=N$ provided that $d_*(x,\tilde \A)\le kN\lmic$. Thus, if we define
\[\Omega_{x,k}=\left\{A\subset\Lambda\colon d_*(x,\tilde A)\le kN\lmic\right\}\]
it remains to choose $N$ in such a way that we can show that $\zeta^\e_\Lambda(\Omega_{x,k})\ge1-\frac{1}{2^k}$.

For an odd integer $N$, let $\Xi_{i,kN\lmic}(x)=(x+\Xi_i)\cap Q_{kN\lmic}(x)$. When $kN\ge N_\dd'$ for some dimensional constant $N_\dd'$ (and so in particular when $N\ge N_\dd'$) the fraction of points in $Q_{kN\lmic}(x)$ that are in $\Xi_{i,kN\lmic}(x)$ is bounded below, i.e. we have $|\Xi_{i,kN\lmic}(x)|\ge\frac{(kN)^\dd\lmic^\dd}{C_\dd}$. On the other hand, $d_*(x,\tilde A)\le kN\lmic$ holds if and only if all $\Xi_{i,kN\lmic}(x)$ contain some point in $\tilde A$. Therefore, using Theorem \ref{t:estimatespinnedset} c) we see
\begin{align*}
1-\zeta^\e_\Lambda(\Omega_{x,k})&=\zeta^\e_\Lambda(d_*(x,\tilde \A)> kN\lmic)\\
&\le\zeta^\e_\Lambda(d_*(x,\A)> kN\lmic)\\
&=\zeta^\e_\Lambda\left(\exists i\in\{1,\ldots,\dd+1\}\colon \A\cap\Xi_{i,kN\lmic}(x)=\varnothing\right)\\
&=\sum_{i=1}^{\dd+1}\zeta^\e_\Lambda\left(\A\cap\Xi_{i,kN\lmic}(x)=\varnothing\right)\\
&\le\sum_{i=1}^{\dd+1}(1-p_{\dd,-})^{|\Xi_{i,kN\lmic}(x)|}\\
&\le\sum_{i=1}^{\dd+1}\exp\left(-p_{\dd,-}\frac{(kN)^\dd\lmic^\dd}{C_\dd}\right)\,.
\end{align*}
For any $\dd\ge4$ we have $p_{\dd,-}\lmic^\dd\ge\frac{1}{C_\dd}$, and so
\begin{align*}
1-\zeta^\e_\Lambda(\Omega_{x,k})&
\le(\dd+1)\exp\left(-\frac{(kN)^\dd}{C_\dd}\right)\\
&\le(\dd+1)\left((\dd+1)\exp\left(-\frac{N^\dd}{C_\dd}\right)\right)^k
\end{align*}
and it suffices to choose $N_\dd\ge N_\dd'$ in such a way that the right hand side is less than $\frac{1}{2^k}$ when $N\ge N_\dd$.

\end{proof}

\subsection{An interpolation inequality}\label{s:inter_ineq}
Let $Q\subset\Lambda$ be a discrete cube of sidelength $R$. In the following section we will need to control $\|\nabla_1u\|_{L^2(Q)}$ by terms involving only $u$ and $\nabla_1^2u$. Usually one would expect to do this by an interpolation inequality of the form 
\[\|\nabla_1u\|^2_{L^2(Q)}\le C_\dd\left(R^2\|\nabla_1^2u\|^2_{L^2(Q)}+\frac{1}{R^2}\|u\|^2_{L^2(Q)}\right)\]
where the factors $R^{\pm2}$ are due to scaling. For our purposes this is not good enough, however, as we do not control $u$ on all of $Q$. So it is crucial for us that a similar inequality still holds when we only control $u$ on a large enough subset of $Q$. Indeed we have the following result.

\begin{lemma}\label{l:interpolationineq}
Let $\dd\in\N$. Let $R$ be an odd integer and let $Q\subset\Lambda$ be a discrete cube of sidelength $R$ (i.e. $Q=Q_R(x_*)$ for some $x_*\in\Z^\dd$), and assume $R\ge 12(\sqrt{\dd})^\dd$. Let $B\subset Q$ such that $|B|\ge\frac12|Q|$. Let $u\colon\Lambda\to\R$. Then we have the estimate
\begin{equation}\label{e:interpolationineq}
\|\nabla_1u\|^2_{L^2(Q)}\le C_\dd\left(R^2\|\nabla_1^2u\|^2_{L^2(Q)}+\frac{1}{R^2}\|\I_{\cdot\in B}u\|^2_{L^2(Q)}\right)\,.
\end{equation}
\end{lemma}
\begin{proof}
By translation invariance we can assume that $Q$ is centred at 0.

We first prove \eqref{e:interpolationineq} with $u$ replaced by an affine function $v$, where $v(x)=b\cdot x+a$ for some $a\in\R$, $b\in\R^\dd$. That is, we want to show
\begin{equation}\label{e:interpolationineq_1}
|b|^2\le C\frac{1}{R^{\dd+2}}\|\I_{\cdot\in B}v\|^2_{L^2(Q)}\,.
\end{equation}
To see this, note first that we can assume $b\neq0$ (else there is nothing to show). Let $\theta=\frac{b}{|b|}\in \Sphere^{\dd-1}$. For $\lambda>0$ consider the set $E=\left\{x\in Q\colon\left|\theta\cdot x+\frac{a}{|b|}\right|\le\lambda\right\}$. The set $E$ is the intersection of $Q$ with a slab of width $2\lambda$, and so for each point $x\in E$ the cube $x+\left[-\frac12,\frac12\right]^\dd$ is contained in $\left[-\frac{R}2,\frac{R}2\right]^\dd$ intersected with a slab of width $2\lambda+\sqrt{\dd}$. Thus, we can estimate the number of points in $E$ as
\[|E|\le(2\lambda+\sqrt{\dd})(R\sqrt{\dd})^{\dd-1}\,.\]
If we assume $\lambda\ge\sqrt{\dd}$, we can bound this by
\[|E|\le3\lambda(R\sqrt{\dd})^{\dd-1}\,.\]
We want to pick $\lambda=\frac{R}{12(\sqrt{\dd})^{\dd-1}}$. This is possible, as $\frac{R}{12(\sqrt{\dd})^{\dd-1}}\ge\sqrt{\dd}$ by our assumption on $R$. Then for our choice of $\lambda$ we see that $|E|\le\frac14R^\dd=\frac14|Q|$.

On the other hand, we know $|B|\ge\frac12|Q|$, and therefore $|B\setminus E|\ge\frac14|Q|$. Now for each $x\in B\setminus E\subset Q\setminus E$ we have $\left|\theta\cdot x+\frac{a}{|b|}\right|>\lambda$ and hence 
\[|v(x)|=|b|\left|\theta\cdot x+\frac{a}{|b|}\right|\ge\lambda|b|\ge \frac{R|b|}{C_\dd}\]
Summing the square of this estimate over all $x\in B\setminus E$, we see that
\[\sum_{x\in B\setminus E}|v(x)|^2\ge\sum_{x\in B\setminus E}\frac{R^2|b|^2}{C_\dd}\ge \frac14|Q|\frac{R^2|b|^2}{C_\dd}\]
which immediately implies \eqref{e:interpolationineq_1}.

Let now $(u)_Q:=\frac{1}{|Q|}\sum_{x\in Q}u(x)$ and $(\nabla_1u)_Q:=\frac{1}{|Q|}\sum_{x\in Q}\nabla_1u(x)$, and define $v(x)=(u)_Q+(\nabla_1u)_Q\cdot x$. Then $v$ is an affine function to which we will be able to apply \eqref{e:interpolationineq_1}, while $u-v$ and $\nabla_1(u-v)$ have average zero over $Q$, which allows using the discrete Poincaré inequality with zero mean. We can thus write

\begin{align*}
\|\nabla_1u\|^2_{L^2(Q)}&\le2\|\nabla_1v\|^2_{L^2(Q)}+2\|\nabla_1(u-v)\|^2_{L^2(Q)}\\
&\le\frac{2}{R^2}\|\I_{\cdot\in B}v\|_{L^2(Q)}^2+C_\dd R^2\|\nabla_1^2(u-v)\|^2_{L^2(Q)}\\
&=\frac{2}{R^2}\|\I_{\cdot\in B}v\|_{L^2(Q)}^2+C_\dd R^2\|\nabla_1^2u\|^2_{L^2(Q)}\\
&\le\frac{C_\dd}{R^2}\|\I_{\cdot\in B}(u-v)\|_{L^2(Q)}^2+\frac{C_\dd}{R^2}\|\I_{\cdot\in B}u\|_{L^2(Q)}^2+C_\dd R^2\|\nabla_1^2u\|^2_{L^2(Q)}\\
&\le\frac{C_\dd}{R^2}\|u-v\|_{L^2(Q)}^2+\frac{C_\dd}{R^2}\|\I_{\cdot\in B}u\|_{L^2(Q)}^2+C_\dd R^2\|\nabla_1^2u\|^2_{L^2(Q)}\\
&\le\frac{C_\dd}{R^2}R^4\|\nabla_1^2u\|_{L^2(Q)}^2+\frac{C_\dd}{R^2}\|\I_{\cdot\in B}u\|_{L^2(Q)}^2+C_\dd R^2\|\nabla_1^2u\|^2_{L^2(Q)}\\
&\le\frac{C_\dd}{R^2}\|\I_{\cdot\in B}u\|_{L^2(Q)}^2+C_\dd R^2\|\nabla_1^2u\|^2_{L^2(Q)}\,.
\end{align*}
This is what we wanted to show.
\end{proof}

\section{\texorpdfstring{Probabilistic decay of the $L^2$-norm for biharmonic functions}{Probabilistic decay of the L2-norm for biharmonic functions}}\label{s:decay_L2}

In this section we will prove a decay estimate for the $L^2$-norm of the Hessian of a discrete biharmonic function. This estimate does not hold for all realizations of $\A$, but we prove that it holds for all but an exceptional set of realizations whose probability decays exponentially.

The precise result is the following. Recall that $\tilde\A=\A\cup(\Z^\dd\setminus\Lambda)$.
\begin{theorem}\label{t:decay_high_prob}
Let $\dd\ge4$. There is an odd integer $\hat N_\dd$ with the following property: let $\Lambda\Subset\Z^\dd$, let $U\in\Pbox_{\hat N_\dd\lmac}$ be a polymer consisting of $n=\frac{|U|}{\hat N_\dd^\dd\lmac^\dd}$ boxes, and $k\in\N$. Then if $\e$ is sufficiently small (depending on $\dd$ only) there is an event $\Omega_{U,k}$ such that $\zeta^\e_\Lambda\left(\Omega_{U,k}\right)\ge1-\frac{n}{2^k}$, and such that whenever $A\in\Omega_{U,k}$ the following estimates hold:
\begin{itemize}
	\item[a)]  If $u\colon\Z^\dd\to\R$ is a function such that $u=0$ on $\tilde A\setminus U$ and $u\Delta_1^2u=0$ on $\Z^\dd\setminus U$, we have the estimate
\begin{equation}\label{e:decay_high_prob_ext}
\left\|\nabla_1^2u\right\|^2_{L^2(\Z^\dd\setminus(U+Q_{2k\hat N_\dd\lmac}(0)))}\le\frac{1}{2^k}\left\|\nabla_1^2u\right\|^2_{L^2((U+Q_{2k\hat N_\dd\lmac}(0))\setminus U)}\,.
\end{equation}
	\item[b)]  If $u\colon\Z^\dd\to\R$ is a function such that $u=0$ on $(U+Q_{2k\hat N_\dd\lmac}(0))\cap\tilde A$ and $u\Delta_1^2u=0$ on $U+Q_{2k\hat N_\dd\lmac}(0)$, we have the estimate
\begin{equation}\label{e:decay_high_prob_int}
\left\|\nabla_1^2u\right\|^2_{L^2(U)}\le\frac{1}{2^k}\left\|\nabla_1^2u\right\|^2_{L^2((U+Q_{2k\hat N_\dd\lmac}(0))\setminus U)}\,.
\end{equation}
\end{itemize}
\end{theorem}
We have already outlined the strategy of the proof in Section \ref{s:main_ideas}. Namely, to prove \eqref{e:decay_high_prob_ext} and \eqref{e:decay_high_prob_int} we want to iterate the Widman hole filler argument $k$ times, that is we need to find $k$ pairs $U_j,U_j'$ on which we can apply it.

In order to make the Widman hole filler argument work, we need to be able to apply Theorem \ref{t:localpoincare} and Lemma \ref{l:interpolationineq}. We can ensure this by finding a cut-off function $\eta_j$ that grows from 0 to 1 in such a way that $\nabla_1^2\eta_j=0$ on those points on which Theorem \ref{t:localpoincare} or Lemma \ref{l:interpolationineq} cannot be applied.

For that purpose pick an odd integer $N$ to be fixed later and decompose $\Z^\dd$ into the boxes in $\Qbox_{N\lmac}$. We will declare some of these boxes to be bad in such a way that on the good (i.e. non-bad) boxes we can construct $\eta$ growing from 0 to 1 on that box and satisfying the conditions on $\nabla_1^2\eta$. If we can show that bad boxes are rare, then with high probability we can find at least $k$ annuli consisting only of good boxes inbetween $U$ and $\Z^\dd\setminus (U+Q_{2kN\lmac}(0)$, and then we can iterate the Widman hole filler argument on these annuli.

In Section \ref{s:const_cutoff} we describe in detail how we choose the bad boxes, and we prove that on the good boxes there exist cut-off functions as required. In Section \ref{s:widman_good_box} we carry out the Widman hole filler argument provided all relevant boxes are good. Finally, in Section \ref{s:est_prob_cutoff} we show that the bad boxes are sparse enough that with sufficiently high probability we can find enough annuli to use the hole filler argument on. Using this result we will complete the proof of Theorem \ref{t:decay_high_prob}.

\subsection{Bad boxes and cut-off functions}\label{s:const_cutoff}
The definition of the bad boxes depends on three odd integers $K$, $L$, $M$, where $K$ is always a multiple of 3, and $M\ge12$. Eventually (in Section \ref{s:est_prob_cutoff}), we will choose them large enough in the order $M$, $K$, $L$ to close the argument. For now we will track all dependencies on $K$, $L$, $M$.
The parameters $K$ and $L$ will play similar roles as in the proof of Theorem \ref{t:estimatespinnedset} c), albeit not quite the same. We hope this will not confuse the reader.

There will be two reasons that lead to a cube $Q\in\Qbox_{KL\lmac}$ being bad. The first one is related to Lemma \ref{l:interpolationineq}. We want to be able to apply that Lemma on each subcube $Q'\in\Qbox_{K\lmac}$ such that $Q'\subset Q$, with $\I_{\cdot\in B}$ being the indicator function from Theorem \ref{t:localpoincare}. So we will define $Q$ to be bad of type II if there is a subcube $Q'\subset Q$ on which the indicator function from Theorem \ref{t:localpoincare} is equal to 0 too often.

The second reason is more complicated. We want to be able to modify an initial guess $\eta_*$ for the cut-off function in such a way that $\nabla_1^2\eta=0$ on those sets on which the indicator function from Theorem \ref{t:localpoincare} is equal to 0 (and we think of those points as bad as well). This is easy if the bad points are very isolated and sparse, as we then can make local adjustments to $\eta_*$ that do not interfere with each other. So we start at scale $\ell_0=K\lmic$ and consider the bad points (or actually the bad boxes in $\Qbox_{K\lmic}$ that contain at least a bad point) and split them into the isolated and the clustered ones. The former ones we ignore for the moment, and the latter ones can be covered by cubes in $\Qbox^\#_{\ell_1}$ such that each cube covers at least two of the small cubes. These are the bad cubes on scale $\ell_1$. Now we can once again split those cubes into the isolated ones (that we ignore for the moment) and those that are clustered and can be covered by cubes on scale $\ell_2$, and we continue like this. This process terminates once at some scale all bad cubes are isolated. Then we can adjust $\eta_*$ on those isolated cubes, and then proceed backwards and apply our adjustment also on the isolated cubes on the smaller scales. We thus call $Q$ bad of type I if this process terminates too late.

We have not mentioned yet how to choose the scales $\ell_j$. There is a trade-off here: on the one hand, the lengthscale should grow fast so that we have enough space around each isolated bad cube on scale $\ell_j$ to adjust $\eta_*$ there. On the other hand we want many of the cubes to be isolated, so that our process terminates soon, and this we can achieve by letting $\ell_j$ not grow too fast. It turns out that a good compromise is
\[\ell_j:=M^{j^3}K\lmic\]
for each $j\ge0$. We give a more detailed explanation of this choice later, in Remark \ref{r:lengthscales}.

\bigskip

Let us now give the precise definitions. Recall that $\Qbox_l=\left\{Q_{l/2}(x)\colon x\in \left(l\Z\right)^\dd\right\}$ and $\Qbox^\#_l=\left\{Q_{l/2}(x)\colon x\in \left(\frac{l}{3}\Z\right)^\dd\right\}$. We consider the microscopic cubes in $\Qbox_{K\lmic}$, the macroscopic cubes in $\Qbox_{KL\lmac}$, and the further sets of boxes $\Qbox_{\ell_0}$ and $\Qbox^\#_{\ell_j}$ for $j\ge1$, associated to $\ell_j=M^{j^3}K\lmic$ for $j\ge0$. The reason that we use $\Qbox^\#_{\ell_j}$ and not $\Qbox_{\ell_j}$ for $j\ge1$ is that we want to ensure that for any two cubes on lengthscale $\ell_{j-1}$ that are sufficiently close there is a cube in $\Qbox^\#_{\ell_j}$ that contains both of them. We assumed $M\ge12$, and so $\ell_{j-1}\le\frac{1}{12}\ell_j$ for each $j\ge1$. We fix for later use an enumeration of the boxes in each $\Qbox^\#_{\ell_j}$.

For $\e$ small enough (depending on $\dd$ only) the sets $x+\Xi_i\cap Q_{K\lmic}(x)$ are non-empty for any $x$ and $i$, and so it is possible that $d_*(x,\tilde A)\le K\lmic$. We consider only such $\e$, and define 
\[S^{(0)}_{K,M,\mathrm{bad}}(A)=\left\{Q\in \Qbox_{K\lmic}\colon \exists x\in Q \text{ with } d_*(x,\tilde A)>K\lmic\right\}\,.\]
Note that $S^{(0)}_{K,M,\mathrm{bad}}(A)\subset \Qbox_{K\lmic}=\Qbox_{\ell_0}\subset\Qbox^\#_{\ell_0}$. The set $S^{(0)}_{K,M,\mathrm{bad}}(A)$ is finite (as cubes far outside of $\Lambda$ will not be bad).

For $j\ge1$ we define $S^{(j)}_{K,M,\mathrm{bad}}(A)\subset\Qbox^\#_{\ell_j}$ inductively as follows:
Given $S^{(j-1)}_{K,M,\mathrm{bad}}(A)\subset\Qbox^\#_{\ell_{j-1}}$ such that the set $S^{(j-1)}_{K,M,\mathrm{bad}}(A)$ is finite, we want to split it into two sets of boxes: $S^{(j-1)}_{K,M,\mathrm{bad},\mathrm{clust}}(A)$ will contain those boxes that are clustered in the sense that there is another bad box at distance $\le\frac{\ell_j}{2}$ that is disjoint from the original box, and $S^{(j-1)}_{K,M,\mathrm{bad},\mathrm{isol}}(A)$ will contain the other boxes. These other boxes are isolated in the sense that all bad boxes that are disjoint from them are far away. Let us make this precise: we define
\begin{align*}
S^{(j-1)}_{K,M,\mathrm{bad},\mathrm{clust}}(A)&=\left\{Q\in S^{(j-1)}_{K,M,\mathrm{bad}}(A)\colon \exists Q'\in S^{(j-1)}_{K,M,\mathrm{bad}}(A)\text{ with }Q'\cap Q=\varnothing,d_\infty(Q,Q')\le\frac{\ell_j}{2}\right\}\,,\\
S^{(j-1)}_{K,M,\mathrm{bad},\mathrm{isol}}(A)&=S^{(j-1)}_{K,M,\mathrm{bad}}(A)\setminus S^{(j-1)}_{K,M,\mathrm{bad},\mathrm{clust}}(A)\,.
\end{align*}
If $S^{(j-1)}_{K,M,\mathrm{bad},\mathrm{clust}}(A)=\varnothing$, we define $S^{(j)}_{K,M,\mathrm{bad}}(A)=\varnothing$. Otherwise, note that if $Q\in S^{(j-1)}_{K,M,\mathrm{bad},\mathrm{clust}}(A)$ and $Q'\in S^{(j-1)}_{K,M,\mathrm{bad}}(A)$ is a witness for this in the sense that $Q'\cap Q=\varnothing$ and $d_\infty(Q,Q')\le\frac{\ell_j}{2}$ then also $Q'\in S^{(j-1)}_{K,M,\mathrm{bad},\mathrm{clust}}(A)$, as $Q$ then is a witness for $Q'$. In particular,
\begin{align*}
&S^{(j-1)}_{K,M,\mathrm{bad},\mathrm{clust}}(A)\\
&\quad=\left\{Q\in S^{(j-1)}_{K,M,\mathrm{bad}}(A)\colon \exists Q'\in S^{(j-1)}_{K,M,\mathrm{bad},\mathrm{clust}}(A)\text{ with }Q'\cap Q=\varnothing,d_\infty(Q,Q')\le\frac{\ell_j}{2}\right\}\,.\end{align*}
Furthermore, if $Q,Q'\in S^{(j-1)}_{K,M,\mathrm{bad},\mathrm{clust}}(A)$ with $Q'\cap Q=\varnothing$ and $d_\infty(Q,Q')\le\frac{\ell_j}{2}$, then $Q$ and $Q'$ are contained in a common box of sidelength at most $\frac{\ell_j}{2}+2\ell_{j-1}\le\frac{2\ell_j}{3}$ (as $\ell_{j-1}\le\frac{1}{12}\ell_j$), and so there is a cube in $\Qbox^\#_{\ell_j}$ containing both of them. This means that we can cover $S^{(j-1)}_{K,M,\mathrm{bad},\mathrm{clust}}(A)$ by finitely many cubes from $\Qbox^\#_{\ell_j}$ in such a way that each cube from the cover contains at least two small cubes from $S^{(j-1)}_{K,M,\mathrm{bad},\mathrm{clust}}(A)$ that are disjoint. We can consider the set of subsets of $\Qbox^\#_{\ell_j}$ with that property,
\begin{align*}
\mathcal{S}^{(j)}_{K,M,\mathrm{bad}}(A)=\bigg\{S\subset \Qbox^\#_{\ell_j}&\colon \bigcup_{Q'\in S^{(j-1)}_{K,M,\mathrm{bad},\mathrm{clust}}(A)}Q'\subset \bigcup_{Q\in S}Q,\\
&\forall Q\in S\ \exists Q',Q''\in S^{(j-1)}_{K,M,\mathrm{bad},\mathrm{clust}}(A)\text{ with }Q'\cap Q''=\varnothing,Q'\cup Q''\subset Q\bigg\}\,.
\end{align*}
The previous discussion implies that if $S^{(j-1)}_{K,M,\mathrm{bad},\mathrm{clust}}(A)$ is non-empty, then also $\mathcal{S}^{(j)}_{K,M,\mathrm{bad}}(A)$ is non-empty. Now we consider the elements of $\mathcal{S}^{(j)}_{K,M,\mathrm{bad}}(A)$ of minimum cardinality, and among those we define $S^{(j)}_{K,M,\mathrm{bad}}(A)$ to be that element that is lexicographically first (with respect to the enumeration of $\Qbox^\#_{\ell_j}$ that we had fixed).

To summarize, $S^{(j)}_{K,M,\mathrm{bad}}(A)$ is a finite subset of $\Qbox^\#_{\ell_j}$ that covers $S^{(j-1)}_{K,M,\mathrm{bad},\mathrm{clust}}(A)$ in such a way that each of its boxes contains two disjoint elements from $S^{(j-1)}_{K,M,\mathrm{bad},\mathrm{clust}}(A)$. 

Finally we can define macroscopic bad boxes. Let $j_*(\e)$ be the largest integer $j$ such that $\ell_j\le\frac{KL\lmac}{8}$ (we assume that $\e$ is small enough so that $j_*(\e)\ge1$). We consider the macroscopic boxes in $\Qbox_{KL\lmac}$, and call a macroscopic box bad of type I if it contains at least one box in $S^{(j_*(\e))}_{K,M,\mathrm{bad}}(A)$, bad of type II if one of its $K\lmac$-subboxes contains many boxes in $S^{(0)}_{K,M,\mathrm{bad}}(A)$, and bad if it is bad of type I or type II. More precisely
\begin{align*}
S^{*,I}_{K,L,M,\mathrm{bad}}(A)&=\left\{Q\in \Qbox_{KL\lmac}\colon \exists q\in S^{(j_*(\e))}_{K,M,\mathrm{bad}}(A)\text{ with }q\cap Q\neq\varnothing \right\}\,,\\
S^{*,II}_{K,L,M,\mathrm{bad}}(A)&=\Bigg\{Q\in \Qbox_{KL\lmac}\colon\exists Q'\in\Qbox_{K\lmac}\text{ with }Q'\subset Q,\\
&\qquad\qquad\left|\left\{q\in S^{(0)}_{K,M,\mathrm{bad}}(A)\colon q\subset Q'\right\}\right|\ge\frac14\left(\frac{\lmac}{\lmic}\right)^\dd\Bigg\}\,,\\
S^*_{K,L,M,\mathrm{bad}}(A)&=S^{*,I}_{K,L,M,\mathrm{bad}}(A)\cup S^{*,II}_{K,L,M,\mathrm{bad}}(A)\,.
\end{align*}
The point of these definitions is that we can apply our construction of a cut-off function on all cubes in $\Qbox_{KL\lmac}\setminus S^{*,I}_{K,L,M,\mathrm{bad}}(A)$, while we can apply Lemma \ref{l:interpolationineq} on every $K\lmac$-subcube of the cubes in $\Qbox_{KL\lmac}\setminus S^{*,II}_{K,L,M,\mathrm{bad}}(A)$. Of course, this is only useful if we show that bad cubes are rare. This will be established in Section \ref{s:est_prob_cutoff}. For now we show that our definition of good boxes fulfills its purpose in the sense that we can construct a cut-off function growing from 0 to 1 on them in such a way that their second derivatives are 0 on the set of microscopic bad boxes.
\begin{lemma}\label{l:existence_cutoff}
Let $\dd\ge1$. Then there is a constant $M_\dd\ge12$ with the following property: Let $K$, $L$, $M$ be odd integers such that $K$ is a multiple of 3, and $M\ge M_\dd$. Then for all $\e$ sufficiently small (depending on $\dd$, $K$) the following holds: Let $U\in\Pbox_{KL\lmac}$ be a polymer. Suppose that none of the $KL\lmac$-boxes touching its boundary (in the $l^\infty$-sense) are bad of type I, i.e. \[\{Q\in\Qbox_{KL\lmac}\colon Q\subset (U+Q_{KL\lmac}(0))\setminus U\}\cap S^{*,I}_{K,L,M,\mathrm{bad}}(A)=\varnothing\]
Then there is a function $\eta\colon\Z^\dd\to\R$ such that
\begin{itemize}
	\item[i)] $\eta(x)=0$ for $x\in U+Q_{2K\lmic}(0)$,
	\item[ii)] $\eta(x)=1$ for $x\in\Z^\dd\setminus(U+Q_{KL\lmac-2K\lmic}(0))$,
	\item[iii)] $|\nabla_1^2\eta(x)|\le\frac{C_\dd}{(KL)^2\lmac^2}\I_{Q_1(x)\cap \bigcup_{Q\in S^{(0)}_{K,M,\mathrm{bad}}(A)}Q=\varnothing}$.
\end{itemize}
Here the constant $C_\dd$ depends only on $\dd$.
\end{lemma}
Morally, property iii) is $|\nabla_1^2\eta(x)|\le\frac{C_\dd}{(KL)^2\lmac^2}\I_{x\notin\bigcup_{Q\in S^{(0)}_{K,M,\mathrm{bad}}(A)}Q}$. However, the discrete product rule lets translation operators arise, and this is why we require the slightly stronger condition iii) above.

To prove Lemma \ref{l:existence_cutoff}, we will begin with a function $\eta_*$ which satisfies i) and ii), but only $|\nabla_1^2\eta(x)|\le\frac{C_\dd}{(KL)^2\lmac^2}$ instead of iii). Then we will modify $\eta_*$ iteratively to make it affine on larger and larger subsets of $S^{(0)}_{K,M,\mathrm{bad}}(A)$, so that eventually iii) is satisfied as well. The following lemma gives details on how to carry out a single of these modification steps.
\begin{lemma}\label{l:construct_affine_correction}
Let $\dd\ge1$. There is a constant $\gamma_\dd>0$ with the following property: Let $x\in\Z^\dd$, let $r,R$ be positive integers such that $R\ge 16r$. Let $v\colon \Z^\dd\to\R$ be a function. Then there is a function $w\colon \Z^\dd\to\R$ with the following properties:
\begin{itemize}
	\item[i)] $w=0$ on $\Z^\dd\setminus Q_{R-1}(x)$,
	\item[ii)] $\nabla_1^2(v+w)=0$ on $Q_r(x)$,
	\item[iii)] $\|\nabla_1^2(v+w)\|_{L^\infty(Q_R(x))}\le \left(1+\frac{\gamma_\dd}{\log R-\log r}\right)\|\nabla_1^2v\|_{L^\infty(Q_R(x))}$.
\end{itemize}
\end{lemma}
Note that condition i) ensures that $\nabla_1^2w=0$ on $\Z^\dd\setminus Q_R(x)$.

\begin{proof}
By translation invariance we can assume $x=0$. Suppose for the moment that there is a function $\xi\colon\Z^\dd\to\R$ such that $\xi=1$ on $Q_{r+1}(0)$, $\xi=0$ on $\Z^\dd\setminus Q_{R-1}(0)$ and 
\begin{align*}
	|\xi(y)|&\le 1\,,\\
	|\nabla_1\xi(y)|&\le\frac{C_\dd}{|y|(\log R-\log r)}\,,\\
	|\nabla_1^2\xi(y)|&\le\frac{C_\dd}{|y|^2(\log R-\log r)}
\end{align*}
 for $y\in Q_R(0)$. 

Let $u$ be the affine function $u(y)=v(0)+y\cdot\nabla_1v(0)$. Then we can set $w(y)=\xi(y)(u(y)-v(y))$. This choice of $w$ clearly satisfies i) and ii), and so we only have to check iii). We know that $u(0)-v(0)=0$ and $\nabla_1u(0)-\nabla_1v(0)=0$, and so by discrete Taylor expansion, using that $u$ is affine, we have
\begin{align*}
|\nabla_1u(y)-\nabla_1v(y)|&\le C_\dd|y|\|\nabla_1^2v\|_{L^\infty(Q_R(0))}\\
|u(y)-v(y)|&\le C_\dd|y|^2\|\nabla_1^2v\|_{L^\infty(Q_R(0))}\,.
\end{align*}
Note that $v+w=\xi u+(1-\xi)v$. The discrete product rule allows us to write
\begin{align*}
D^1_iD^1_{-j}(v+w)(y)&=D^1_iD^1_{-j}v(y)\tau^1_i\tau^1_{-j}(1-\xi)(y)+D^1_i(u-v)(y)\tau^1_iD^1_{-j}\xi(y)\\
&\quad+D^1_{-j}(u-v)(y)\tau^1_{-j}D^1_i\xi(y)+(u-v)(y)D^1_iD^1_{-j}\xi(y)
\end{align*}
and thus, using the estimates on $\xi$ and $v-u$,
\begin{align*}
|D^1_iD^1_{-j}(v+w)(y)|&\le |D^1_iD^1_{-j}v(y)|+|D^1_i(u-v)(y)|\max_{|z-y|_\infty\le1}|\nabla_1\xi(z)|\\
&\quad+|D^1_{-j}(u-v)(y)|\max_{|z-y|_\infty\le1}|\nabla_1\xi(z)|+|(u-v)(y)|\max_{|z-y|_\infty\le1}|\nabla_1^2\xi(z)|\\
&\le|D^1_iD^1_{-j}v(y)|+C_\dd|y|\|\nabla_1^2v\|_{L^\infty(Q_R(0))}\frac{1}{|y|(\log R-\log r)}\\
&\quad+C_\dd|y|^2\|\nabla_1^2v\|_{L^\infty(Q_R(0))}\frac{1}{|y|^2(\log R-\log r)}\\
&\le|D^1_iD^1_{-j}v(y)|+C_\dd\|\nabla_1^2v\|_{L^\infty(Q_R(0))}\frac{1}{(\log R-\log r)}\,.
\end{align*}
This immediately implies that $v$ satisfies iii).

It remains to show the existence of $\xi$ with the desired properties. To do so, we choose a function $\chi\in C^\infty(\R^\dd)$ that is 1 on $[-1,1]^\dd$, 0 outside of $[-2,2]^\dd$ such that $0\le\chi\le1$, and for $\rho>0$ define $\chi_\rho=\chi\left(\frac{\cdot}{\rho}\right)$. We define $\tilde\xi\colon\R^\dd\to\R$ by
\[\tilde\xi(y)=\chi_{2r}(y)+\left(\chi_{R/4}(y)-\chi_{2r}(y)\right)\frac{\log R-\log|y|}{\log R-\log r}\,.\]
One can check that $0\le\tilde\xi\le1$, $\tilde\xi=1$ on $[-2r,2r]^\dd\supset Q_{r+1}(0)$, $\tilde\xi=0$ on $\R^\dd\setminus\left[-\frac R2,\frac R2\right]^\dd\supset \Z^\dd\setminus Q_{R-1}R(0)$ as well as
\[|\nabla^k\xi(y)|\le\frac{C_{\dd,k}}{|y|^k(\log R-\log r)}\]
for $k\ge1$. We can now let $\xi$ be the restriction of $\tilde\xi$ to $\Z^\dd$. The estimates on $\tilde\xi$ together with Taylor's theorem then imply the corresponding estimates on $\xi$.
\end{proof}

Before we turn to the proof of Lemma \ref{l:existence_cutoff}, let us investigate the structure of the $S^{(j)}_{K,M,\mathrm{bad}}(A)$ in more detail. Note first that
\[\left|S^{(j)}_{K,M,\mathrm{bad}}(A)\right|\ge\left|S^{(j)}_{K,M,\mathrm{bad},\mathrm{clust}}(A)\right|\ge\left|S^{(j+1)}_{K,M,\mathrm{bad}}(A)\right|+1\]
as each cube of $S^{(j+1)}_{K,M,\mathrm{bad}}(A)$ covers at least two cubes of $S^{(j)}_{K,M,\mathrm{bad},\mathrm{clust}}(A)$, and $S^{(j+1)}_{K,M,\mathrm{bad}}(A)$ is chosen with the smallest possible cardinality. We have already noted that $S^{(j)}_{K,M,\mathrm{bad}}(A)$ is a finite set when $\e$ is small enough. This implies that $S^{(0)}_{K,M,\mathrm{bad}}(A)=\varnothing$ for $j$ sufficiently large.

If $Q\in S^{(j)}_{K,M,\mathrm{bad}}(A)$ for some $j\ge0$, then either $Q\in S^{(j)}_{K,M,\mathrm{bad},\textrm{isol}}(A)$ or $Q\in S^{(j)}_{K,M,\mathrm{bad},\textrm{clust}}(A)$. In the latter case there is at least one $Q'\in S^{(j+1)}_{K,M,\mathrm{bad}}(A)$ such that $Q\subset Q'$. If there is more than one such $Q'$, we choose the one that comes first with respect to the enumeration of $\Qbox^\#_{\ell_j}$ that we had fixed, and call it the parent of $Q$. In this manner, given $Q\in S^{(0)}_{K,M,\mathrm{bad}}(A)$ we can find a sequence \[Q\subset Q^{(1)}\subset \ldots \subset Q^{(j)}\]
such that each cube is the parent of the preceding cube. This sequence is necessarily finite as $S^{(j)}_{K,M,\mathrm{bad}}(A)=\varnothing$ for $j$ sufficiently large. It terminates as soon as we reach a cube $Q^{(j)}\in S^{(j)}_{K,M,\mathrm{bad},\mathrm{isol}}(A)$. We denote that value $j$ by $j_{\textrm{isol},Q}$. In summary, we have a sequence
\begin{equation}\label{e:badboxesondiffscales}
Q=Q^{(0)}\subset Q^{(1)}\subset \ldots \subset Q^{(j_{\textrm{isol},Q})}
\end{equation}
where each cube is the parent of the preceding cube, the $Q^{(j)}$ for $j<j_{\textrm{isol},Q}$ are in $S^{(j)}_{K,M,\mathrm{bad},\mathrm{clust}}(A)$, while $Q^{(j_{\textrm{isol},Q})}\in S^{(j_{\textrm{isol},Q})}_{K,M,\mathrm{bad},\mathrm{isol}}(A)$.
This allows us to find for each cube $Q\in S^{(0)}_{K,M,\mathrm{bad}}(A)$ a lengthscale $\ell_{j_{\textrm{isol},Q}}$ on which its parents become isolated. It will be that lengthscale on which we will ensure that $\eta$ is locally affine on $Q$.

After these preparations we can turn to the proof of the main result of this section, the construction of a cut-off function.

\begin{proof}[Proof of Lemma \ref{l:existence_cutoff}]$ $

\emph{Step 1: Construction of a function satisfying a weaker version of iii)}\\
We assume that $\e$ is small enough so that $\lmic\ge4$, say. Then also $\ell_j\ge4$ for all $j\ge0$.
We first claim that there is a function $\eta_*$ satisfying i) and ii) and such that $|\nabla_1^2\eta_*(x)|\le\frac{C_\dd}{(KL)^2\lmac^2}$. This should be intuitively clear, as we want to interpolate from 0 to 1 on scale $KL\lmac$. One way to make this rigorous is as follows.
Let $\tilde U=U+\left[-\frac12,\frac12\right]^\dd$. Choose a function $\tilde\chi\in C^\infty(\R^\dd)$ that is 1 on $[-1,1]^\dd$, 0 outside of $\left[-\frac97,\frac97\right]^\dd$ and such that $0\le\tilde\chi\le1$. For each $Q\in\Qbox_{KL\lmac}$ such that $Q\subset U$ let $x_Q$ be its centre, and define $\tilde\eta_*\colon\R^\dd\to\R$ by
\[\tilde\eta_*(y)=\prod_{\substack{Q\in \Qbox_{KL\lmac}\\Q\subset U}}\left(1-\tilde\chi\left(\frac{y-x_Q}{\frac78KL\lmac}\right)\right)\,.\]
This function is then equal to 0 on $\tilde U+\left[-\frac{3KL\lmac}{8},\frac{3KL\lmac}{8}\right]^\dd$ and also equal to 1 on $\R^\dd\setminus \left(\tilde U+\left[-\frac{5KL\lmac}{8},\frac{5KL\lmac}{8}\right]^\dd\right)$. Each of the factors in the definition of $\tilde\eta_*$ satisfies
\[\left\|\nabla^k\left(1-\tilde\chi\left(\frac{\cdot-x_Q}{\frac78KL\lmac}\right)\right)\right\|_{L^\infty(\R^\dd)}\le\frac{C_{\dd,k}}{(KL)^k\lmac^k}\]
for $k\ge0$. Furthermore, for each fixed $y\in\R^\dd$ there is a neighbourhood in which at most $3^\dd$ of the factors are non-constant. Thus, if we compute $\nabla^k\tilde\eta_*(y)$ we get contributions only from these at most $3^\dd$ factors. Therefore,
\[\|\nabla^k\tilde\eta_*\|_{L^\infty(\R^\dd)}\le\frac{C_{\dd,k}}{(KL)^k\lmac^k} \]
for $k\ge0$. We can now let $\eta_*$ be the restriction of $\tilde\eta_*$ to $\Z^\dd$. Then Taylor's theorem implies easily that $|\nabla_1^2\eta(x)|\le\frac{C_\dd}{(KL)^2\lmac^2}$. 

Note also that $\eta_*$ is equal to 0 on $U+\left[-\frac{3KL\lmac}{8},\frac{3KL\lmac}{8}\right]$ and equal to 1 on $\Z^\dd\setminus (U+Q_{5KL\lmac/8}(0))$. Therefore, $\nabla_1^2\eta_*$ is equal to 0 except possibly on $V:=(U+Q_{5KL\lmac/8+1}(0)\setminus(U+Q_{3KL\lmac/8-1}(0))$.

\emph{Step 2: Modification of $\eta_*$}\\
Let $q\in S^{(0)}_{K,M,\mathrm{bad}}(A)$ with $q\subset (U+Q_{KL\lmac}(0))\setminus U$, and consider the sequence of cubes \eqref{e:badboxesondiffscales} with $Q=q$. By our assumption none of the macroscopic cubes in $(U+Q_{KL\lmac}(0))\setminus U$ are bad of type I, i.e. there is no $Q'\in S^{(j_*(\e))}_{K,M,\mathrm{bad}}(A)$ intersecting $(U+Q_{KL\lmac}(0))\setminus U$. This means that the sequence \eqref{e:badboxesondiffscales} necessarily terminates before $j=j_*(\e)$. In particular, we have $j_{\textrm{isol},q}<j_*(\e)$. We can now partition the cubes in $S^{(0)}_{K,M,\mathrm{bad}}(A)$ according to the value of $j_{\textrm{isol},q}$, and define for $j\in\{0,\ldots,j_*(\e)-1\}$ the set \[T_{K,M,V,j}(A)=\left\{q\in S^{(0)}_{K,M,\mathrm{bad}}(A)\colon q\subset (U+Q_{KL\lmac}(0))\setminus U,j_{\textrm{isol},q}=j\right\}\,.\]

We will now construct a sequence of functions $\eta^{(j)}$ by reverse induction in such a way that $\nabla_1^2\eta^{(j)}=0$ on $\bigcup_{k\ge j}\bigcup_{q\in T_{K,M,V,k}(A)}(q+Q_1(0))$. Eventually, we will show that the choice $\eta=\eta^{(0)}$ satisfies the properties claimed in the Lemma.

Thus, we start with $\eta^{(j_*(\e))}=\eta_*$. Let also $V^{(j_*(\e))}=V$, and note that $\supp\nabla_1^2\eta^{(j_*(\e))}$ is a subset of $V^{(j_*(\e))}$. Suppose now that for some $j\in\{1,\ldots,j_*(\e)\}$ we have defined $\eta^{(j)}$ and $V^{(j)}$ with $\supp\nabla_1^2\eta^{(j)}\subset V^{(j)}$ and $\eta^{(j)}=\eta_*$ on $\Z^\dd\setminus V^{(j)}$ such that $\nabla_1^2\eta^{(j)}=0$ on $\bigcup_{k\ge j}\bigcup_{q\in T_{K,M,V,k}(A)}(q+Q_1(0))$, and let us define $\eta^{(j-1)}$ and $V^{(j-1)}$.

Since $\supp\nabla_1^2\eta^{(j)}\subset V^{(j)}$, we trivially have $\nabla_1^2\eta^{(j)}=0$ on those cubes that do not intersect $V^{(j)}$, and so there is no need to change $\eta^{(j)}$ there. 
Let 
\[Y^{(j-1)}_{K,M,V}(A)=\left\{Q\in S^{(j-1)}_{K,M,\mathrm{bad},\mathrm{isol}}(A)\colon (Q+Q_1(0))\cap V^{(j)}\neq\varnothing\right\}\]
be the set of cubes where we will adjust $\eta^{(j)}$. By definition, this is a set of cubes on scale $\ell_{j-1}$ that either overlap or are far apart. That is, if $Q\in Y^{(j-1)}_{K,M,V}(A)$, then all $Q'\in S^{(j-1)}_{K,M,\mathrm{bad}}(A)$ either satisfy $Q\cap Q'\neq\varnothing$ or $d_\infty(Q,Q')>\frac{\ell_j}2$. Let $\tilde Y^{(j-1)}_{K,M,V}(A)$ be a subset of $Y^{(j-1)}_{K,M,V}(A)$ of maximum cardinality such that the cubes in $\tilde Y^{(j-1)}_{K,M,V}(A)$ are pairwise disjoint. By construction for each $Q\in Y^{(j-1)}_{K,M,V}(A)$ there is a $Q'\in\tilde Y^{(j-1)}_{K,M,V}(A)$ (possibly equal to $Q$) with $Q\cap Q'\neq\varnothing$. In particular,
\[\bigcup_{Q\in Y^{(j-1)}_{K,M,V}(A)}Q\subset \bigcup_{Q\in \tilde Y^{(j-1)}_{K,M,V}(A)}(Q+Q_{\ell_{j-1}}(0))\,.\]

Let now $Q\in \tilde Y^{(j-1)}_{K,M,V}(A)$. We know that there is no cube $Q'\in S^{(j-1)}_{K,M,\mathrm{bad}}(A)$ that intersects $(Q+Q_{\ell_j/2}(0))\setminus (Q+Q_{\ell_{j-1}}(0))$. This has several implications. An obvious one is that the cubes $Q+Q_{\ell_j/4}(0)$ for $Q\in \tilde Y^{(j-1)}_{K,M,V}(A)$ are pairwise disjoint. Slightly less obviously, we claim that for $Q\in \tilde Y^{(j-1)}_{K,M,V}(A)$ there is no $q\in\bigcup_{k\ge j} T_{K,M,V,k}(A)$ such that $q\cap ((Q+Q_{\ell_j/4+1}(0))\setminus (Q+Q_{\ell_{j-1}}(0)))\neq\varnothing$. Indeed, if $q\in\bigcup_{k\ge j}T_{K,M,V,k}(A)$ then we have $j_{\textrm{isol},q}\ge j$, so the sequence \eqref{e:badboxesondiffscales} contains a parent $q\subset q^{(j-1)}\in S^{(j-1)}_{K,M,\mathrm{bad}}(A)$. Then $q^{(j-1)}$ cannot intersect $(Q+Q_{\ell_j/2}(0))\setminus (Q+Q_{\ell_{j-1}}(0))$, and so the same holds true for $q$.

We would now like to apply Lemma \ref{l:construct_affine_correction} to the function $v=\eta^{(j)}$ and the cubes $Q+Q_{\ell_{j-1}+1}(0)\subset Q+Q_{\ell_j/4}(0)$. For that purpose we fix $M_\dd=8\max\left(16,\exp(\gamma_\dd)\right)$, where $\gamma_\dd$ is the constant from Lemma \ref{l:construct_affine_correction}. If $M\ge M_\dd$, then 
\begin{equation}\label{e:existence_cutoff3}
\frac{\left\lfloor \frac{\ell_{j-1}}{2}+\frac{\ell_j}{4}\right\rfloor}{\left\lceil \frac{\ell_{j-1}}{2}+\ell_{j-1}+1\right\rceil}\ge\frac{\frac{\ell_j}{4}}{2\ell_{j-1}}=\frac{1}{8}\frac{M^{j^3}K\lmic}{M^{(j-1)^3}K\lmic}=\frac{1}{8}M^{3j^2-3j+1}\,.
\end{equation}
The term on the right hand side is in particular bounded below by $\frac M{8}\ge16$, and thus we can indeed apply the lemma. We obtain that there is a function $w_Q$ such that $w_Q=0$ on $\Z^\dd\setminus (Q+Q_{\ell_j/4-1}(0))$, $\nabla_1^2\left(w_Q+\eta^{(j)}\right)$ is zero on $Q+Q_{\ell_{j-1}+1}(0)$, and such that
\begin{equation}\label{e:existence_cutoff4}
\begin{aligned}
&\left\|\nabla_1^2\left(w_Q+\eta^{(j)}\right)\right\|_{L^\infty(\Z^\dd)}\\
&\quad\le\left(1+\frac{\gamma_\dd}{\log \left(\frac{\ell_{j-1}}{2}+\frac{\ell_j}{4}\right)-\log \left(\frac{\ell_{j-1}}{2}+\ell_{j-1}+1\right)}\right)\left\|\nabla_1^2\eta^{(j)}\right\|_{L^\infty(\Z^\dd)}\\
&\quad\le\left(1+\frac{\gamma_\dd}{\log \left(\frac{1}{8}M^{3j^2-3j+1}\right)}\right)\left\|\nabla_1^2\eta^{(j)}\right\|_{L^\infty(\Z^\dd)}\\
&\quad\le\left(1+\frac{1}{3j^2-3j+1}\right)\left\|\nabla_1^2\eta^{(j)}\right\|_{L^\infty(\Z^\dd)}
\end{aligned}
\end{equation}
where we have used \eqref{e:existence_cutoff3} and the fact that $\frac{\gamma_\dd}{\log\left(\frac{M}{8}\right)}\le\frac{\gamma_\dd}{\log\exp(\gamma_\dd)}=1$. We set
\[\eta^{(j-1)}=\eta^{(j)}+\sum_{Q\in \tilde Y^{(j)}_{K,M,V}(A)}w_Q\]
and
\[V^{(j-1)}=V^{(j)}+Q_{\ell_j}(0)\]
and finally we set $\eta=\eta^{(0)}$.

\emph{Step 3: Proof that the $\eta^{(j)}$ are locally affine on the bad cubes with $j_{\textrm{isol},Q}\ge j$}\\
We prove by reverse induction that $\supp \eta^{(j)}\subset V^{(j)}$, $\eta^{(j)}=\eta_*$ on $\Z^\dd\setminus V^{(j)}$ and $\nabla_1^2\eta^{(j)}=0$ on $\bigcup_{k\ge j}\bigcup_{q\in T_{K,M,V,k}(A)}(q+Q_1(0))$. This is obvious for $j=j_*(\e)$, so assume that it holds for some $j\in \{1,\ldots,j_*(\e)\}$.

We claim that then also $\nabla_1^2\eta^{(j-1)}=0$ on $\bigcup_{k\ge j-1}\bigcup_{q\in T_{K,M,V,k}(A)}(q+Q_1(0))$. To see this, let $q\in T_{K,M,V,k}(A)$ for $k=j_{\textrm{isol},q}\ge j-1$. We distinguish the two cases $j_{\textrm{isol},q}\ge j$ and $j_{\textrm{isol},q}=j-1$. 

In the former case by our inductive assumption already $\nabla_1^2\eta^{(j)}=0$ on $q+Q_1(0)$. Furthermore, we have argued that $q$ does not intersect $(Q+Q_{\ell_j/4+1}(0))\setminus (Q+Q_{\ell_{j-1}}(0))$ for any $Q\in \tilde Y^{(j)}_{K,M,V}(A)$. So either $q$ does not intersect $Q+Q_{\ell_j/4+1}(0)$ for any $Q\in \tilde Y^{(j)}_{K,M,V}(A)$, or $q$ is contained in $Q+Q_{\ell_{j-1}}(0)$ for exactly one $Q\in \tilde Y^{(j)}_{K,M,V}(A)$. In the former case, all $\nabla_1^2w_Q$ are equal to 0 on $q+Q_1(0)$, and thus $\nabla_1^2\eta^{(j-1)}=\nabla_1^2\eta^{(j)}=0$ on $q+Q_1(0)$, while in the latter case it holds that $\nabla_1^2\eta^{(j-1)}=\nabla_1^2(\eta^{(j)}+w_Q)$ on $q+Q_1(0)$, and thus by construction of $w_Q$ we have $\nabla_1^2\eta^{(j-1)}=0$ on $q+Q_1(0)$.

It remains to consider the case that $j_{\textrm{isol},q}=j-1$. In that case the sequence \eqref{e:badboxesondiffscales} contains a parent $q\subset q^{(j-1)}\in S^{(j-1)}_{K,M,\mathrm{bad}}(A)$. If $(q^{(j-1)}+Q_1(0))\cap V^{(j)}=\varnothing$, then $\nabla_1^2\eta^{(j)}=0$ on $q^{(j-1)}+Q_1(0)$. Furthermore, by the definition of $Y^{(j-1)}_{K,M,V}(A)$, $q^{(j-1)}$ does not intersect $(Q+Q_{\ell_j/4+1}(0))\setminus (Q+Q_{\ell_{j-1}}(0))$ for any $Q\in \tilde Y^{(j-1)}_{K,M,V}(A)\subset Y^{(j-1)}_{K,M,V}(A)$, and so neither does $q$. Arguing as in the previous case, we find that $\nabla_1^2\eta^{(j-1)}=\nabla_1^2\eta^{(j)}$ on $q+Q_1(0)$. On the other hand, it could be that $(q^{(j-1)}+Q_1(0))\cap V^{(j)}\neq\varnothing$. Then $q^{(j-1)}\in Y^{(j-1)}_{K,M,V}(A)$, and so there is some $Q\in\tilde Y^{(j-1)}_{K,M,V}(A)$ with $q\subset q^{(j-1)}\subset Q+Q_{\ell_{j-1}}(0)$. Then it holds once again that holds that $\nabla_1^2\eta^{(j-1)}=\nabla_1^2(\eta^{(j)}+w_Q)=0$ on $q+Q_1(0)$.

This proves that $\nabla_1^2\eta^{(j-1)}=0$ on $\bigcup_{k\ge j-1}\bigcup_{q\in T_{K,M,V,k}(A)}(q+Q_1(0))$. Furthermore, each $Q\in\tilde Y^{(j-1)}_{K,M,V}(A)$ is contained in $V^{(j)}+Q_{2(\ell_{j-1}+2)}(0)$, so the support of the associated $w_Q$ is contained in $V^{(j)}+Q_{2(\ell_{j-1}+2)+\ell_j/4-1}(0)$. Thus, $\nabla_1^2\eta^{(j-1)}$ is supported in $V^{(j)}+Q_{2(\ell_{j-1}+2)+\ell_j/4}(0)\subset V^{(j)}+Q_{\ell_j}(0)=V^{(j-1)}$, and $\eta^{(j)}=\eta_*$ on $\Z^\dd\setminus V^{(j-1)}$. This completes the induction.

\emph{Step 4: Proof that $\eta$ satisfies i), ii) and iii)}\\
We define $\eta=\eta^{(0)}$. In the previous step we have shown that $\nabla_1^2\eta$ is supported in $V^{(0)}$ and $\eta=\eta_*$ on $\Z^\dd\setminus V^{(0)}$. We have \begin{align*}
V^{(0)}&=V+Q_{\ell_0}(0)+\ldots+Q_{\ell_{j_*(\e)}}(0)\\
&\subset V+Q_{2\ell_{j_*(\e)}}(0)\subset V+Q_{KL\lmac/4}(0)\subset (U+Q_{7KL\lmac/8+1}(0))\setminus(U+Q_{KL\lmac/8-1}(0))\,.
\end{align*}
Thus, $\eta=\eta_*=0$ on $U+Q_{KL\lmac/8-1}(0)$, $\eta=\eta_*=1$ on $\Z^\dd\setminus \left(U+Q_{7KL\lmac/8+1}(0)\right)$. This means that $\eta$ satisfies i) and ii) as soon as $\e$ is small enough (depending on $\dd$ and $K$).

In Step 3 we have also seen that $\eta=\eta^{(0)}$ satisfies $\nabla_1^2\eta=0$ on $\bigcup_{k\ge 0}\bigcup_{q\in T_{K,M,V,k}(A)}(q+Q_1(0))=\bigcup_{q\in S^{(0)}_{K,M,\mathrm{bad}}(A)}(q+Q_1(0))$. Thus, to show that $\eta$ also satisfies iii) we only have to check that $\|\nabla_1^2\eta\|_{L^\infty(\Z^\dd)}\le \frac{C_\dd}{(KL)^2\lmac^2}$.

To see this, note the supports of the functions $\nabla_1^2w_Q$ for $Q\in \tilde Y^{(j-1)}_{K,M,V}(A)$  are disjoint. Thus, \eqref{e:existence_cutoff4} implies the bound
\[\left\|\nabla_1^2\eta^{(j-1)}\right\|_{L^\infty(\Z^\dd)}\le\left(1+\frac{1}{3j^2-3j+1}\right)\left\|\nabla_1^2\eta^{(j)}\right\|_{L^\infty(\Z^\dd)}
\]
for $j\ge1$. Iterating this, we find that
\begin{equation}\label{e:existence_cutoff5}
\begin{aligned}
\left\|\nabla_1^2\eta\right\|_{L^\infty(\Z^\dd)}&\le\left(\prod_{j=1}^{j_*(\e)}\left(1+\frac{1}{3j^2-3j+1}\right)\right)\left\|\nabla_1^2\eta_*\right\|_{L^\infty(\Z^\dd)}\\
&\le\left(\prod_{j=1}^{\infty}\left(1+\frac{1}{3j^2-3j+1}\right)\right)\frac{C_\dd}{(KL)^2\lmac^2}\\
&\le\frac{C_\dd}{(KL)^2\lmac^2}
\end{aligned}
\end{equation}
where in the last step we have used that $\frac{1}{3j^2-3j+1}$ is summable and hence $\prod_{j=1}^{\infty}\left(1+\frac{1}{3j^2-3j+1}\right)<\infty$. This completes the proof.
\end{proof}

\subsection{Decay estimates on good domains}\label{s:widman_good_box}

Now that we have a construction of a cut-off function at our disposal, we can execute the Widman hole filler argument.

\begin{lemma}\label{l:decay_one_annulus}
Let $\dd\ge1$ and let $M_\dd$ be the constant from Lemma \ref{l:existence_cutoff}. Let $K$, $L$, $M$ be odd integers such that $K$ is a multiple of 3 and $M\ge M_\dd$. Then for all $\e$ sufficiently small (depending on $\dd$, $K$) the following holds. Let $U\in\Pbox_{KL\lmac}$ be a polymer. Suppose that none of the $KL\lmac$-boxes touching $U$ (in the $l_\infty$-sense) are bad of type I or II, i.e. \[\{Q\in\Qbox_{KL\lmac}\colon Q\subset (U+Q_{KL\lmac}(0))\setminus U\}\cap S^*_{K,L,M,\mathrm{bad}}(A)=\varnothing\,.\]
Then the following holds.
\begin{itemize}
	\item[a)]  If $u\colon\Z^\dd\to\R$ is a function such that $u=0$ on $\tilde A\setminus U$ and $u\Delta_1^2u=0$ on $\Z^\dd\setminus U$, we have the estimate
\begin{equation}\label{e:decay_one_annulus_ext}
\hspace{-9pt}
\left\|\nabla_1^2u\right\|^2_{L^2(\Z^\dd\setminus(U+Q_{KL\lmac}(0)))}\le\left(\frac{C_\dd K^{\dd-4}\left(1+\I_{\dd=4}\log K\right)}{L^2}+\frac14\right)\left\|\nabla_1^2u\right\|^2_{L^2((U+Q_{KL\lmac}(0))\setminus U))}\,.
\end{equation}
	\item[b)]  If $u\colon\Z^\dd\to\R$ is a function such that $u=0$ on $\left(U+Q_{KL\lmac}(0)\right)\cap\tilde A$ and $u\Delta_1^2u=0$ on $U+Q_{KL\lmac}(0)$, we have the estimate
\begin{equation}\label{e:decay_one_annulus_int}
\left\|\nabla_1^2u\right\|^2_{L^2(U)}\le\left(\frac{C_\dd K^{\dd-4}\left(1+\I_{\dd=4}\log K\right)}{L^2}+\frac14\right)\left\|\nabla_1^2u\right\|^2_{L^2((U+Q_{KL\lmac}(0))\setminus U))}\,.
\end{equation}
\end{itemize}
\end{lemma}

\begin{proof}
The proof proceeds as outlined in Section \ref{s:main_ideas}. We begin with the proof of part a); the proof of part b) will be very similar.

\emph{Step 1: Discrete integration by parts}\\
We know that none of the cubes in $(U+Q_{KL\lmac}(0))\setminus U$ is bad of type I. Thus, we can apply Lemma \ref{l:existence_cutoff}. Let $\eta$ be the cut-off function that we obtain from that Lemma.

We now carry out the discrete analogue of the calculation that lead to \eqref{e:intbyparts3}. Namely we see that
\begin{equation}\label{e:decayl2annuli_1}
\begin{aligned}
0&=(\Delta_1^2u,\eta u)_{L^2(\Z^\dd)}\\
&=\sum_{i,j=1}^\dd\left(D^1_iD^1_{-j}u,D^1_iD^1_{-j}\left(\eta u\right)\right)_{L^2(\Z^\dd)}\\
&=\sum_{i,j=1}^\dd\left(D^1_iD^1_{-j}u,\left(uD^1_iD^1_{-j}\eta+D^1_iu\tau^1_iD^1_{-j}\eta+D^1_{-j}u\tau^1_{-j}D^1_i\eta+D^1_iD^1_{-j}u\tau^1_i\tau^1_{-j}\eta\right)\right)_{L^2(\Z^\dd)}\\
&=\sum_{i,j=1}^\dd\sum_{x\in\Z^\dd}\left|D^1_iD^1_{-j}u(x)\right|^2\tau^1_i\tau^1_{-j}\eta(x)+\sum_{i,j=1}^\dd\sum_{x\in\Z^\dd} u(x)D^1_iD^1_{-j}u(x)D^1_iD^1_{-j}\eta(x)\\
&\quad+\sum_{i,j=1}^\dd\sum_{x\in\Z^\dd} D^1_iD^1_{-j}u(x)D^1_iu(x)\tau^1_iD^1_{-j}\eta(x)+\sum_{i,j=1}^\dd\sum_{x\in\Z^\dd} D^1_iD^1_{-j}u(x)D^1_{-j}u(x)\tau^1_{-j}D^1_i\eta(x)\,.
\end{aligned}
\end{equation}
Consider the third term in this sum. We can apply summation by parts here and obtain
\begin{align*}
&\sum_{i,j=1}^\dd\sum_{x\in\Z^\dd} D^1_iD^1_{-j}u(x)D^1_iu(x)\tau^1_iD^1_{-j}\eta(x)\\
&\quad=-\sum_{i,j=1}^\dd\sum_{x\in\Z^\dd} D^1_iu(x)D^1_j\left(D^1_iu(x)\tau^1_iD^1_{-j}\eta(x)\right)\\
&\quad=-\sum_{i,j=1}^\dd\sum_{x\in\Z^\dd} D^1_iu(x)D^1_iu(x)\tau^1_iD^1_jD^1_{-j}\eta(x)-\sum_{i,j=1}^\dd\sum_{x\in\Z^\dd} D^1_iu(x)D^1_iD^1_ju(x)\tau^1_i\tau^1_jD^1_{-j}\eta(x)\\
&\quad=-\sum_{i,j=1}^\dd\sum_{x\in\Z^\dd} |D^1_iu(x)|^2\tau^1_iD^1_jD^1_{-j}\eta(x)-\sum_{i,j=1}^\dd\sum_{x\in\Z^\dd} D^1_iu(x)D^1_iD^1_{-j}u(x)\tau^1_iD^1_{-j}\eta(x)
\end{align*}
where in the last step we changed the index of summation from $x$ to $\tau^1_{-j}x$. This implies
\[\sum_{i,j=1}^\dd\sum_{x\in\Z^\dd} D^1_iD^1_{-j}u(x)D^1_iu(x)\tau^1_iD^1_{-j}\eta(x)=-\frac12\sum_{i,j=1}^\dd\sum_{x\in\Z^\dd} |D^1_iu(x)|^2\tau^1_iD^1_jD^1_{-j}\eta(x)\,.\]
Similarly, we find for the fourth summand in \eqref{e:decayl2annuli_1} that
\[\sum_{i,j=1}^\dd\sum_{x\in\Z^\dd} D^1_iD^1_{-j}u(x)D^1_{-j}u(x)\tau^1_{-j}D^1_i\eta(x)=-\frac12\sum_{i,j=1}^\dd\sum_{x\in\Z^\dd} |D^1_{-j}u(x)|^2\tau^1_{-j}D^1_iD^1_{-i}\eta(x)\,.\]
If we use the last two equalities in \eqref{e:decayl2annuli_1}, we arrive at
\begin{align*}
&\sum_{i,j=1}^\dd\sum_{x\in\Z^\dd}\left|D^1_iD^1_{-j}u(x)\right|^2\tau^1_i\tau^1_{-j}\eta(x)\\
&\quad=\frac12\sum_{i,j=1}^\dd\sum_{x\in\Z^\dd} |D^1_iu(x)|^2\tau^1_iD^1_jD^1_{-j}\eta(x)+\frac12\sum_{i,j=1}^\dd\sum_{x\in\Z^\dd} |D^1_{-j}u(x)|^2\tau^1_{-j}D^1_iD^1_{-i}\eta(x)\\
&\qquad-\sum_{i,j=1}^\dd\sum_{x\in\Z^\dd} u(x)D^1_iD^1_{-j}u(x)D^1_iD^1_{-j}\eta(x)\,.
\end{align*}
Here, in the second summand on the right-hand side, we can shift the summation from $x$ to $x+e_j$ and then interchange the indices $i$ and $j$ to see that
\begin{equation}\label{e:decay_one_annulus_1}
\begin{aligned}
&\sum_{i,j=1}^\dd\sum_{x\in\Z^\dd}\left|D^1_iD^1_{-j}u(x)\right|^2\tau^1_i\tau^1_{-j}\eta(x)\\
&\quad=\frac12\sum_{i,j=1}^\dd\sum_{x\in\Z^\dd} |D^1_iu(x)|^2\tau^1_iD^1_jD^1_{-j}\eta(x)+\frac12\sum_{i,j=1}^\dd\sum_{x\in\Z^\dd} |D^1_iu(x)|^2D^1_jD^1_{-j}\eta(x)\\
&\qquad-\sum_{i,j=1}^\dd\sum_{x\in\Z^\dd} u(x)D^1_iD^1_{-j}u(x)D^1_iD^1_{-j}\eta(x)\\
&\quad=\frac12\sum_{i,j=1}^\dd\sum_{x\in\Z^\dd} |D^1_iu(x)|^2\left(\tau^1_iD^1_jD^1_{-j}\eta(x)+D^1_jD^1_{-j}\eta(x)\right)-\sum_{i,j=1}^\dd\sum_{x\in\Z^\dd} u(x)D^1_iD^1_{-j}u(x)D^1_iD^1_{-j}\eta(x)\,.
\end{aligned}
\end{equation}
This is the discrete analogue of \eqref{e:intbyparts2}. To continue, we can use the Cauchy-Schwarz inequality on the second term on the right hand side of \eqref{e:decay_one_annulus_1} and obtain
\begin{equation}\label{e:decay_one_annulus_2}
\begin{aligned}
&\sum_{i,j=1}^\dd\sum_{x\in\Z^\dd}\left|D^1_iD^1_{-j}u(x)\right|^2\tau^1_i\tau^1_{-j}\eta(x)\\
&\quad\le\frac12\sum_{i,j=1}^\dd\sum_{x\in\Z^\dd} |D^1_iu(x)|^2\left(\tau^1_iD^1_jD^1_{-j}\eta(x)+D^1_jD^1_{-j}\eta(x)\right)+\sum_{i,j=1}^\dd\sum_{x\in\Z^\dd} |u(x)|^2\left|D^1_iD^1_{-j}\eta(x)\right|^2\\
&\qquad+\frac14\sum_{i,j=1}^\dd\sum_{x\in\Z^\dd} \left|D^1_iD^1_{-j}u(x)\right|^2\I_{\nabla_1^2\eta(x)\neq 0}
\end{aligned}
\end{equation}
which is the discrete analogue of \eqref{e:intbyparts3}. Next, we use the properties of $\eta$. First, recall that $\eta=1$ on $\Z^\dd\setminus(U+Q_{KL\lmac-2K\lmic}(0))$, and $\eta=0$ on $U+Q_{2K\lmic}(0)$. Therefore certainly $\tau^1_i\tau^1_{-j}\eta=1$ on $\Z^\dd\setminus(U+Q_{KL\lmac}(0))$, and $\nabla_1^2\eta=0$ on $\Z^\dd\setminus V$, where $V:=(U+Q_{KL\lmac-K\lmic}(0))\setminus(U+Q_{K\lmic}(0))$, which we can use to bound the left-hand side and the third term on the right-hand side of \eqref{e:decay_one_annulus_2}. We also know that 
\[|\nabla_1^2\eta(x)|\le\frac{C_\dd}{(KL)^2\lmac^2}\I_{Q_1(x)\cap \bigcup_{Q\in S^{(0)}_{K,M,\mathrm{bad}}(A)}Q=\varnothing}\]
which implies that
\[|\tau^1_{\pm k}D^1_iD^1_{-j}\eta(x)|\le\frac{C_\dd}{(KL)^2\lmac^2}\I_{x\notin\bigcup_{Q\in S^{(0)}_{K,M,\mathrm{bad}}(A)}Q}\]
for any $i,j,k\in\{1,\ldots,\dd\}$. We can use this for the first and second term on the right-hand side of \eqref{e:decay_one_annulus_2}. Putting everything together, we see that
\begin{equation}\label{e:decay_one_annulus_3}
\begin{aligned}
&\left\|\nabla_1^2u\right\|^2_{L^2(\Z^\dd\setminus(U+Q_{KL\lmac}(0)))}\\
&\quad=\sum_{i,j=1}^\dd\sum_{x\in\Z^\dd\setminus(U+Q_{KL\lmac}(0))}\left|D^1_iD^1_{-j}u(x)\right|^2\\
&\quad\le\frac{C_\dd}{(KL)^2\lmac^2}\sum_{i=1}^\dd\sum_{x\in V} |D^1_iu(x)|^2\I_{x\notin\bigcup_{Q\in S^{(0)}_{K,M,\mathrm{bad}}(A)}Q}\\
&\qquad+\frac{C_\dd}{(KL)^4\lmac^4}\sum_{x\in V} |u(x)|^2\I_{x\notin\bigcup_{Q\in S^{(0)}_{K,M,\mathrm{bad}}(A)}Q}+\frac14\sum_{i,j=1}^\dd\sum_{x\in V} \left|D^1_iD^1_{-j}u(x)\right|^2\\
&\quad=\frac{C_\dd}{(KL)^2\lmac^2}\left\|\nabla_1u\I_{\cdot\notin\bigcup_{Q\in S^{(0)}_{K,M,\mathrm{bad}}(A)}Q}\right\|^2_{L^2(V)}+\frac{C_\dd}{(KL)^4\lmac^4}\left\|u\I_{\cdot\notin\bigcup_{Q\in S^{(0)}_{K,M,\mathrm{bad}}(A)}Q}\right\|^2_{L^2(V)}\\
&\qquad+\frac14\left\|\nabla_1^2u\right\|^2_{L^2(V)}\,.
\end{aligned}
\end{equation}

\emph{Step 2: Use of Poincaré and interpolation inequalities}\\
We continue by estimating the first term on the right hand side of \eqref{e:decay_one_annulus_3}. To do so, we want to apply Lemma \ref{l:interpolationineq} on each of the $K\lmac$-boxes in $(U+Q_{KL\lmac}(0))\setminus U$. Note that $(U+Q_{KL\lmac}(0))\setminus U\in\Pbox_{KL\lmac}$, i.e. it is the disjoint union of some cubes in $\Qbox_{KL\lmac}$. Let $Q'$ be one such cube. It is the disjoint union of $L^\dd$ cubes in $\Qbox_{K\lmac}$. Let $q$ be one of them, and let $B_q:=q\cap V\setminus\bigcup_{Q\in S^{(0)}_{K,M,\mathrm{bad}}(A)}Q$. We claim that $|B_q|\ge\frac12|q|=\frac12 K^\dd\lmac^\dd$. To see this, note that 
\[|q\setminus V|\le2\dd K^{\dd-1}\lmac^{\dd-1}K\lmic=2\dd K^\dd\lmac^{\dd-1}\lmic\,.\]
Furthermore, by assumption $Q'$ is not bad of type II. Therefore, each of its $K\lmac$-subcubes contains at most $\frac14\left(\frac{\lmac}{\lmic}\right)^\dd$ cubes in $S^{(0)}_{K,M,\mathrm{bad}}(A)$, i.e. 
\[\left|q\cap\bigcup_{Q\in S^{(0)}_{K,M,\mathrm{bad}}(A)}Q\right|\le\frac14\left(\frac{\lmac}{\lmic}\right)^\dd K^\dd\lmic^\dd=\frac14K^\dd\lmac^\dd=\frac14|q|\,.\]
 Therefore,
\[|B_{q}|\ge|q|-|q\setminus V|-\left|q\cap\bigcup_{Q\in S^{(0)}_{K,M,\mathrm{bad}}(A)}Q\right|\ge|q|-\frac14|q|-2\dd K^\dd\lmac^{\dd-1}\lmic\ge\frac12|q|\]
whenever $\e$ is small enough (in terms of $\dd$, $K$) so that $2K^\dd\lmic\le\frac14\lmac$. Thus, we can apply Lemma \ref{l:interpolationineq} on $q$ with $B=B_q$ and obtain

\begin{align*}
&\left\|\nabla_1u\I_{\cdot\notin\bigcup_{Q\in S^{(0)}_{K,M,\mathrm{bad}}(A)}Q}\right\|^2_{L^2(V\cap q)}\\
&\quad=\left\|\nabla_1u\I_{\cdot\in V\setminus\bigcup_{Q\in S^{(0)}_{K,M,\mathrm{bad}}(A)}Q}\right\|^2_{L^2(q)}\\
&\quad\le C_\dd K^2\lmac^2\left\|\nabla_1^2u\right\|^2_{L^2(V\cap q)}+\frac{C_\dd}{K^2\lmac^2}\left\|u\I_{\cdot\in V\setminus\bigcup_{Q\in S^{(0)}_{K,M,\mathrm{bad}}(A)}Q}\right\|^2_{L^2(q)}\,,
\end{align*}
and summing this over all $Q'$ and all $q\subset Q'$ we see that
\begin{align*}
&\left\|\nabla_1u\I_{\cdot\notin\bigcup_{Q\in S^{(0)}_{K,M,\mathrm{bad}}(A)}Q}\right\|^2_{L^2(V)}\\
&\quad\le C_\dd K^2\lmac^2\left\|\nabla_1^2u\right\|^2_{L^2(V)}+\frac{C_\dd}{K^2\lmac^2}\left\|u\I_{\cdot\notin\bigcup_{Q\in S^{(0)}_{K,M,\mathrm{bad}}(A)}Q}\right\|^2_{L^2(V)}\,.
\end{align*}
Using this estimate in \eqref{e:decay_one_annulus_3} we arrive at
\begin{equation}\label{e:decay_one_annulus_4}
\begin{aligned}
&\left\|\nabla_1^2u\right\|^2_{L^2(\Z^\dd\setminus(U+Q_{KL\lmac}(0)))}\\
&\quad\le\frac{C_\dd}{L^2}\left\|\nabla_1^2u\right\|^2_{L^2(V)}+\frac{C_\dd}{K^4L^2\lmac^4}\left\|u\I_{\cdot\notin\bigcup_{Q\in S^{(0)}_{K,M,\mathrm{bad}}(A)}Q}\right\|^2_{L^2(V)}\\
&\qquad+\frac{C_\dd}{(KL)^4\lmac^4}\left\|u\I_{\cdot\notin\bigcup_{Q\in S^{(0)}_{K,M,\mathrm{bad}}(A)}Q}\right\|^2_{L^2(V)}+\frac14\left\|\nabla_1^2u\right\|^2_{L^2(V)}\\
&\quad\le\left(\frac{C_\dd}{L^2}+\frac14\right)\left\|\nabla_1^2u\right\|^2_{L^2(V)}+\frac{C_\dd}{K^4L^2\lmac^4}\left\|u\I_{\cdot\notin\bigcup_{Q\in S^{(0)}_{K,M,\mathrm{bad}}(A)}Q}\right\|^2_{L^2(V)}
\end{aligned}
\end{equation}
where we used $L\ge1$ in the last step. 

Theorem \ref{t:localpoincare} with $R=K\lmic$ allows to bound

\begin{equation}\label{e:decay_one_annulus_5}
\begin{aligned}
\left\|u\I_{\cdot\notin\bigcup_{Q\in S^{(0)}_{K,M,\mathrm{bad}}(A)}Q}\right\|^2_{L^2(V)}&\le C_\dd K^\dd\lmic^\dd\left(1+\I_{\dd=4}\log(K\lmic)\right)\left\|\nabla_1^2u\right\|^2_{L^2(V+Q_{K\lmic}(0))}\\
&\le C_\dd K^\dd\lmic^\dd\left(1+\I_{\dd=4}\log(K\lmic)\right)\left\|\nabla_1^2u\right\|^2_{L^2((U+Q_{KL\lmac}(0))\setminus U)}\,.
\end{aligned}
\end{equation}
Now one easily checks that for any $\dd\ge4$ we have $\lmic^\dd\left(1+\I_{\dd=4}\log\lmic\right)\le C_\dd\lmac^4$. Thus, combining \eqref{e:decay_one_annulus_4} and \eqref{e:decay_one_annulus_5} we see that
\begin{align*}
&\left\|\nabla_1^2u\right\|^2_{L^2(\Z^\dd\setminus(U+Q_{KL\lmac}(0)))}\\
&\le\left(\frac{C_\dd}{L^2}+\frac14\right)\left\|\nabla_1^2u\right\|^2_{L^2((U+Q_{KL\lmac}(0))\setminus U)}+\frac{C_\dd K^{\dd-4}\left(1+\I_{\dd=4}\log K\right)}{L^2}\left\|\nabla_1^2u\right\|^2_{L^2((U+Q_{KL\lmac}(0))\setminus U)}\\
&\le\left(\frac{C_\dd K^{\dd-4}\left(1+\I_{\dd=4}\log K\right)}{L^2}+\frac14\right)\left\|\nabla_1^2u\right\|^2_{L^2((U+Q_{KL\lmac}(0))\setminus U)}
\end{align*}
which is \eqref{e:decay_one_annulus_int}.

\emph{Step 3: Proof of part b)}\\
We proceed completely analogously as in Step 1 and 2. The only difference is that we work with $\hat\eta:=1-\eta$ instead of $\eta$. The assumptions for $\nabla_1^2\eta$ carry over to $\nabla_1^2\hat\eta$, and so the proof carries over.

\end{proof}
Lemma \ref{l:decay_one_annulus} has the following straightforward corollary:
\begin{lemma}\label{l:decay_many_annuli}
Let $\dd\ge1$ and let $M_\dd$ be the constant from Lemma \ref{l:existence_cutoff}. Let $K$ be an odd multiple of 3 and $M\ge M_\dd$. Then there is a constant $L_{\dd,K}$ depending on $\dd$ and $K$ only such that for any odd integers $L\ge L_{\dd,K}$, $M\ge M_\dd$ and for all $\e$ sufficiently small (depending on $\dd$, $K$) the following holds: Let $U_0,\ldots,U_k\in\Pbox_{KL\lmac}$ be polymers. Suppose that for each $j\in\{0,\ldots,k-1\}$ we have $U_j+Q_{KL\lmac}(0)\subset U_{j+1}$ and \[\{Q\in\Qbox_{KL\lmac}\colon Q\subset (U_j+Q_{KL\lmac}(0))\setminus U\}\cap S^*_{K,L,M,\mathrm{bad}}(A)=\varnothing\,.\]
Then the following holds:
\begin{itemize}
	\item[a)]  If $u\colon\Z^\dd\to\R$ is a function such that $u=0$ on $\tilde \A\setminus U_0$ and $u\Delta_1^2u=0$ on $\Z^\dd\setminus U_0$, we have the estimate
\begin{equation}\label{e:decay_many_annuli_ext}
\left\|\nabla_1^2u\right\|^2_{L^2(\Z^\dd\setminus U_k)}\le\frac{1}{2^k}\left\|\nabla_1^2u\right\|^2_{L^2((U_0+Q_{KL\lmac}(0))\setminus U_0)}\,.
\end{equation}
	\item[b)]  If $u\colon\Z^\dd\to\R$ is a function such that $u=0$ on $U_k\cap\tilde \A$ and $u\Delta_1^2u=0$ on $U_k$, we have the estimate
\begin{equation}\label{e:decay_many_annuli_int}
\left\|\nabla_1^2u\right\|^2_{L^2(U_0)}\le\frac{1}{2^k}\left\|\nabla_1^2u\right\|^2_{L^2((U_{k-1}+Q_{KL\lmac}(0))\setminus U_{k-1})}\,.
\end{equation}
\end{itemize}

\end{lemma}

\begin{proof}
We choose $L$ large enough so that the prefactors on the right-hand side in \eqref{e:decay_one_annulus_ext} and \eqref{e:decay_one_annulus_int} become less than $\frac12$, and then apply Lemma \ref{l:decay_one_annulus} iteratively on each $U_j$.
\end{proof}

\subsection{Sparsity of bad boxes}\label{s:est_prob_cutoff}
In order to conclude Theorem \ref{t:decay_high_prob} from Lemma \ref{l:decay_many_annuli} it remains to show that with sufficiently high probability we can find sets $U_j$ as in that Lemma. For that purpose we need to show that bad cubes are sparse enough.

In fact, we will show that we can make the probability of a cube in $Q\in\Qbox_{KL\lmac}$ being bad arbitrarily small. We even have a slightly stronger result, namely that for each finite $T^*\subset \Qbox_{KL\lmac}$ we can control the probability that all cubes in $T^*$ are bad.
\begin{lemma}\label{l:badmacrocube}
Let $\dd\ge4$, let $p>0$ be arbitrary. Let $M\ge12$ be an odd integer. Then there is $K_{\dd,M,p}$ depending on $\dd$, $M$ and $p$ only with the following property: let $K\ge K_{\dd,M,p}$ be an odd multiple of 3, let $L$ be an odd integer, let $\e$ be small enough (depending on $\dd$, $L$, $M$ and $p$), and let $T^*$ be an arbitrary finite subset of $\Qbox_{KL\lmac}$. Then
\begin{align}
\zeta^\e_\Lambda\left(T^*\subset S^{*,I}_{K,L,M,\mathrm{bad}}(\A)\right)&\le p^{|T^*|}\label{e:badmacrocube_typeI}\,,\\
\zeta^\e_\Lambda\left(T^*\subset S^{*,II}_{K,L,M,\mathrm{bad}}(\A)\right)&\le p^{|T^*|}\label{e:badmacrocube_typeII}\,,\\
\zeta^\e_\Lambda\left(T^*\subset S^{*}_{K,L,M,\mathrm{bad}}(\A)\right)&\le 2(4p)^{\frac{|T^*|}{2}}\label{e:badmacrocube}\,.
\end{align}
\end{lemma}
In order to prove this Lemma, we will have go to into the definition of the bad cubes of type I and II, and, in particular, we will have to understand how rare cubes in $S^{(j)}_{K,M,\mathrm{bad}}(\A)$ are. This is quantified in the following Lemma.

\begin{lemma}\label{l:badmicrocube}
Let $\dd\ge4$. There is a constant $K_\dd'$ with the following property: Let $\e$ be small enough (depending on $\dd$ only) so that the conclusion of Theorem \ref{t:estimatespinnedset} holds. Let $M\ge12$ be an odd integer and assume that $K\ge K_\dd'$ is an odd multiple of 3. Let $j\ge0$. Then we have the following estimates.
\begin{itemize}
\item[a)] If $j=0$ and $T^{(0)}\subset\Qbox_{\ell_0}$ is a finite subset, then
\begin{equation}\label{e:badmicrocube_j=0}
\zeta^\e_\Lambda\left(T^{(0)}\subset S^{(0)}_{K,M,\mathrm{bad}}(\A)\right)\le\left(K^{\frac{\dd}{2^\dd}}(\dd+1)^{\frac{1}{2^\dd}}\exp\left(-\frac{K^\dd}{C_\dd}\right)\right)^{|T^{(0)}|}\,.
\end{equation}
\item[b)] If $j>0$ and $T^{(j)}\subset\Qbox^\#_{\ell_j}$ is a finite subset such that the elements of $T^{(j)}$ are pairwise disjoint, then
\begin{equation}\label{e:badmicrocube_j>0}
\zeta^\e_\Lambda\left(T^{(j)}\subset S^{(j)}_{K,M,\mathrm{bad}}(\A)\right)\le\left(3^{(2^j-2)\dd}M^{s_j\dd}\left(K^{\frac{\dd}{2^\dd}}(\dd+1)^{\frac{1}{2^\dd}}\exp\left(-\frac{K^\dd}{C_\dd}\right)\right)^{2^j}\right)^{|T^{(j)}|}
\end{equation}
where $s_j:=j^3+\sum_{m=0}^jm^32^{j-m}$.
\end{itemize}
\end{lemma}
Here part a) is rather easy to show. If $Q\in S^{(0)}_{K,M,\mathrm{bad}}(\A)$, then there are too few pinned points around $Q$, and the probability for that can be estimated using Theorem \ref{t:estimatesfield}. The crucial point is that by choosing $K$ large this probability can be made arbitrarily small.

Part b) then follows by induction. Each cube in $S^{(j)}_{K,M,\mathrm{bad}}(\A)$ contains at least two cubes in $S^{(j-1)}_{K,M,\mathrm{bad}}(\A)$, and so if the latter cubes are rare, an union bound will show that the former cubes will be rare as well.

\begin{proof} We show first part a), then part b) for $j=1$, and then use that result to start an induction that will yield part b) for $j>1$ as well.

\emph{Step 1: Proof of part a)}\\
This is similar to the proof of Lemma \ref{l:localpoincare_prob}. However, we want a uniform estimate over the cubes in $T^{(0)}$, and so we need to be more careful.

Let $Q\in T^{(0)}$, and let $q\in\Qmic$ be such that $q\subset Q$. Suppose that $q$ has centre $x\in\Z^\dd$. For $i\in\{1,\ldots,\dd+1\}$ consider the sets $\Xi_i(q)=\bigcap_{y\in q}(y+\Xi_i)$ and $\Xi_{i,K\lmic/2}(q)=\Xi_i(q)\cap Q_{K\lmic/2}(x)$.
\begin{figure}[ht]
\centering
\begin{tikzpicture}
\path[use as bounding box] (0,0) rectangle (7,3.5);
\clip (0,0) rectangle (7,3.5);
\draw (0,0) rectangle (1,1);
\coordinate (A) at (0,1);
\coordinate (B) at (1,0);
\draw[dashed] (0,0)--++(10:10);
\draw[dashed] (0,0)--++(30:10);
\draw[dashed] (0,1)--++(10:10) node (C) {};
\draw[dashed] (0,1)--++(30:10);
\draw[dashed] (1,0)--++(10:10);
\draw[dashed] (1,0)--++(30:10) node (D) {};
\draw[dashed] (1,1)--++(10:10);
\draw[dashed] (1,1)--++(30:10);
\draw[name path=A--C,draw=none] (A)--(C.center);
\draw[name path=B--D,draw=none] (B)--(D.center);
\path[name intersections={of=A--C and B--D,by=E}];
\draw[pattern=north east lines,draw=none] (C.center)--(D.center)--(E)--cycle;
\node at (0.4,0.5) {$q$};
\node at (6.6,1.9) {$\Xi_i(q)$};
\end{tikzpicture}
\caption{A set $\Xi_i(q)$, given as the intersection of $x+\Xi_i$ for $x\in q$.}\label{f:Xi_i(q)}
\end{figure}
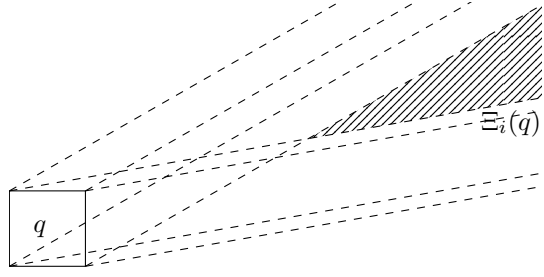

 The set $\Xi_i(q)$ is an intersection of translates of the same cone, where the tips of the cone range over the set $q$ of diameter $\le\sqrt{\dd}\lmic$ (cf. Figure \ref{f:Xi_i(q)}). As soon as $K\ge K_\dd'$ for some dimensional constant $K_\dd'$ the fraction of points in $Q_{K\lmic/2}(x)$ that are in $\Xi_{i,K\lmic/2}(q)$ is bounded below. We fix such a $K_\dd'$. In other words, we have the estimate
\begin{equation}\label{e:badmicrocube_1}
|\Xi_{i,K\lmic/2}(q)|\ge\frac{1}{C_\dd}(K\lmic)^\dd
\end{equation}
for $K\ge K_\dd'$. Furthermore, if $\Xi_{i,K\lmic/2}(q)\cap\tilde \A\neq\varnothing$ then $d^{(i)}(y,\tilde \A)\le \frac K2\lmic+\frac12\lmic\le K\lmic$ for all $y\in q$.

The preceding discussion implies that if $\Xi_{i,K\lmic/2}(q)\cap\tilde \A\neq\varnothing$ for all $i$ then $d_*(y,\A)\le K\lmic$ for all $y\in q$. Thus, if $\Xi_{i,K\lmic/2}(q)\cap\tilde \A\neq\varnothing$ holds for all $i$ and all $q\subset Q$ then $Q\notin S^{(0)}_{K,M,\mathrm{bad}}(\A)$. Using this we can write
\begin{align*}
&\zeta^\e_\Lambda\left(T^{(0)}\subset S^{(0)}_{K,M,\mathrm{bad}}(\A)\right)\\
&\quad=\zeta^\e_\Lambda\left(\forall Q\in T^{(0)}\ \exists q_Q\in\Qmic\ \exists i_Q\in\{1,\ldots,\dd+1\}\text{ with }q_Q\subset Q,\Xi_{i_Q,K\lmic/2}(q_Q)\cap\tilde\A=\varnothing\right)\\
&\quad=\zeta^\e_\Lambda\bigg(\forall Q\in T^{(0)}\ \exists q_Q\in\Qmic\ \exists i_Q\in\{1,\ldots,\dd+1\}\text{ with }q_Q\subset Q\\
&\quad\quad\quad\quad\quad\quad\text{ such that }\bigcup_{Q\in T^{(0)}}\Xi_{i_Q,K\lmic/2}(q_Q)\cap\tilde\A=\varnothing\bigg)\,.
\end{align*}
We can estimate this probability by summing over all choices $\underline q=\left(q_Q\right)_{Q\in T^{(0)}}\in(\Qmic)^{T^{(0)}}$ and $\underline{i}=\left(i_Q\right)_{Q\in T^{(0)}}\in\{1,\ldots,\dd+1\}^{T^{(0)}}$ to find
\begin{align*}
\zeta^\e_\Lambda\left(T^{(0)}\subset S^{(0)}_{K,M,\mathrm{bad}}(\A)\right)&\le\sum_{\substack{\underline{q}\in(\Qmic)^{T^{(0)}}\\q_Q\subset Q\ \forall Q\in T^{(0)}}}\sum_{\underline{i}\in\{1,\ldots,\dd+1\}^{T^{(0)}}}\zeta^\e_\Lambda\left(\bigcup_{Q\in T^{(0)}}\Xi_{i_Q,K\lmic/2}(q_Q)\cap\tilde\A=\varnothing\right)\,.
\end{align*}
Assume for the moment that the elements of $T^{(0)}$ are well-separated in the sense that for any $Q,Q'\in T^{(0)}$ with $Q\neq Q'$ we have $d_\infty(Q,Q')\ge K\lmic$. Because $\Xi_{i_Q,K\lmic/2}(q_Q)$ is a subset of $q_Q+Q_{K\lmic/2}(0)$, it is a subset of the cube with sidelength $2K\lmic$ concentric to $Q$. In particular, by our temporary assumption, if $Q\neq Q'$, then $\Xi_{i_Q,K\lmic/2}(q_Q)$ and $\Xi_{i_{Q'},K\lmic/2}(q_{Q'})$ are disjoint. Thus, \eqref{e:badmicrocube_1} implies
\[\left|\bigcup_{Q\in T^{(0)}}\Xi_{i_Q,K\lmic/2}(q_Q)\right|\ge\frac{|T^{(0)}|}{C_\dd}(K\lmic)^\dd\,.\]
Using Theorem \ref{t:estimatespinnedset} we now see
\begin{align*}
\zeta^\e_\Lambda\left(T^{(0)}\subset S^{(0)}_{K,M,\mathrm{bad}}(\A)\right)&\le\sum_{\substack{\underline{q}\in(\Qmic)^{T^{(0)}}\\q_Q\subset Q\ \forall Q\in T^{(0)}}}\sum_{\underline{i}\in\{1,\ldots,\dd+1\}^{T^{(0)}}}(1-p_{\dd,-})^{\left|\bigcup_{Q\in T^{(0)}}\Xi_{i_Q,K\lmic/2}(q_Q)\right|}\\
&\le\sum_{\substack{\underline{q}\in(\Qmic)^{T^{(0)}}\\q_Q\subset Q\ \forall Q\in T^{(0)}}}\sum_{\underline{i}\in\{1,\ldots,\dd+1\}^{T^{(0)}}}\exp\left(-p_{\dd,-}\frac{|T^{(0)}|(K\lmic)^\dd}{C_\dd}\right)\\
&\le\sum_{\substack{\underline{q}\in(\Qmic)^{T^{(0)}}\\q_Q\subset Q\ \forall Q\in T^{(0)}}}\sum_{\underline{i}\in\{1,\ldots,\dd+1\}^{T^{(0)}}}\exp\left(-p_{\dd,-}\frac{|T^{(0)}|(K\lmic)^\dd}{C_\dd}\right)\,.
\end{align*}
In any dimension $\dd\ge4$ we have $p_{\dd,-}\lmic^\dd\ge\frac{1}{C_\dd}$. We also have $(K^\dd)^{|T^{(0)}|}$ choices for $\underline{q}$, and $(\dd+1)^{|T^{(0)}|}$ choices for $\underline{i}$, and so
\begin{equation}\label{e:badmicrocube_2}
\begin{aligned}
\zeta^\e_\Lambda\left(T^{(0)}\subset S^{(0)}_{K,M,\mathrm{bad}}(\A)\right)&\le(K^\dd)^{|T^{(0)}|}(\dd+1)^{|T^{(0)}|}\exp\left(-\frac{|T^{(0)}|K^\dd}{C_\dd}\right)\,.
\end{aligned}
\end{equation}

This estimate was derived under the assumption that $T^{(0)}$ is such that for any $Q,Q'\in T^{(0)}$ with $Q\neq Q'$ we have $d_\infty(Q,Q')\ge K\lmic$. In general, this will not be the case. However, we can partition $T^{(0)}$ into $2^\dd$ subsets $T^{(0)}_i$ for $i\in\{1,\ldots,2^\dd\}$ such that for any $i$ and any $Q,Q'\in T^{(0)}_i$ with $Q\neq Q'$ we have $d_\infty(Q,Q')\ge K\lmic$. At least one of these subsets, say $T^{(0)}_{i_*}$, will contain at least $\frac{|T^{(0)}|}{2^\dd}$ boxes. Then we can apply the estimate \eqref{e:badmicrocube_2} to $T^{(0)}_{i_*}$ and obtain
\begin{align*}
\zeta^\e_\Lambda\left(T^{(0)}\subset S^{(0)}_{K,M,\mathrm{bad}}(\A)\right)&\le\zeta^\e_\Lambda\left(T^{(0)}_{i_*}\subset S^{(0)}_{K,M,\mathrm{bad}}(\A)\right)\\
&\le(K^\dd)^{\frac{|T^{(0)}|}{2^\dd}}(\dd+1)^{\frac{|T^{(0)}|}{2^\dd}}\exp\left(-\frac{|T^{(0)}|K^\dd}{2^\dd C_\dd}\right)\\
&\le\left(K^{\frac{\dd}{2^\dd}}(\dd+1)^{\frac{1}{2^\dd}}\exp\left(-\frac{K^\dd}{C_\dd}\right)\right)^{|T^{(0)}|}
\end{align*}
which is \eqref{e:badmicrocube_j=0}.

\emph{Step 2: Proof of part b) for $j=1$}\\
We want to prove \eqref{e:badmicrocube_j>0} by induction on $j$. In principle, we would want to use \eqref{e:badmicrocube_j=0} as the base case. However, that statement is for $\Qbox_{\ell_0}$ instead of $\Qbox^\#_{\ell_0}$, and so we first derive \eqref{e:badmicrocube_j>0} for $j=1$ from \eqref{e:badmicrocube_j=0} and then use this assertion to start our induction.

Let $Q\in T^{(1)}$. By construction $Q\in S^{(1)}_{K,M,\mathrm{bad}}(\A)$ if and only if there are at least two disjoint cubes $q,q'\in S^{(0)}_{K,M,\mathrm{bad},\mathrm{clust}}(\A)\subset S^{(0)}_{K,M,\mathrm{bad}}(\A)$ such that $q,q'\subset Q$, and so
\begin{align*}
&\zeta^\e_\Lambda\left(T^{(1)}\subset S^{(1)}_{K,M,\mathrm{bad}}(\A)\right)\\
&\quad=\zeta^\e_\Lambda\left(\forall Q\in T^{(1)}\ \exists q_Q,q'_Q\in\Qbox_{\ell_0}\text{ with }q_Q\cap q_Q'=\varnothing,q_Q\cup q_Q'\subset Q,\{q_Q,q_Q'\}\subset S^{(0)}_{K,M,\mathrm{bad}}(\A)\right)\\
&\quad=\zeta^\e_\Lambda\Bigg(\forall Q\in T^{(1)}\ \exists q_Q,q'_Q\in\Qbox_{\ell_0}\text{ with }q_Q\cap q_Q'=\varnothing,q_Q\cup q_Q'\subset Q\\
&\qquad\qquad\qquad\text{ such that }\bigcup_{Q\in T^{(1)}}\{q_Q,q_Q'\}\subset S^{(0)}_{K,M,\mathrm{bad}}(\A)\Bigg)\,.
\end{align*}
As in Step 1 we can estimate this expression by the sum over all possibilities for $q_Q,q'_Q$ to find
\[\zeta^\e_\Lambda\left(T^{(1)}\subset S^{(1)}_{K,M,\mathrm{bad}}(\A)\right)\le\sum_{\substack{\underline{q},\underline{q'}\in(\Qbox_{\ell_0})^{T^{(1)}}\\q_Q\cap q_Q'=\varnothing,q_Q\cup q_Q'\subset Q\ \forall Q\in T^{(1)}}}\zeta^\e_\Lambda\left(\bigcup_{Q\in T^{(1)}}\{q_Q,q_Q'\}\subset S^{(0)}_{K,M,\mathrm{bad}}(\A)\right)\,.
\]
By assumption the elements of $T^{(1)}$ are pairwise disjoint. Therefore, all $q_Q$ and $q_Q'$ are pairwise distinct. Hence, the set $\bigcup_{Q\in T^{(1)}}\{q_Q,q_Q'\}$ has cardinality $2|T^{(1)}|$, and so \eqref{e:badmicrocube_j=0} implies
\begin{align*}
\zeta^\e_\Lambda\left(T^{(1)}\subset S^{(1)}_{K,M,\mathrm{bad}}(\A)\right)&\le\sum_{\substack{\underline{q},\underline{q'}\in(\Qbox_{\ell_0})^{T^{(1)}}\\q_Q\cap q_Q'=\varnothing,q_Q\cup q_Q'\subset Q\ \forall Q\in T^{(1)}}}\left(K^{\frac{\dd}{2^\dd}}(\dd+1)^{\frac{1}{2^\dd}}\exp\left(-\frac{K^\dd}{C_\dd}\right)\right)^{2|T^{(1)}|}\\
&\le\left(\left(\frac{\ell_1}{\ell_0}\right)^{2\dd}\right)^{|T^{(1)}|}\left(K^{\frac{\dd}{2^\dd}}(\dd+1)^{\frac{1}{2^\dd}}\exp\left(-\frac{K^\dd}{C_\dd}\right)\right)^{2|T^{(1)}|}\\
&=M^{2\dd|T^{(1)}|}\left(K^{\frac{\dd}{2^\dd}}(\dd+1)^{\frac{1}{2^\dd}}\exp\left(-\frac{K^\dd}{C_\dd}\right)\right)^{2|T^{(1)}|}\\
&=\left(M^{2\dd}\left(K^{\frac{\dd}{2^\dd}}(\dd+1)^{\frac{1}{2^\dd}}\exp\left(-\frac{K^\dd}{C_\dd}\right)\right)^2\right)^{|T^{(1)}|}
\end{align*}
which is \eqref{e:badmicrocube_j>0} (as $s_1=2$).

\emph{Step 3: Proof of part b) for $j>1$}\\
We proceed by induction on $j$, using the result from Step 2 as the base case. That is, we assume that \eqref{e:badmicrocube_j>0} holds for $j-1$, and we want to conclude that it also holds for $j$. The argument for this is analogous to the previous step. The only difference is that now the smaller cubes $q,q'$ live in $\Qbox^\#_{\ell_{j-1}}$ instead of $\Qbox_{\ell_{j-1}}$, and so the number of possible $\underline{q},\underline{q'}$ is now larger. Arguing as in Step 2, we obtain, using the assumption that the elements of $|T^{(j)}|$ are pairwise disjoint, that
\begin{equation}\label{e:badmicrocube_3}
\begin{aligned}
&\zeta^\e_\Lambda\left(T^{(j)}\subset S^{(j)}_{K,M,\mathrm{bad}}(\A)\right)\\
&\quad\le\sum_{\substack{\underline{q},\underline{q'}\in(\Qbox^\#_{\ell_{j-1}})^{T^{(j)}}\\q_Q\cap q_Q'=\varnothing,q_Q\cup q_Q'\subset Q\ \forall Q\in T^{(j)}}}\zeta^\e_\Lambda\left(\bigcup_{Q\in T^{(j)}}\{q_Q,q_Q'\}\subset S^{(j-1)}_{K,M,\mathrm{bad}}(\A)\right)\\
&\quad\le\left(\left(\frac{3\ell_j}{\ell_{j-1}}\right)^{2\dd}\right)^{|T^{(j)}|}\left(3^{(2^{j-1}-2)\dd}M^{s_{j-1}\dd}\left(K^{\frac{\dd}{2^\dd}}(\dd+1)^{\frac{1}{2^\dd}}\exp\left(-\frac{K^\dd}{C_\dd}\right)\right)^{2^{j-1}}\right)^{2|T^{(j)}|}\\
&\quad=\left(3^{2\dd+2(2^{j-1}-2)\dd}M^{2(j^3-(j-1)^3)\dd+2s_{j-1}\dd}\left(K^{\frac{\dd}{2^\dd}}(\dd+1)^{\frac{1}{2^\dd}}\exp\left(-\frac{M^\dd}{C_\dd}\right)\right)^{2^j}\right)^{|T^{(j)}|}\,.
\end{aligned}
\end{equation}
It remains to observe that $2\dd+2(2^{j-1}-2)\dd=(2^j-2)\dd$ and
\begin{align*}
2(j^3-(j-1)^3)\dd+2s_{j-1}\dd&=2\dd\left(j^3-(j-1)^3+(j-1)^3+\sum_{m=0}^{j-1}m^32^{j-1-m}\right)\\
&=\dd\left(2j^3+\sum_{m=0}^{j-1}m^32^{j-m}\right)\\
&=\dd\left(j^3+\sum_{m=0}^jm^32^{j-m}\right)=s_j\dd\,.
\end{align*}

\end{proof}

Now we can turn to the proof of Lemma \ref{l:badmacrocube}.

\begin{proof}[Proof of Lemma \ref{l:badmacrocube}]$ $\\
\emph{Step 1: Proof of \eqref{e:badmacrocube_typeII}}\\ 
Our main tool will be Lemma \ref{l:badmicrocube} a), and the argument is similar to the one in Step 3 of the proof of Theorem \ref{t:estimatespinnedset} c). We will choose $K_{\dd,M,p}\ge K_\dd'$ so that Lemma \ref{l:badmicrocube} can be applied.
 
Proceeding as in the proof of Lemma \ref{l:badmicrocube} we can estimate
\begin{align*}
&\zeta^\e_\Lambda\left(T^*\subset S^{*,II}_{K,L,M,\mathrm{bad}}(\A)\right)\\
&\quad=\zeta^\e_\Lambda\left(\forall Q\in T^*\exists Q'_Q\in\Qbox_{K\lmac}\text{ with }Q'_Q\subset Q,\left|\left\{q\in S^{(0)}_{K,M,\mathrm{bad}}(\A)\colon q\subset Q'_Q\right\}\right|\ge\frac14\left(\frac{\lmac}{\lmic}\right)^\dd\right)\\
&\quad\le\sum_{\substack{\underline{Q'}\in(\Qbox_{\ell_{j_*(\e)}})^{T^*}\\Q'_Q\subset Q\ \forall Q\in T^*}}\zeta^\e_\Lambda\left(\forall Q\in T^*\ \left|\left\{q\in S^{(0)}_{K,M,\mathrm{bad}}(\A)\colon q\subset Q'_Q\right\}\right|\ge\frac14\left(\frac{\lmac}{\lmic}\right)^\dd\right)\\
&\quad\le\sum_{\substack{\underline{Q'}\in(\Qbox_{\ell_{j_*(\e)}})^{T^*}\\Q'_Q\subset Q\ \forall Q\in T^*}}\zeta^\e_\Lambda\left(\left|\left\{q\in S^{(0)}_{K,M,\mathrm{bad}}(\A)\colon q\subset \bigcup_{Q\in T^*}Q'_Q\right\}\right|\ge\frac14\left(\frac{\lmac}{\lmic}\right)^\dd|T^*|\right)\,.
\end{align*}
Let \[\tilde T_{\underline{Q'}}=\{q\in\Qbox_{\ell_0}\colon q\subset\bigcup_{Q\in T^*}Q'_Q\}\] and note that $\left|\tilde T_{\underline{Q'}}\right|=\left(\frac{\lmac}{\lmic}\right)^\dd|T^*|$. Using \eqref{e:badmicrocube_j=0} we can now continue to estimate
\begin{align*}
\zeta^\e_\Lambda\left(T^*\subset S^{*,II}_{K,L,M,\mathrm{bad}}(\A)\right)&\le\sum_{\substack{\underline{Q'}\in(\Qbox_{\ell_{j_*(\e)}})^{T^*}\\Q'_Q\subset Q\ \forall Q\in T^*}}\zeta^\e_\Lambda\left(\left|\tilde T_{\underline{Q'}}\cap S^{(0)}_{K,M,\mathrm{bad}}(\A)\right|\ge\frac14\left|\tilde T_{\underline{Q'}}\right|\right)\\
&=\sum_{\substack{\underline{Q'}\in(\Qbox_{\ell_{j_*(\e)}})^{T^*}\\Q'_Q\subset Q\ \forall Q\in T^*}}\sum_{\substack{T\subset \tilde T_{\underline{Q'}}\\|T|\ge|\tilde T_{\underline{Q'}}|/4}}\zeta^\e_\Lambda\left(\tilde T_{\underline{Q'}}\cap S^{(0)}_{K,M,\mathrm{bad}}(\A)=T\right)\\
&\le\sum_{\substack{\underline{Q'}\in(\Qbox_{\ell_{j_*(\e)}})^{T^*}\\Q'_Q\subset Q\ \forall Q\in T^*}}\sum_{\substack{T\subset \tilde T_{\underline{Q'}}\\|T|\ge|\tilde T_{\underline{Q'}}|/4}}\zeta^\e_\Lambda\left(T\subset S^{(0)}_{K,M,\mathrm{bad}}(\A)\right)\\
&\le\sum_{\substack{\underline{Q'}\in(\Qbox_{\ell_{j_*(\e)}})^{T^*}\\Q'_Q\subset Q\ \forall Q\in T^*}}\sum_{\substack{T\subset \tilde T_{\underline{Q'}}\\|T|\ge|\tilde T_{\underline{Q'}}|/4}}\left(K^{\frac{\dd}{2^\dd}}(\dd+1)^{\frac{1}{2^\dd}}\exp\left(-\frac{K^\dd}{C_\dd}\right)\right)^{|T|}\\
&=\sum_{\substack{\underline{Q'}\in(\Qbox_{\ell_{j_*(\e)}})^{T^*}\\Q'_Q\subset Q\ \forall Q\in T^*}}\sum_{i=\lceil|\tilde T_{\underline{Q'}}|/4\rceil}^{|\tilde T_{\underline{Q'}}|}\binom{|\tilde T_{\underline{Q'}}|}{i}\left(K^{\frac{\dd}{2^\dd}}(\dd+1)^{\frac{1}{2^\dd}}\exp\left(-\frac{K^\dd}{C_\dd}\right)\right)^{i}\\
&=(L^\dd)^{|T^*|}\sum_{i=\lceil N/4\rceil}^{N}\binom{N}{i}\left(K^{\frac{\dd}{2^\dd}}(\dd+1)^{\frac{1}{2^\dd}}\exp\left(-\frac{K^\dd}{C_\dd}\right)\right)^{i}
\end{align*}
where we have abbreviated $N:=|\tilde T_{\underline{Q'}}|=\left(\frac{\lmac}{\lmic}\right)^\dd|T^*|$.
Let $p_{\dd,K}:=K^{\frac{\dd}{2^\dd}}(\dd+1)^{\frac{1}{2^\dd}}\exp\left(-\frac{K^\dd}{C_\dd}\right)$. Clearly $\lim_{K\to\infty}p_{\dd,K}=0$, so we can pick $K_{\dd,M,p}$ (at this point independently of $M$) large enough such that for $K\ge K_{\dd,M,p}$ we have $p_{\dd,K}\le\frac1{32}$. Then Lemma \ref{l:tailbound} with that choice of $p$ implies that
\begin{align*}
\zeta^\e_\Lambda\left(T^*\subset S^{*,II}_{K,L,M,\mathrm{bad}}(\A)\right)&\le(L^\dd)^{|T^*|}\left(\frac12\right)^{\frac N4}\\
&=\left(L^\dd\left(\frac12\right)^{\frac14\left(\frac{\lmac}{\lmic}\right)^\dd}\right)^{|T^*|}\,.
\end{align*}
For $\e$ small enough (depending on $\dd$, $L$, and $p$) the term in brackets is less than $p$, and we obtain \eqref{e:badmacrocube_typeII}.

\emph{Step 2: Proof of \eqref{e:badmacrocube_typeI}}\\
We can use Lemma \ref{l:badmicrocube} b). Given $Q\in \Qbox_{KL\lmac}$, there are at most \[\left(\left\lceil 3\frac{KL\lmac}{\ell_{j_*(\e)}}\right\rceil+3\right)^\dd\le\left(\frac{4KL\lmac}{\ell_{j_*(\e)}}\right)^\dd\]
cubes in $\Qbox^\#_{\ell_{j_*(\e)}}$ that intersect $Q$ (we used that $\frac{KL\lmac}{\ell_{j_*(\e)}}\ge8$). We can now proceed as in the proof of Lemma \ref{l:badmicrocube} and obtain
\begin{align*}
&\zeta^\e_\Lambda\left(T^*\subset S^{*,I}_{K,L,M,\mathrm{bad}}(\A)\right)\\
&\quad=\zeta^\e_\Lambda\left(\forall Q\in T^*\ \exists q_Q\in\Qbox^\#_{\ell_{j_*(\e)}}\text{ with }q_Q\cap Q\neq\varnothing,q_Q\in S^{(j_*(\e))}_{K,M,\mathrm{bad}}(\A)\right)\\
&\quad=\zeta^\e_\Lambda\left(\forall Q\in T^*\ \exists q_Q\in\Qbox^\#_{\ell_{j_*(\e)}}\text{ with }q_Q\cap Q\neq\varnothing\text{ such that }\bigcup_{Q\in T^*}\{q_Q\}\subset S^{(j_*(\e))}_{K,M,\mathrm{bad}}(\A)\right)\\
&\quad\le\sum_{\substack{\underline{q}\in(\Qbox^\#_{\ell_{j_*(\e)}})^{T^*}\\q_Q\cap Q\neq\varnothing\ \forall Q\in T^*}}\zeta^\e_\Lambda\left(\bigcup_{Q\in T^*}\{q_Q\}\subset S^{(j_*(\e))}_{K,M,\mathrm{bad}}(\A)\right)\,.
\end{align*}
If we assume for the moment that none of the cubes in $T^*$ are $l^\infty$-neighbours, then the $q_Q$ are pairwise distinct, and so $\bigcup_{Q\in T^*}\{q_Q\}$ has cardinality $|T^*|$. Now \eqref{e:badmicrocube_j>0} implies
\begin{equation}\label{e:badmacrocube_1}
\begin{aligned}
&\zeta^\e_\Lambda\left(T^*\subset S^{*,I}_{K,L,M,\mathrm{bad}}(\A)\right)\\
&\quad\le\sum_{\substack{\underline{q}\in(\Qbox^\#_{\ell_{j_*(\e)}})^{T^*}\\q_Q\cap Q\neq\varnothing\ \forall Q\in T^*}}\left(3^{(2^{j_*(\e)}-2)\dd}M^{s_{j_*(\e)}\dd}\left(K^{\frac{\dd}{2^\dd}}(\dd+1)^{\frac{1}{2^\dd}}\exp\left(-\frac{K^\dd}{C_\dd}\right)\right)^{2^{j_*(\e)}}\right)^{|T^*|}\\
&\quad\le\left(\left(\frac{4KL\lmac}{\ell_{j_*(\e)}}\right)^\dd\right)^{|T^*|}\left(3^{(2^{j_*(\e)}-2)\dd}M^{s_{j_*(\e)}\dd}\left(K^{\frac{\dd}{2^\dd}}(\dd+1)^{\frac{1}{2^\dd}}\exp\left(-\frac{K^\dd}{C_\dd}\right)\right)^{2^{j_*(\e)}}\right)^{|T^*|}\,.
\end{aligned}
\end{equation}
To remove the assumption that none of the the elements of $T^*$ are neighbours we proceed as in Step 1 of the proof of Lemma \ref{l:badmicrocube} and find a subset $T^*_{i_*}$ of $T^*$ of cardinality at least $\frac{|T^*|}{2^\dd}$ for which this is the case. Using \eqref{e:badmacrocube_1} for $T^*_{i_*}$ instead of $T^*$ we arrive at
\begin{equation}\label{e:badmacrocube_2}
\begin{aligned}
&\zeta^\e_\Lambda\left(T^*\subset S^{*,I}_{K,L,M,\mathrm{bad}}(\A)\right)&\\
&\quad\le\zeta^\e_\Lambda\left(T^*_{i_*}\subset S^{*,I}_{K,L,M,\mathrm{bad}}(\A)\right)\\
&\quad\le\left(\left(\frac{4KL\lmac}{\ell_{j_*(\e)}}\right)^\dd\right)^{\frac{|T^*|}{2^\dd}}\left(3^{(2^{j_*(\e)}-2)\dd}M^{s_{j_*(\e)}\dd}\left(K^{\frac{\dd}{2^\dd}}(\dd+1)^{\frac{1}{2^\dd}}\exp\left(-\frac{K^\dd}{C_\dd}\right)\right)^{2^{j_*(\e)}}\right)^{\frac{|T^*|}{2^\dd}}\,.
\end{aligned}
\end{equation}
It remains to bound the right hand side in \eqref{e:badmacrocube_2}. We begin by bounding $s_j$ from above. We have
\[s_j=j^3+\sum_{m=0}^jm^32^{j-m}=2^j\left(\frac{j^3}{2^j}+\sum_{m=0}^j\frac{m^3}{2^m}\right)\le2^j\left(4+\sum_{m=0}^\infty\frac{m^3}{2^m}\right)=30\cdot2^j\]
and so \eqref{e:badmacrocube_2} implies that
\begin{align*}
\zeta^\e_\Lambda\left(T^*\subset S^{*,I}_{K,M,\mathrm{bad}}(\A)\right)&\le\left(\left(\frac{4KL\lmac}{\ell_{j_*(\e)}}\left(3M^{30}K^{\frac{1}{2^\dd}}(\dd+1)^{\frac{1}{\dd2^\dd}}\exp\left(-\frac{K^\dd}{C_\dd}\right)\right)^{2^{j_*(\e)}}\right)^{\frac{\dd}{2^\dd}}\right)^{|T^*|}\\
&=\left(\left(\frac{4KL\lmac}{\ell_{j_*(\e)}}p_{\dd,K,M}^{2^{j_*(\e)}}\right)^{\frac{\dd}{2^\dd}}\right)^{|T^*|}
\end{align*}
where we have abbreviated $p_{\dd,K,M}=3M^{30}K^{\frac{1}{2^\dd}}(\dd+1)^{\frac{1}{\dd2^\dd}}\exp\left(-\frac{K^\dd}{C_\dd}\right)$. Note for each fixed $M$ and $\dd$ we have $\lim_{K\to\infty}p_{\dd,K,M}=0$, and so we can pick $K_{\dd,M,p}$ such that $p_{\dd,K,M}\le\frac12$ for $K\ge K_{\dd,M,p}$.

For these choices of $K$ we then know that
\[\zeta^\e_\Lambda\left(T^*\subset S^{*,I}_{K,L,M,\mathrm{bad}}(\A)\right)\le \left(\left(\frac{4KL\lmac}{\ell_{j_*(\e)}}2^{-2^{j_*(\e)}}\right)^{\frac{\dd}{2^\dd}}\right)^{|T^*|}\]
and we only need to show that 
\begin{equation}\label{e:badmacrocube_3}
\frac{4KL\lmac}{\ell_{j_*(\e)}}2^{-2^{j_*(\e)}}\le p^{\frac{2^\dd}{\dd}}
\end{equation}
when $\e$ is small enough. To show this, we need to bound $j_*(\e)$ from below. By definition, $j_*(\e)$ is equal to the largest integer $j$ such that $\ell_j\le\frac{KL\lmac}{8}$. In particular, $\ell_{j_*(\e)+1}>\frac{KL\lmac}{8}$, i.e. $M^{(j_*(\e)+1)^3}K\lmic>\frac{KL\lmac}8$. Estimating $(j_*(\e)+1)^3\le(2j_*(\e))^3$, we conclude
\[j_*(\e)\ge\frac12\sqrt[3]{\log_M\frac{L\lmac}{8\lmic}}\,.\]
Let us also abbreviate $X_{\dd,\e,L}=\frac{L\lmac}{8\lmic}$, and observe that $\lim_{\e\to0}X_{\dd,\e,L}=\infty$. For $t$ sufficiently large we have $\frac12\sqrt[3]{\log_Mt}\ge\log_2\log_2(t^2)$. This means that for $X_{\dd,\e,L}$ sufficiently large (i.e. $\e$ sufficiently small) we have $j_*(\e)\ge\frac12\sqrt[3]{\log_M X_{\dd,\e,L}}\ge\log_2\log_2(X_{\dd,\e,L}^2)$. Using this and the rather crude estimate $\ell_{j_*(\e)}\ge\ell_{j_0}=K\lmic$ for the denominator, we find
\[\frac{4KL\lmac}{\ell_{j_*(\e)}}2^{-2^{j_*(\e)}}\le 32X_{\dd,\e,L}2^{-2^{\log_2\log_2(X_{\dd,\e,L}^2)}}\le \frac{32X_{\dd,\e,L}}{X_{\dd,\e,L}^2}=\frac{32}{X_{\dd,\e,L}}\]
for $\e$ sufficiently small. This clearly implies that \eqref{e:badmacrocube_3} holds for $\e$ sufficiently small, which is \eqref{e:badmacrocube_typeII}.

\emph{Step 3: Proof of \eqref{e:badmacrocube}}\\
We can assume without loss of generality that $p\le\frac14$, as otherwise the estimate is trivial. Using \eqref{e:badmacrocube_typeI} and \eqref{e:badmacrocube_typeII} we see that
\begin{align*}
&\zeta^\e_\Lambda\left(T^*\subset S^{*}_{K,L,M,\mathrm{bad}}(\A)\right)\\
&\quad\le\sum_{T^*_{I}\cup T^*_{II}=T^*}\zeta^\e_\Lambda\left(T^*_{I}\subset S^{*,I}_{K,L,M,\mathrm{bad}}(\A),T^*_{II}\subset S^{*,II}_{K,L,M,\mathrm{bad}}(\A)\right)\\
&\quad\le\sum_{\substack{T^*_{I}\subset T^*\\|T^*_{I}|\ge|T^*|/2}}\zeta^\e_\Lambda\left(T^*_{I}\subset S^{*,I}_{K,L,M,\mathrm{bad}}(\A)\right)+\sum_{\substack{T^*_{II}\subset T^*\\|T^*_{II}|\ge|T^*|/2}}\zeta^\e_\Lambda\left(T^*_{II}\subset S^{*,II}_{K,L,M,\mathrm{bad}}(\A)\right)\\
&\quad\le2\sum_{i=\lceil |T^*|/2\rceil}^{|T^*|}\binom{|T^*|}{i}p^i\\
&\quad\le2(4p)^{\frac{|T^*|}{2}}
\end{align*}
where we have used Lemma \ref{l:tailbound} in the last step.
\end{proof}

Using Lemma \ref{l:badmacrocube} we can now estimate the probability that we find sets $U_j$ as in Lemma \ref{s:est_prob_cutoff}.

\begin{lemma}\label{l:existence_annuli}
Let $\dd\ge4$, and $\Lambda\Subset\Z^\dd$. Let $M\ge12$ be an odd integer. Then there is $K_{\dd,M}$ depending on $\dd$, $M$ only with the following property: Let $K\ge K_{\dd,M}$ be an odd multiple of 3, let $L$ be an odd integer. Let $U\in\Pbox_{KL\lmac}$ be a polymer consisting of $n=\frac{|U|}{(KL)^\dd\lmac^\dd}$ boxes. Let $k\ge0$ be an integer and let $\Omega_{U,k}$ be the event that there exist $U_0,\ldots,U_k\in\Pbox_{KL\lmac}$ such that $U\subset U_0$, for $j\in\{0,\ldots,k-1\}$ we have $U_j+Q_{KL\lmac}(0)\subset U_{j+1}$, $U_k\subset U+Q_{2kKL\lmac}(0)$, and \[\{Q\in\Qbox_{KL\lmac}\colon Q\subset (U_j+Q_{KL\lmac}(0))\setminus U\}\cap S^*_{K,L,M,\mathrm{bad}}(\A)=\varnothing\,.\]
Then, if $\e$ is small enough (depending on $K$, $L$ and $\dd$), we have
\begin{equation}\label{e:existence_annuli}
\zeta^\e_\Lambda(\Omega_{U,k})\ge 1-\frac{n}{2^k}\,.
\end{equation}
\end{lemma}
\begin{proof}
Let $p>0$ be a constant to be chosen later (depending on $\dd$ only). We pick $K_{\dd,M}\ge K_{\dd,M,p}$ with the $K_{\dd,M,p}$ from Lemma \ref{l:badmacrocube} so that this lemma can be applied.

We try to define the $U_j$ using a greedy algorithm. That is, we define $U_j$ as the union of all cubes $Q\in \Qbox_{KL\lmac}$ that can be connected to $U$ by a non-selfintersecting $l^\infty$-path of cubes in $\Qbox_{KL\lmac}$ that contains at most $j$ non-bad cubes. More precisely, $Q\in \Qbox_{KL\lmac}$ is a subset of $U$ if and only if there are $l\ge0$ and $Q^{(0)}=Q,Q^{(1)},\ldots, Q^{(l)}\subset U\in\Qbox_{KL\lmac}$ pairwise disjoint, with $d_\infty(Q^{(i)}, Q^{(i+1)})\le1$ for all $i\in\{0,\ldots,l-1\}$, such that at most $j$ of $Q^{(0)},Q^{(1)},\ldots ,Q^{(l-1)}$ are not in $S^*_{K,L,M,\mathrm{bad}}(\A)$.

This definition ensures that all $l^\infty$-neighbouring cubes to $U_j$ are not in $S^*_{K,L,M,\mathrm{bad}}(\A)$. So one sees that the $U_j$ satisfies all the conditions from $\Omega_{U,k}$ except that we do not yet know whether $U_k\subset U+Q_{2kKL\lmac}(0)$. 
This means that
\begin{equation}\label{e:existence_annuli1}
1-\zeta^\e_\Lambda(\Omega_{U,k})\le\zeta^\e_\Lambda\left(U_k\not\subset U+Q_{2kKL\lmac}(0)\right)
\end{equation}
and so it suffices to estimate the latter probability.

To do so, we define $\Pi_{U,k}$ to be the set of non-selfintersecting $l^\infty$-nearest neighbour paths $\Psi=\left(Q^{(0)}=Q,Q^{(1)},\ldots, Q^{(l)}\right)$ of cubes, that connect a cube $Q$ outside of $Q_{2kKL\lmac}(0)$ with $Q^{(l)}\subset U$. For $\Psi=\left(Q^{(0)}=Q,Q^{(1)},\ldots ,Q^{(l)}\right)$ let $\tilde\Psi=\left\{Q^{(0)},Q^{(1)},\ldots ,Q^{(l)}\right\}$ be the set of cubes in $\Psi$, and let $|\Psi|=|\tilde\Psi|=l+1$ be the number of cubes in it.

If $U_k\not\subset U+Q_{2kKL\lmac}(0)$, then there is some $\Psi\in\Pi_{U,k}$ that contains at most $k$ cubes within $Q^{(0)},\ldots, Q^{(l-1)}$ (and thus at most $k+1$ cubes within the cubes in $\Psi$) that are not bad. 
 Because $\Psi$ connects $U$ with a cube outside of $U+Q_{2kKL\lmac}(0)$, we have $|\Psi|\ge2k+2$. We can now continue \eqref{e:existence_annuli1} by using a union bound over all $\Psi\in\Pi_{U,k}$, and later over all bad subsets of $\tilde\Psi$, and obtain using Lemma \ref{l:badmacrocube} that
\begin{align*}
1-\zeta^\e_\Lambda(\Omega_{U,k})&\le\zeta^\e_\Lambda\left(\exists\Psi\in\Pi_{U,k}\colon\left|\tilde\Psi\setminus S^*_{K,L,M,\mathrm{bad}}(\A)\right|\le k+1\right)\\
&\le\sum_{\Psi\in\Pi_{U,k}}\zeta^\e_\Lambda\left(\left|\tilde\Psi\setminus S^*_{K,L,M,\mathrm{bad}}(\A)\right|\le k+1\right)\\
&=\sum_{\Psi\in\Pi_{U,k}}\zeta^\e_\Lambda\left(\exists T^*_\Psi\subset \tilde\Psi\colon T^*_\Psi\subset S^*_{K,L,M,\mathrm{bad}}(\A),|T^*_\Psi|\ge|\Psi|-k-1\right)\\
&\le\sum_{\Psi\in\Pi_{U,k}}\sum_{\substack{T^*_\Psi\subset\tilde\Psi\\|T^*_\Psi|\ge|\Psi|-k-1}}\zeta^\e_\Lambda\left(T^*_\Psi\subset S^*_{K,L,M,\mathrm{bad}}(\A)\right)\\
&\le\sum_{\Psi\in\Pi_{U,k}}\sum_{\substack{T^*_\Psi\subset\tilde\Psi\\|T^*_\Psi|\ge|\Psi|-k-1}}2(4p)^{\frac{|T^*_\Psi|}{2}}\\
&\le\sum_{\Psi\in\Pi_{U,k}}2(4p)^{\frac{|\Psi|-k-1}{2}}\,.
\end{align*}
We can reorganize this expression by summing over the lengths of $\Psi$. Recall that this length needs to be at least $2k+2$, and note that there are at most $n(2\dd)^l$ paths in $\Pi_{U,k}$ of length $l+1$. Thus, we obtain
\begin{align*}
1-\zeta^\e_\Lambda(\Omega_{U,k})&\le\sum_{l=2k+1}^\infty\sum_{\substack{\Psi\in\Pi_{U,k}\\|\tilde\Psi|=l+1}}2(4p)^{\frac{l-k}2}\\
&\le\sum_{l=2k+1}^\infty n(2\dd)^l2(4p)^{\frac{l-k}2}\\
&=\frac{2n}{(2\sqrt{p})^k}\sum_{l=2k+1}^\infty(4\dd\sqrt{p})^l\,.
\end{align*}
We choose $p\le\frac{1}{144\dd^2}$, so that $4\dd\sqrt{p}\le\frac13$. Then, in particular, the series on the right-hand side converges, and we can continue
\begin{align*}
1-\zeta^\e_\Lambda(\Omega_{U,k})&\le\frac{2n}{(2\sqrt{p})^k}\frac{(4\dd\sqrt{p})^{2k+1}}{1-4\dd\sqrt{p}}\\
&\le\frac{n(4\dd\sqrt{p})^{2k}}{(2\sqrt{p})^k}\\
&=n(8\dd^2\sqrt{p})^k\,.
\end{align*}
We finalize our choice of $p$ as $p=\frac{1}{256\dd^4}$. Then \eqref{e:existence_annuli} follows. 
\end{proof}

\begin{proof}[Proof of Theorem \ref{t:decay_high_prob}]
We want to combine Lemma \ref{l:decay_many_annuli} and Lemma \ref{l:existence_annuli}. That is, we first choose $M$ so large that Lemma \ref{l:decay_many_annuli} can be applied. Then we choose $K$ large enough that Lemma \ref{l:existence_annuli} can be applied. Then Lemma \ref{l:decay_many_annuli} applies for sufficiently large $L$. We choose $\hat N_\dd=KL$, and note that Lemma \ref{l:existence_annuli} implies the bound on the probability of $\Omega_{U,k}$. 

It remains to check that \eqref{e:decay_many_annuli_ext} and \eqref{e:decay_many_annuli_int} imply \eqref{e:decay_high_prob_ext} and \eqref{e:decay_high_prob_int} if $\Omega_{U,k}$ holds. This follows from the observation that $(U_0+Q_{\hat N_\dd\lmac}(0))\setminus U_0\subset (U+Q_{2k\hat N_\dd\lmac}(0))\setminus U$ and $(U_{k-1}+Q_{\hat N_\dd\lmac}(0))\setminus U_{k-1}\subset (U+Q_{2k\hat N_\dd\lmac}(0))\setminus U$.
\end{proof}

\begin{remark}\label{r:lengthscales}
Let us comment on why the lengthscales $\ell_j=M^{j^3}K\lmic$ are a natural choice. For the construction in Step 2 of the proof of Lemma \ref{l:existence_cutoff} we need that $\log\frac{\ell_j}{\ell_{j+1}}$ is summable as otherwise we could not bound $\|\nabla_1^2\eta\|_{L^\infty}$ in \eqref{e:existence_cutoff5}. This means that $\ell_j$ needs to grow rather fast (e.g., $\ell_j=M^{j^2}K\lmic$ would not be fast enough). On the other hand, for the estimate on the probability of bad cubes of type I in Lemma \ref{l:badmacrocube} we need that the exponent $s_j\dd$ of $M$ in \eqref{e:badmicrocube_j>0} is at most $C_\dd2^j$. This exponent arises from the combinatorial factors $\left(\frac{\ell_j}{\ell_{j-1}}\right)^{2\dd}$ in \eqref{e:badmicrocube_3}. This means that $\ell_j$ cannot grow too fast (e.g., $\ell_j=M^{2^j}K\lmic$ would be too fast).

Fortunately, both requirements are compatible, and in fact, our choice $\ell_j=M^{j^3}K\lmic$ satisfies both of them.
\end{remark}

\section{Pathwise bounds on the field}
We can now turn to the proof of Theorem \ref{t:estimatesfield} and of the second part of Theorem \ref{t:thermolimit}. Before we actually give the proofs, however, we state and prove various quenched estimates for $G_{\Lambda\setminus A}$ that hold for all $A$, or at least up to exponentially small probability in $A$. The main tool for that will be Theorem \ref{t:decay_high_prob}.

We prove those estimates in Section \ref{s:quenched}. Then, in Sections \ref{s:annealed} and \ref{s:thermo_limit_field} we use them to deduce Theorem \ref{t:estimatesfield} and the second part of Theorem \ref{t:thermolimit}, respectively.
\subsection{Quenched estimates on the Green's function}\label{s:quenched}
We write $G_{\Lambda,y}$ for $G_\Lambda(\cdot,y)$. 
We have the following straightforward result for $G_\Lambda$. This is essentially the same as \cite[Lemma 4.2]{Schweiger2020}.
\begin{lemma}\label{l:identity_G}
Let $\Lambda\Subset\Z^\dd$ and $x,y\in\Z^\dd$. Then
\begin{equation}\label{e:identity_G_1}
G_\Lambda(x,y)=\left(\nabla_1^2G_{\Lambda,x},\nabla_1^2G_{\Lambda,y}\right)_{L^2(\Z^\dd)}\,.
\end{equation}
Furthermore, we have
\begin{equation}\label{e:identity_G_2}
|G_\Lambda(x,y)|\le \sqrt{G_\Lambda(x,x)G_\Lambda(y,y)}\,.
\end{equation}
\end{lemma}
\begin{proof}
For \eqref{e:identity_G_1} we calculate
\[
G_\Lambda(x,y)=\left(\I_{\cdot=x},G_{\Lambda,y}\right)_{L^2(\Z^\dd)}=\left(\Delta_1^2G_{\Lambda,x},G_{\Lambda,y}\right)_{L^2(\Z^\dd)}=\left(\nabla_1^2G_{\Lambda,x},\nabla_1^2G_{\Lambda,y}\right)_{L^2(\Z^\dd)}\,.
\]
The estimate \eqref{e:identity_G_2} follows directly from the interpretation of $G_\Lambda$ as a covariance. Alternatively, we can use \eqref{e:identity_G_1} together with the Cauchy-Schwarz inequality to estimate
\[
|G_\Lambda(x,y)|=\left|\left(\nabla_1^2G_{\Lambda,x},\nabla_1^2G_{\Lambda,y}\right)_{L^2(\Z^\dd)}\right|\le\left\|\nabla_1^2G_{\Lambda,x}\right\|_{L^2(\Z^\dd)}\left\|\nabla_1^2G_{\Lambda,y}\right\|_{L^2(\Z^\dd)}=\sqrt{G_\Lambda(x,x)G_\Lambda(y,y)}\,.
\]
\end{proof}
Next, we establish some quenched tail estimates on $G_{\Lambda\setminus \A}(x,x)$. If $\dd\ge5$, then there are deterministic bounds on $G_{\Lambda\setminus \A}(x,x)$ by \eqref{e:estvariance1}, so this is only interesting if $\dd=4$. 
\begin{lemma}\label{l:var_quenched}
If $\dd=4$, there is a constant $\tilde\gamma>0$ such that if $\Lambda\Subset\Z^d$, $x\in\Lambda$ and $\e$ is small enough (depending on $\dd$ only) then for any $t\ge\tilde\gamma$ we have
\begin{equation}
\zeta^\e_\Lambda\left(G_{\Lambda\setminus\A}(x,x)\le t\right)\ge1-\exp\left(-\frac{\e\exp(16\pi^2(t-\tilde\gamma))}{C|\log\e|^\frac12}\right)\label{e:var_quenched_upp}\,,
\end{equation}
and for $\alpha>0$, $x\in\Lambda$ with $d(x,\Z^\dd\setminus\Lambda)\ge\e^{-\alpha}$ and any $0\le t\le \frac{1}{8\pi^2}\log\left(1+d(x,\Z^\dd\setminus\Lambda)-\e^{-\alpha}\right)-\tilde\gamma$ we have
\begin{equation}
\zeta^\e_\Lambda\left(G_{\Lambda\setminus\A}(x,x)\le t\right)\le1-\exp\left(-\frac{C_\alpha\e\exp(32\pi^2(t+\tilde\gamma))}{|\log\e|^\frac12}\right)\label{e:var_quenched_low}
\end{equation}
for some constant $C$.

Furthermore, if $\dd\ge4$, $k\in\N$, and $y\in\Lambda$ there are constants $\tilde\gamma_\dd$ such that
\begin{equation}\label{e:event_var_bounded}
\zeta^\e_\Lambda\left(G_{\Lambda\setminus \A}(y,y)\le \I_{\dd=4}\frac{\log k+|\log\e|}{16\pi^2}+\tilde\gamma_\dd\right)\ge1-\frac{1}{2^k}\,.
\end{equation}
\end{lemma}
\begin{proof}
We begin with \eqref{e:var_quenched_upp}. This follows easily from Lemma \ref{l:estvariance} and Theorem \ref{t:estimatespinnedset} c). Indeed, if $x\in\A$ then $G_{\Lambda\setminus\A}(x,x)=0$, while if $x\notin\A$ we know from \eqref{e:estvariance1} that 
\[G_{\Lambda\setminus\A}(x,x)\le \frac{1}{4\pi^2}\log(1+d(x,\tilde\A))+C\le \frac{1}{4\pi^2}\log(d(x,\tilde\A))+C\le \frac{1}{4\pi^2}\log(d(x,\A))+C\,.\]
So, there is a constant $\tilde\gamma'$ such that $G_{\Lambda\setminus\A}(x,x)> t$ for $t\ge\tilde\gamma'$ implies $d(x,\A)\ge\exp(4\pi^2(t-\tilde\gamma'))$. Using \eqref{e:lowerestpinnedset_dim4} we can estimate that
\begin{align*}
\zeta^\e_\Lambda\left(G_{\Lambda\setminus\A}(x,x)\le t\right)&\ge\zeta^\e_\Lambda\left(d(x,\A)\le\exp(4\pi^2(t-\tilde\gamma'))\right)\\
&=1-\zeta^\e_\Lambda\left(\A\cap Q_{\exp(4\pi^2(t-\tilde\gamma'))}(x)=\varnothing\right)\\
&\ge1-(1-p_{4,-})^{|Q_{\exp(4\pi^2(t-\tilde\gamma'))}(x)|}\\
&\ge1-\exp\left(-\frac{p_{4,-}\exp(4\pi^2(t-\tilde\gamma'))^4}{C}\right)\\
&\ge1-\exp\left(-\frac{\e\exp(16\pi^2(t-\tilde\gamma'))}{C|\log\e|^\frac12}\right)
\end{align*}
which is \eqref{e:var_quenched_upp}, if we choose $\tilde\gamma\ge\tilde\gamma'$.

The argument for \eqref{e:var_quenched_low} is similar. We have that
\[G_{\Lambda\setminus\A}(x,x)\ge \frac{1}{8\pi^2}\log(1+d(x,\tilde\A))-C\ge \frac{1}{8\pi^2}\log(d(x,\tilde\A))-C\]
if $x\notin\A$ and $G_{\Lambda\setminus\A}(x,x)=0$ if $x\in\A$. So there is a constant $\tilde\gamma''$ such that $G_{\Lambda\setminus\A}(x,x)\le t$ implies $d(x,\tilde\A)\le\exp(8\pi^2(t+\tilde\gamma''))$. Choosing $\tilde\gamma\ge\tilde\gamma''$ our assumption $t\le \frac{1}{8\pi^2}\log\left(1+d(x,\Z^\dd\setminus\Lambda)-\e^{-\alpha}\right)-\tilde\gamma$ ensures that $\exp(8\pi^2(t+\gamma''))\le1+d(x,\Z^\dd\setminus\Lambda)-\e^{-\alpha}$. This means that $Q_{\exp(8\pi^2(t+\tilde\gamma''))}(x)$ still has distance at least $\e^{-\alpha}$ from $\Z^\dd\setminus\Lambda$ (and in particular $d(x,\tilde\A)<d(x,\Z^\dd\setminus\Lambda)$, so that $d(x,\tilde\A)=d(x,\A)$). Thus, we can apply \eqref{e:upperestpinnedset_dim4} and obtain
\begin{align*}
\zeta^\e_\Lambda\left(G_{\Lambda\setminus\A}(x,x)\le t\right)&\le\zeta^\e_\Lambda\left(d(x,\A)\le\exp(8\pi^2(t+\tilde\gamma''))\right)\\
&=1-\zeta^\e_\Lambda\left(\A\cap Q_{\exp(8\pi^2(t+\tilde\gamma''))}(x)=\varnothing\right)\\
&\le1-\exp\left(-\frac{C_\alpha\e\exp(32\pi^2(t+\tilde\gamma''))}{|\log\e|^\frac12}\right)
\end{align*}
This is \eqref{e:var_quenched_low}, if we choose $\tilde\gamma\ge\tilde\gamma''$.

Regarding \eqref{e:event_var_bounded}, note that if $\dd\ge5$ this is a trivial consequence of \eqref{e:estvariance1}, while if $\dd=4$ we can consider the choice $t=\frac{\log k+|\log\e|}{16\pi^2}+\tilde\gamma$ in \eqref{e:var_quenched_upp} to obtain
\begin{align*}
	\zeta^\e_\Lambda\left(G_{\Lambda\setminus\A}(y,y)\le \frac{\log k+|\log\e|}{16\pi^2}+\tilde\gamma\right)&\ge1-\exp\left(-\frac{\e\exp(16\pi^2\frac{\log k+|\log\e|}{16\pi^2})}{C|\log\e|^\frac12}\right)\\
&\ge1-\exp\left(-\frac{k}{C|\log\e|^\frac12}\right)
\end{align*}
and the right-hand side is at least $1-\frac{1}{2^k}$ if $\e$ is small enough.
\end{proof}

Next, we prove quenched bounds on the covariance.
\begin{lemma}\label{l:cov_quenched}
Let $\dd\ge4$, $\Lambda\Subset\Z^\dd$, and $x,y\in\Lambda$. Then, if $\e$ is small enough (depending on $\dd$), we have
\begin{equation}\label{e:cov_quenched}
\zeta^\e_\Lambda\left(|G_{\Lambda\setminus\A}(x,y)|\le C_\dd\exp\left(-\frac{|x-y|_\infty}{C_\dd\lmac}\right)\frac{1+\I_{\dd=4}|\log\e|^{5/4}}{\e^{1/2}}\right)\ge 1-\exp\left(-\frac{|x-y|_\infty}{C_\dd\lmac}\right)
\end{equation}
for some constant $C_\dd$.
\end{lemma}
\begin{proof}
By translating $\Lambda$ and $\A$ we can assume $y=0$. This ensures in particular that $y$ is in the centre of a box in $\Qbox_{l}$ for any $l$. Let $U=Q_{\hat N_\dd\lmac}(0)$ with the $\hat N_\dd$ from Theorem \ref{t:decay_high_prob}, and consider for now the case that $|x-y|_\infty\ge8\hat N_\dd\lmac$. Let $k=\left\lceil \frac{|x-y|_\infty}{8\hat N_\dd\lmac}\right\rceil$, and note that $k\le\frac{|x-y|_\infty}{2\hat N_\dd\lmac}$.

Assume that $A\in\Omega_{U,k}$ with the $\Omega_{U,k}$ from Theorem \ref{t:decay_high_prob}. Then that theorem (applied to $G_{\Lambda\setminus A,y}$) and \eqref{e:identity_G_1} imply that
\begin{equation}\label{e:cov_quenched1}
\begin{aligned}
\left\|\nabla_1^2G_{\Lambda\setminus A,y}\right\|^2_{L^2(\Z^\dd\setminus(U+Q_{2k\hat N_\dd\lmac}(0)))}&\le\frac{1}{2^k}\left\|\nabla_1^2G_{\Lambda\setminus A,y}\right\|^2_{L^2(\Z^\dd\setminus U)}\\
&=\frac{1}{2^k}G_{\Lambda\setminus A}(y,y)\,.
\end{aligned}
\end{equation}

Furthermore, suppose that $A\in\Omega_{x,k}$ with the $\Omega_{x,k}$ from Lemma \ref{l:localpoincare_prob}. Then we can conclude
\begin{equation}\label{e:cov_quenched2}
\begin{aligned}
|G_{\Lambda\setminus A}(x,y)|^2&=|G_{\Lambda\setminus A,y}(x)|^2\\
&\le C_\dd \frac{k^{\dd}\left(1+\I_{\dd=4}(\log k+|\log\e|^{3/2})\right)}{\e}\|\nabla_1^2G_{\Lambda\setminus A,y}\|^2_{L^2(Q_{kN_\dd\lmic}(x))}\,.
\end{aligned}
\end{equation}
For $\e$ small enough (depending on $\dd$) we have $N_\dd\lmic\le\hat N_\dd\lmac$. Then $U+Q_{2k\hat N_\dd\lmac}(0)=Q_{(4k+1)\hat N_\dd\lmac/2}(0)$ and $Q_{kN_\dd\lmic}(x)$ are disjoint, and so we can combine \eqref{e:cov_quenched1} and \eqref{e:cov_quenched2} into
\begin{equation}\label{e:cov_quenched3}
\begin{aligned}
|G_{\Lambda\setminus A}(x,y)|^2&\le C_\dd \frac{k^{\dd}\left(1+\I_{\dd=4}(\log k+|\log\e|^{3/2})\right)}{\e}\|\nabla_1^2G_{\Lambda\setminus A,y}\|^2_{L^2(Q_{kN_\dd\lmic}(x))}\\
&\le C_\dd \frac{k^{\dd}\left(1+\I_{\dd=4}(\log k+|\log\e|^{3/2})\right)}{\e}\left\|\nabla_1^2G_{\Lambda\setminus A,y}\right\|^2_{L^2(\Z^\dd\setminus(U+Q_{2k\hat N_\dd\lmac}(0)))}\\
&\le C_\dd \frac{k^{\dd}\left(1+\I_{\dd=4}(\log k+|\log\e|^{3/2})\right)}{2^k\e}G_{\Lambda\setminus A}(y,y)\,.
\end{aligned}
\end{equation}

Next, let $\tilde\Omega_{y,k}$ be the event from \eqref{e:event_var_bounded}. If $A\in\tilde\Omega_{y,k}$, then \eqref{e:event_var_bounded} and \eqref{e:cov_quenched3} imply
\begin{equation}\label{e:cov_quenched4}
\begin{aligned}
|G_{\Lambda\setminus A}(x,y)|^2&\le C_\dd \frac{k^{\dd}\left(1+\I_{\dd=4}(\log k+|\log\e|^{3/2})\right)}{2^k\e}\left(1+\I_{\dd=4}(\log k+|\log\e|)\right)\\
&\le C_\dd\left(\frac34\right)^k\frac{1+\I_{\dd=4}|\log\e|^{5/2}}{\e}\\
&\le C_\dd\exp\left(-\log\frac34\frac{|x-y|_\infty}{4\hat N_\dd\lmac}\right)\frac{1+\I_{\dd=4}|\log\e|^{5/2}}{\e}\\
&\le C_\dd\exp\left(-\frac{|x-y|_\infty}{C_\dd\lmac}\right)\frac{1+\I_{\dd=4}|\log\e|^{5/2}}{\e}\,.
\end{aligned}
\end{equation}
This estimate holds if $A\in\Omega_{U,k}\cap\Omega_{x,k}\cap\tilde\Omega_{y,k}$. But that probability is easy to bound:
\[\zeta^\e_\Lambda\left(\Omega_{U,k}\cap\Omega_{x,k}\cap\tilde\Omega_{y,k}\right)\ge1-\frac{1}{2^k}-\frac{1}{2^{k^\dd}}-\frac{1}{2^k}\ge1-\exp\left(-\frac{k}{C_\dd}\right)\ge1-\exp\left(-\frac{|x-y|_\infty}{C_\dd\lmac}\right)\,.\]
Therefore we have shown that the set of $A$ for which \eqref{e:cov_quenched4} holds has measure at least $1-\exp\left(-\frac{|x-y|_\infty}{C_\dd\lmac}\right)$, and this implies \eqref{e:cov_quenched}.

It remains to consider the case that $|x-y|_\infty<8\hat N_\dd\lmac$. In that case we need to show
\[
\zeta^\e_\Lambda\left(|G_{\Lambda\setminus\A}(x,y)|\le C_\dd\frac{1+\I_{\dd=4}|\log\e|^{5/4}}{\e^{1/2}}\right)\ge c_\dd\,.
\]
This follows immediately from \eqref{e:event_var_bounded} and \eqref{e:identity_G_2}.
\end{proof}

We also need to quantify that for a large domain $\Lambda$ the covariances far inside $\Lambda$ depend only weakly on the precise shape of $\Lambda$.
\begin{lemma}\label{l:diff_covs_quenched}
Let $\dd\ge4$, $\Lambda'\subset\Lambda\Subset\Z^\dd$. Let $\e$ be small enough (depending on $\dd$ only). Suppose that $r$, $R$ are integers with $2\hat N_\dd\lmac\le r$, $8r\le R$ and $Q_R(0)\subset\Lambda'$.  we have
\begin{equation}\label{e:diff_covs_quenched}
\begin{aligned}
&\zeta^\e_\Lambda\left(\max_{x,y\in Q_r(0)}\left|G_{\Lambda\setminus\A}(x,y)-G_{\Lambda'\setminus\A}(x,y)\right|\le C_\dd\exp\left(-\frac{R-r}{C_\dd\lmac}\right)\frac{1+\I_{\dd=4}|\log\e|^{5/4}}{\e^{1/2}}\right)\\
&\quad\ge1-C_\dd r^\dd\exp\left(-\frac{R-r}{C_\dd\lmac}\right)\,.
\end{aligned}
\end{equation}
\end{lemma}
\begin{proof}
The idea is that $H_y:=G_{\Lambda\setminus\A,y}-G_{\Lambda'\setminus\A,y}$ is biharmonic in $Q_R(0)$. We will use Theorem \ref{t:decay_high_prob} a) to conclude that the $L^2$-norm of $\nabla_1^2H$ outside of $Q_{R/2}(0)$ is exponentially small, and then use Theorem \ref{t:decay_high_prob} b) to conclude that the $L^2$-norm of $\nabla_1^2H$ in $Q_r(0)$ is exponentially small. Of course these estimates hold not for all realizations of $\A$, but we will estimate that they hold for sufficiently many.

Let $\tilde r=\left(\left\lceil \frac{r}{\hat N_\dd\lmac}\right\rceil+\frac12\right)\hat N_\dd\lmac$ and $\tilde R=\left(\left\lfloor \frac{R}{\hat N_\dd\lmac}\right\rfloor-\frac12\right)\hat N_\dd\lmac$. Then $r\le\tilde r\le 2r$, $\frac R2\le\tilde R\le R$. We let $U=Q_{\tilde r}(0)$ and note that $U\in\Pbox_{\hat N_\dd\lmac}$ is a polymer consisting of $\left(\frac{2\tilde r}{\hat N_\dd\lmac}\right)^\dd$ boxes in $\Qbox_{\hat N_\dd\lmac}$. Let $k=\left\lfloor \frac{\tilde R-\tilde r}{4\hat N_\dd\lmac}\right\rfloor$ and note that $k\ge\frac{R-r}{C_\dd\lmac}$.
Theorem \ref{t:decay_high_prob} a) implies that on the event $\Omega_{U,k}$ we have
\begin{align*}
\left\|\nabla_1^2G_{\Lambda\setminus A,y}\right\|^2_{L^2(\Z^\dd\setminus(U+Q_{2k\hat N_\dd\lmac}(0)))}&\le\frac{1}{2^k}\left\|\nabla_1^2G_{\Lambda\setminus A,y}\right\|^2_{L^2(\Z^\dd\setminus U)}\\
&\le\frac{1}{2^k}\left\|\nabla_1^2G_{\Lambda\setminus A,y}\right\|^2_{L^2(\Z^\dd)}\\
&=\frac{1}{2^k}G_{\Lambda\setminus A}(y,y)
\end{align*}
as $G_{\Lambda\setminus A,y}=0$ on $\tilde A\setminus U$ and $G_{\Lambda\setminus A,y}\Delta_1^2G_{\Lambda\setminus A,y}=0$ on $\Z^\dd\setminus U$. 

Analogously we have 
\[
\left\|\nabla_1^2G_{\Lambda'\setminus A,y}\right\|^2_{L^2(\Z^\dd\setminus(U+Q_{2k\hat N_\dd\lmac}(0)))}\le\frac{1}{2^k}G_{\Lambda'\setminus A}(y,y)
\]
as $G_{\Lambda\setminus A,y}=0$ on $\tilde A\setminus U$ and $G_{\Lambda'\setminus A,y}\Delta_1^2G_{\Lambda'\setminus A,y}=0$ on $\Z^\dd\setminus U$ (even though $G_{\Lambda'\setminus A,y}$ is not biharmonic everywhere on $\Lambda\setminus(A\cup U)$).

If we define $H_{A,y}=:G_{\Lambda\setminus A,y}-G_{\Lambda'\setminus A,y}$, the preceding two estimates imply that
\begin{equation}\label{e:diff_covs_quenched1}
\begin{aligned}
&\left\|\nabla_1^2H_{A,y}\right\|^2_{L^2(\Z^\dd\setminus(U+Q_{2k\hat N_\dd\lmac}(0)))}\\
&\quad\le2\left\|\nabla_1^2G_{\Lambda\setminus A,y}\right\|^2_{L^2(\Z^\dd\setminus(U+Q_{2k\hat N_\dd\lmac}(0)))}+2\left\|\nabla_1^2G_{\Lambda'\setminus A,y}\right\|^2_{L^2(\Z^\dd\setminus(U+Q_{2k\hat N_\dd\lmac}(0)))}\\
&\quad\le\frac{1}{2^{k-1}}\left(G_{\Lambda\setminus A}(y,y)+G_{\Lambda'\setminus A}(y,y)\right)\,.
\end{aligned}
\end{equation}
The polymer $U+Q_{2k\hat N_\dd\lmac}(0)$ consists of $\left(\frac{\tilde r}{2\hat N_\dd\lmac}+4k\right)^\dd$ boxes in $\Qbox_{\hat N_\dd\lmac}$. The function $H_{A,y}$ satisfies $H_{A,y}\Delta_1^2H_{A,y}=0$ on $U+Q_{2k\hat N_\dd\lmac}(0)\subset Q_R(0)$ as the two singularities cancel out. So we can apply Theorem \ref{t:decay_high_prob} b) and obtain on the event $\Omega_{U+Q_{2k\hat N_\dd\lmac}(0),k}$ that
\begin{equation}\label{e:diff_covs_quenched2}
\left\|\nabla_1^2H_{A,y}\right\|^2_{L^2(U+Q_{2k\hat N_\dd\lmac}(0))}\le\frac{1}{2^k}\left\|\nabla_1^2H_{A,y}\right\|^2_{L^2((U+Q_{4k\hat N_\dd\lmac}(0))\setminus(U+Q_{2k\hat N_\dd\lmac}(0)))}\,.
\end{equation}

Furthermore, we can introduce the event $\tilde\Omega_{y,k}$ as in \eqref{e:event_var_bounded}. By definition we have 
\begin{equation}\label{e:diff_covs_quenched3}
G_{\Lambda\setminus A}(y,y)\le C_\dd\left(1+\I_{ \dd=4}(\log k+|\log\e|)\right)
\end{equation}
on that event. We claim that on the event $\tilde\Omega_{y,k}$ we also have
\begin{equation}\label{e:diff_covs_quenched4}
G_{\Lambda'\setminus A}(y,y)\le C_\dd\left(1+\I_{ \dd=4}(\log k+|\log\e|)\right)\,.
\end{equation}
Indeed, if $\dd\ge5$ this is once more a trivial consequence of \eqref{e:estvariance1}, while if $\dd=4$ we can use \eqref{e:estvariance2} to estimate
\begin{align*}G_{\Lambda'\setminus A}(y,y)&\le \frac{1}{4\pi^2}\log(1+d(x,(A\cup(\Z^\dd\setminus\Lambda')))+C\\
&\le\frac{1}{4\pi^2}\log(1+d(x,(A\cup(\Z^\dd\setminus\Lambda)))+C\\
&\le G_{\Lambda\setminus A}(y,y)+C
\end{align*}
so that \eqref{e:diff_covs_quenched4} is a consequence of \eqref{e:diff_covs_quenched3}.

Finally, if $A\in\Omega_{x,k}$ with the event $\Omega_{x,k}$ from Lemma \ref{l:localpoincare_prob}, we have
\begin{equation}\label{e:diff_covs_quenched5}
|H_{A,y}(x)|^2\le C_\dd \frac{k^\dd\left(1+\I_{\dd=4}(\log k+|\log\e|^{3/2})\right)}{\e}\|\nabla_1^2H_{A,y}(x)\|^2_{L^2(Q_{kN_\dd\lmic}(x))}\,.
\end{equation}
We choose $\e$ small enough so that $N_\dd\lmic\le 2\hat N_\dd\lmac$. Then, in particular, $Q_{kN_\dd\lmic}(x)\subset U+Q_{kN_\dd\lmic}(0)\subset U+Q_{2kN_\dd\lmac}(0)$.

Now we can combine the estimates we have just collected. More precisely, assume that $A\in \Omega_{U,k}\cap \Omega_{U+Q_{2k\hat N_\dd\lmac}(0),k}\cap \tilde\Omega_{y,k}\cap \Omega_{x,k}$. Then we can use \eqref{e:diff_covs_quenched1}, \eqref{e:diff_covs_quenched2}, \eqref{e:diff_covs_quenched3}, \eqref{e:diff_covs_quenched4} and \eqref{e:diff_covs_quenched5} to obtain
\begin{equation}\label{e:diff_covs_quenched6}
\begin{aligned}
&\left|G_{\Lambda\setminus A}(x,y)-G_{\Lambda'\setminus A}(x,y)\right|^2\\
&\quad=|H_{A,y}(x)|^2\\
&\quad\le C_\dd \frac{k^\dd\left(1+\I_{\dd=4}(\log k+|\log\e|^{3/2})\right)}{\e}\|\nabla_1^2H_{A,y}\|^2_{L^2(Q_{kN_\dd\lmic}(x))}\\
&\quad\le C_\dd \frac{k^\dd\left(1+\I_{\dd=4}(\log k+|\log\e|^{3/2})\right)}{\e}\left\|\nabla_1^2H_{A,y}\right\|^2_{L^2(U+Q_{2k\hat N_\dd\lmac}(0))}\\
&\quad\le C_\dd \frac{k^\dd\left(1+\I_{\dd=4}(\log k+|\log\e|^{3/2})\right)}{2^k\e}\left\|\nabla_1^2H_{A,y}\right\|^2_{L^2((U+Q_{4k\hat N_\dd\lmac})\setminus(U+Q_{2k\hat N_\dd\lmac}(0)))}\\
&\quad\le C_\dd \frac{k^\dd\left(1+\I_{\dd=4}(\log k+|\log\e|^{3/2})\right)}{2^k\e}\left\|\nabla_1^2H_{A,y}\right\|^2_{L^2(\Z^\dd\setminus(U+Q_{2k\hat N_\dd\lmac}(0)))}\\
&\quad\le C_\dd \frac{k^\dd\left(1+\I_{\dd=4}(\log k+|\log\e|^{3/2})\right)}{2^{2k-1}\e}\left(G_{\Lambda\setminus A}(y,y)+G_{\Lambda'\setminus A}(y,y)\right)\\
&\quad\le C_\dd \frac{k^\dd\left(1+\I_{\dd=4}(\log k+|\log\e|^{3/2})\right)}{2^{2k-1}\e}\left(1+\I_{\dd=4}(\log k+|\log\e|)\right)\\
&\quad\le C_\dd\left(\frac12\right)^k\frac{1+\I_{\dd=4}|\log\e|^{5/2}}{\e}\\
&\quad\le C_\dd\exp\left(-\frac{R-r}{C_\dd\lmac}\right)\frac{1+\I_{\dd=4}|\log\e|^{5/2}}{\e}\,.
\end{aligned}
\end{equation}
From \eqref{e:diff_covs_quenched6} we see that on the event \[\Omega:=\Omega_{U,k}\cap \Omega_{U+Q_{2k\hat N_\dd\lmac}(0),k}\cap \bigcap_{y\in Q_r(0)}\tilde\Omega_{y,k}\cap \bigcap_{x\in Q_r(0)}\Omega_{x,k}\] we have the desired estimate.
So it only remains to bound the probability of $\Omega$ from below. For this we use a union bound to see
\begin{align*}
\zeta^\e_\Lambda(\Omega)&\ge 1-\left(\frac{2\tilde r}{\hat N_\dd\lmac}\right)^\dd\frac{1}{2^k}-\left(\frac{2\tilde r}{\hat N_\dd\lmac}+4k\right)^\dd\frac{1}{2^k}-(2r+1)^\dd\frac{1}{2^{k^\dd}}-(2r+1)^\dd\frac{1}{2^k}\\
&\ge 1-C_\dd r^\dd\exp\left(-\frac{k}{C_\dd}\right)\\
&\ge 1-C_\dd r^\dd\exp\left(-\frac{R-r}{C_\dd\lmac}\right)\,.
\end{align*}
This completes the proof.
\end{proof}

\subsection{Estimates on variance and covariance}\label{s:annealed}
\begin{proof}[Proof of Theorem \ref{t:estimatesfield}]
We first prove part a) and then part b).

\emph{Step 1: Estimates on the variance}\\
We have that
\begin{equation}\label{e:decomp_variance}
\E^\e_\Lambda(\psi_x^2)=\sum_{A\subset\Lambda}\zeta^\e_\Lambda(A)\E_{\Lambda\setminus A}(\psi_x^2)=\sum_{A\subset\Lambda}\zeta^\e_\Lambda(A)G_{\Lambda\setminus A}(x,x)\,.
\end{equation}
Thus, \eqref{e:est_var_d>4} follows immediately from \eqref{e:estvariance1}. For \eqref{e:est_var_d=4} we use Lemma \ref{l:var_quenched}. Indeed, using Fubini's theorem and \eqref{e:var_quenched_upp} we can rewrite \eqref{e:decomp_variance} as
\begin{align*}
\E^\e_\Lambda(\psi_x^2)&=\int_0^\infty\zeta^\e_\Lambda\left(G_{\Lambda\setminus \A}(x,x)\ge t\right)\ud t\\
&\le \int_{\tilde\gamma}^\infty\zeta^\e_\Lambda\left(G_{\Lambda\setminus\A}(x,x)\ge t\right)\ud t+\tilde\gamma\\
&\le \int_{\tilde\gamma}^\infty\exp\left(-\frac{\e\exp(16\pi^2(t-\tilde\gamma))}{C|\log\e|^\frac12}\right)\ud t+\tilde\gamma\\
&\le \int_0^\infty\exp\left(-\frac{\e\exp(16\pi^2t)}{C|\log\e|^\frac12}\right)\ud t+C\\
&\le \int_0^{\frac{|\log\e|}{16\pi^2}+\frac{\log|\log\e|}{32\pi^2}}1\ud t+\int_{\frac{|\log\e|}{16\pi^2}+\frac{\log|\log\e|}{32\pi^2}}^\infty\exp\left(-\frac{\e\exp(16\pi^2t)}{C|\log\e|^\frac12}\right)\ud t+C\\
&\le \frac{|\log\e|}{16\pi^2}+\frac{\log|\log\e|}{32\pi^2}+\int_{0}^\infty\exp\left(-\frac{\exp(16\pi^2t)}{C}\right)\ud t+C\\
&\le \frac{|\log\e|}{16\pi^2}+\frac{\log|\log\e|}{32\pi^2}+C\\
&\le \frac{|\log\e|}{16\pi^2}+C\log|\log\e|
\end{align*}
for $\e$ small enough, which establishes the upper bound in \eqref{e:est_var_d=4}. For the lower bound we argue similarly using \eqref{e:var_quenched_low} and obtain
\begin{align*}
\E^\e_\Lambda(\psi_x^2)&=\int_0^\infty\zeta^\e_\Lambda\left(G_{\Lambda\setminus \A}(x,x)\ge t\right)\ud t\\
&\ge \int_0^{\frac{1}{8\pi^2}\log\left(1+d(x,\Z^\dd\setminus\Lambda)-\e^{-\alpha}\right)-\tilde\gamma}\exp\left(-\frac{C_\alpha\e\exp(32\pi^2(t+\tilde\gamma))}{|\log\e|^\frac12}\right)\ud t-C\\
&\ge \int_0^{\min\left(\frac{1}{8\pi^2}\log\left(1+d(x,\Z^\dd\setminus\Lambda)-\e^{-\alpha}\right),\frac{|\log\e|}{32\pi^2}-\frac{\log|\log\e|}{64\pi^2}\right)-\tilde\gamma}\exp\left(-\frac{C_\alpha}{|\log\e|}\right)\ud t-C\,.
\end{align*}
The assumption $d(x,\Z^\dd\setminus\Lambda)\ge \e^{-\alpha}+\e^{-1/4}$ ensures that the second term in the minimum here is smaller than the first, and so we see that indeed
\begin{align*}
\E^\e_\Lambda(\psi_x^2)&\ge\frac{|\log\e|}{32\pi^2}-\frac{\log|\log\e|}{64\pi^2}\left(1-\frac{C_\alpha}{|\log\e|}\right)\\
&\ge\frac{|\log\e|}{32\pi^2}-C_\alpha\log|\log\e|\,.
\end{align*}

\emph{Step 2: Estimates on the covariance}\\
As in \eqref{e:decomp_covariance} we have
\begin{equation}\label{e:decomp_covariance}
|\E^\e_\Lambda(\psi_x\psi_y)|=\left|\sum_{A\subset\Lambda}\zeta^\e_\Lambda(A)\E_{\Lambda\setminus A}(\psi_x\psi_y)\right|\le \sum_{A\subset\Lambda}\zeta^\e_\Lambda(A)|G_{\Lambda\setminus A}(x,y)|\,.
\end{equation}
From Lemma \ref{l:cov_quenched} we know
\[\zeta^\e_\Lambda\left(|G_{\Lambda\setminus\A}(x,y)|\le C_\dd\exp\left(-\frac{|x-y|_\infty}{C_\dd\lmac}\right)\frac{1+\I_{\dd=4}|\log\e|^{5/4}}{\e^{1/2}}\right)\ge 1-\exp\left(-\frac{|x-y|_\infty}{C_\dd\lmac}\right)\,.\]
Abbreviate the event described here by $\Omega$. The decomposition \eqref{e:decomp_covariance} implies
\begin{equation}\label{e:est_covariance}
\begin{aligned}
|\E^\e_\Lambda(\psi_x\psi_y)|&\le\sum_{\substack{A\subset\Lambda\\A\in\Omega}}C_\dd\exp\left(-\frac{|x-y|_\infty}{C_\dd\lmac}\right)\frac{1+\I_{\dd=4}|\log\e|^{5/4}}{\e^{1/2}}+\sum_{\substack{A\subset\Lambda\\A\notin\Omega}}\zeta^\e_\Lambda(A)|G_{\Lambda\setminus A}(x,y)|\\
&\le C_\dd\exp\left(-\frac{|x-y|_\infty}{C_\dd\lmac}\right)\frac{1+\I_{\dd=4}|\log\e|^{5/4}}{\e^{1/2}}+\sum_{\substack{A\subset\Lambda\\A\notin\Omega}}\zeta^\e_\Lambda(A)|G_{\Lambda\setminus A}(x,y)|
\end{aligned}
\end{equation}
and so we only need to bound $|G_{\Lambda\setminus A}(x,y)|$ on the rare event $\Omega^c$. 

If $\dd\ge5$, we can use the bound \[G_{\Lambda\setminus A}(x,y)\le\max\left(G_{\Lambda\setminus A}(x,x),G_{\Lambda\setminus A}(y,y)\right)\le C_\dd\] that follows from \eqref{e:estvariance1} and \eqref{e:identity_G_2} to conclude from \eqref{e:est_covariance} that
\begin{align*}
|\E^\e_\Lambda(\psi_x\psi_y)|&\le C_\dd\exp\left(-\frac{|x-y|_\infty}{C_\dd\lmac}\right)\frac{1}{\e^{1/2}}+C_\dd\zeta^\e_\Lambda(\Omega^c)\\
&\le C_\dd\exp\left(-\frac{|x-y|_\infty}{C_\dd\lmac}\right)\frac{1}{\e^{1/2}}+C_\dd\exp\left(-\frac{|x-y|_\infty}{C_\dd\lmac}\right)\\
&\le \frac{C_\dd}{\e^{1/2}}\exp\left(-\frac{|x-y|_\infty}{C_\dd\lmac}\right)
\end{align*}
which implies \eqref{e:est_cov_d>4}.

If $\dd=4$, the estimate is slighty more complicated, as $G_{\Lambda\setminus A}(x,y)$ is no longer uniformly bounded. Instead we use Lemma \ref{l:var_quenched} to deduce a tail bound on $G_{\Lambda\setminus A}(x,y)$. Note first that if $x=y$ then \eqref{e:est_cov_d=4} follows from \eqref{e:est_var_d=4}, and so we can assume $x\neq y$. By \eqref{e:identity_G_2} and \eqref{e:var_quenched_upp} we have for any $t\ge \tilde\gamma$ that 
\begin{align*}
\zeta^\e_\Lambda\left(|G_{\Lambda\setminus \A}(x,y)|\ge t\right)&\le\zeta^\e_\Lambda\left(\max\left(G_{\Lambda\setminus \A}(x,x),G_{\Lambda\setminus \A}(y,y)\right)\ge t\right)\\
&\le\zeta^\e_\Lambda\left(G_{\Lambda\setminus \A}(x,x)\ge t\right)+\zeta^\e_\Lambda\left(G_{\Lambda\setminus \A}(y,y)\ge t\right)\\
&\le2\exp\left(-\frac{\e\exp(16\pi^2(t-\tilde\gamma))}{C|\log\e|^\frac12}\right)\,.
\end{align*}
We can now use Fubini's theorem to estimate the second summand in \eqref{e:est_covariance} as
\begin{equation}\label{e:est_covariance2}
\begin{aligned}
&\sum_{\substack{A\subset\Lambda\\A\notin\Omega}}\zeta^\e_\Lambda(A)|G_{\Lambda\setminus A}(x,y)|\\
&\quad\le\int_{0}^\infty\zeta^\e_\Lambda\left(|G_{\Lambda\setminus \A}(x,y)|\ge t,\A\notin\Omega\right)\ud t\\
&\quad\le\int_{0}^\infty\min\left(\zeta^\e_\Lambda\left(|G_{\Lambda\setminus \A}(x,y)|\ge t\right),\zeta^\e_\Lambda(\Omega^c)\right)\ud t\\
&\quad\le\int_{\tilde\gamma}^\infty\min\left(\zeta^\e_\Lambda\left(|G_{\Lambda\setminus \A}(x,y)|\ge t\right),\zeta^\e_\Lambda(\Omega^c)\right)\ud t+\int_0^{\tilde\gamma}\zeta^\e_\Lambda(\Omega^c)\ud t\\
&\quad\le\int_{0}^\infty\min\left(2\exp\left(-\frac{\e\exp(16\pi^2t)}{C|\log\e|^\frac12}\right),\exp\left(-\frac{|x-y|_\infty}{C\lmac}\right)\right)\ud t+\tilde\gamma\exp\left(-\frac{|x-y|_\infty}{C\lmac}\right)\,.
\end{aligned}
\end{equation}
To estimate the remaining integral, note that for $a,b<1$ we have $a=\exp(-b\exp(16\pi^2t))$ for $t=t_*:=\frac{1}{16\pi^2}(\log|\log a|+|\log b|)$ and so
\begin{align*}
\int_0^\infty\min(a,\exp(-b\exp(16\pi^2t)))\ud t&=\int_0^{t_*}a\ud t+\int_{t_*}^\infty\exp(-b\exp(16\pi^2t))\ud t\\
&=t_*a+\int_0^\infty\exp(-b\exp(16\pi^2t_*)\exp(16\pi^2t))\ud t\\
&\le t_*a+\int_0^\infty\exp(-b\exp(16\pi^2t_*)(1+16\pi^2t))\ud t\\
&=t_*a+\exp(-b\exp(16\pi^2t_*))\int_0^\infty\exp(-16\pi^2tb\exp(16\pi^2t_*))\ud t\\
&=t_*a+\frac{\exp(-b\exp(16\pi^2t_*))}{16\pi^2b\exp(16\pi^2t_*)}\\
&=t_*a+\frac{a}{16\pi^2|\log a|}\\
&=\frac{1}{16\pi^2}\left(a\log|\log a|+a|\log b|+\frac{a}{|\log a|}\right)\,.\\
&\le Ca(\log|\log a|+|\log b|)
\end{align*}
With the choices $a=\exp\left(-\frac{|x-y|_\infty}{C_\dd\lmac}\right)$ and $b=\frac{\e}{C|\log\e|^{1/2}}$ we then obtain from \eqref{e:est_covariance2} that
\begin{align*}
\sum_{\substack{A\subset\Lambda\\A\notin\Omega}}\zeta^\e_\Lambda(A)|G_{\Lambda\setminus A}(x,y)|&\le C\exp\left(-\frac{|x-y|_\infty}{C\lmac}\right)\left(\log\frac{|x-y|_\infty}{C\lmac}+\left|\log\frac{\e}{C|\log\e|^{1/2}}\right|+1\right)\\
&\le C\exp\left(-\frac{|x-y|_\infty}{C\lmac}\right)\left(\log\frac{|x-y|_\infty}{\lmac}+\log\frac{|\log\e|^{1/2}}{\e}+1\right)\\
&\le C\exp\left(-\frac{|x-y|_\infty}{C\lmac}\right)\left(\log|x-y|_\infty-\log|\log\e|+1\right)
\end{align*}
where we have used that $4\log\frac{1}{\lmac}+\log\frac{|\log\e|^{1/2}}{\e}=-\log|\log\e|$. Finally we can return to \eqref{e:est_covariance} and obtain
\begin{align*}
	|\E^\e_\Lambda(\psi_x\psi_y)|&\le C\exp\left(-\frac{|x-y|_\infty}{C_\dd\lmac}\right)\frac{|\log\e|^{5/4}}{\e^{1/2}}+C\exp\left(-\frac{|x-y|_\infty}{C\lmac}\right)\left(\log|x-y|_\infty-\log|\log\e|+1\right)\\
	&\le \frac{C|\log\e|^{5/4}}{\e^{1/2}}\left(1+\frac{\e^{1/2}|x-y|_\infty}{|\log\e|^{5/4}}\right)\exp\left(-\frac{|x-y|_\infty}{C\lmac}\right)\\
	&\le \frac{C|\log\e|^{5/4}}{\e^{1/2}}\left(1+\frac{|x-y|_\infty}{\lmac}\right)\exp\left(-\frac{|x-y|_\infty}{C\lmac}\right)\\
	&\le \frac{C|\log\e|^{5/4}}{\e^{1/2}}\exp\left(-\frac{|x-y|_\infty}{C\lmac}\right)
\end{align*}
which implies \eqref{e:est_cov_d=4}.

Finally, the estimates \eqref{e:est_mass_d>4} and \eqref{e:est_mass_d=4} are straightforward consequences of \eqref{e:est_cov_d>4} and \eqref{e:est_cov_d=4}, respectively.
\end{proof}
\subsection{Existence of the thermodynamic limit of the field}\label{s:thermo_limit_field}
It remains to prove the existence of the thermodynamic limit of the pinned field. This is significantly more difficult than the existence of the thermodynamic limit of the set of pinned points, as we do not have correlation inequalities for the field or a random walk representation. Instead we show by hand that the exponential decay of correlations implies convergence of $\E^\e_\Lambda(f)$ for any bounded local $f$.
\begin{proof}[Proof of Theorem \ref{t:thermolimit}, second part]
As in the proof of the first part it suffices to check that the limit $\lim_{\Lambda\nearrow\Z^\dd}\E^\e_\Lambda(f)$ exists for any bounded local function $f\colon \Z^\dd\to\R$. Our tail estimates on $\PP^\e_\Lambda$ easily imply boundedness of $\E^\e_\Lambda(f)$, so if the limit exists it is finite.

So let a local function $f$ be given. Suppose that $f$ only depends on the values of $\psi$ in $Q_r(0)$ for some $r$. We can assume that $r\ge \hat N_\dd$. Let $R\in\N$ with $R\ge 8r$. We set $\Lambda'=Q_R(0)$. Let $\Omega$ be the event described in \eqref{e:diff_covs_quenched}. Let also $k\in\N$ and consider the event $\tilde\Omega_{0,k}$ from \eqref{e:event_var_bounded} (with $y=0$). Note that if $A\in\tilde\Omega_{0,k}$ we have
\[G_{\Lambda\setminus A}(0,0)\le C_\dd\left(1+\I_{ \dd=4}(\log k+|\log\e|)\right)\,.\]
Similarly as for \eqref{e:diff_covs_quenched4}, we see that this implies for any $x\in Q_r(0)$
\[G_{\Lambda\setminus A}(x,x)\le C_\dd\left(1+\I_{ \dd=4}(\log r+\log k+|\log\e|)\right)=:X_{\dd,\e,k,r}\]
and in combination with \eqref{e:identity_G_2} also
\[\max_{x,y\in Q_r(0)}\left|G_{\Lambda\setminus A}(x,y)\right|\le X_{\dd,\e,k,r}\,.
\]

We can now write
\begin{equation}\label{e:thermolimit2}
\E^\e_\Lambda(f)=\sum_{A\subset\Lambda}\E_{\Lambda\setminus A}(f)\zeta^\e_\Lambda(A)=\sum_{\substack{A\subset\Lambda\\A\in\Omega\cap\tilde\Omega_{0,k}}}\E_{\Lambda\setminus A}(f)\zeta^\e_\Lambda(A)+\sum_{\substack{A\subset\Lambda\\A\notin\Omega\cap\tilde\Omega_{0,k}}}\E_{\Lambda\setminus A}(f)\zeta^\e_\Lambda(A)\,.
\end{equation}
The second summand here is an error term that is easy to estimate. We have
\begin{equation}\label{e:thermolimit3}
\begin{aligned}
\left|\sum_{\substack{A\subset\Lambda\\A\notin(\Omega\cap\tilde\Omega_{0,k})}}\E_{\Lambda\setminus A}(f)\zeta^\e_\Lambda(A)\right|&\le\|f\|_{L^\infty(\Z^\dd)}\sum_{\substack{A\subset\Lambda\\A\notin\Omega\cap\tilde\Omega_{0,k}}}\zeta^\e_\Lambda(A)\\
&\le\|f\|_{L^\infty(\Z^\dd)}\left(\zeta^\e_\Lambda(\Omega^c)+\zeta^\e_\Lambda(\tilde\Omega_{0,k}^c)\right)\\
&\le \left(C_\dd r^\dd\exp\left(-\frac{R-r}{C_\dd\lmac}\right)+\frac{1}{2^k}\right)\|f\|_{L^\infty(\Z^\dd)}
\end{aligned}
\end{equation}
and note that the right-hand side tends to 0 as $k,R\to\infty$, uniformly in $\Lambda\supset Q_R(0)$.

Next, we begin to analyse the main term in \eqref{e:thermolimit2}, i.e. the first summand. $\PP_{\Lambda\setminus A}$ is the law of a multivariate Gaussian measure. Thus, $\E_{\Lambda\setminus A}(f)$ depends only on the variances and covariances of that measure. Because of the locality of $f$ it depends only on those variances and covariances where the sites are in $Q_r(0)$. In particular, $\E(f)$ is a continuous function of $\left(G_{\Lambda\setminus A}(x,y)\right)_{x,y\in Q_r(0)}\in\R^{Q_r(0)\times Q_r(0)}$. If we restrict it to the compact set $[-X_{\dd,\e,k,r},X_{\dd,\e,k,r}]^{Q_r(0)\times Q_r(0)}$, it is uniformly continuous.

From Lemma \ref{l:diff_covs_quenched} we know that for $A\in\Omega$
\[\left|\left(G_{\Lambda\setminus A}(x,y)\right)_{x,y\in Q_r(0)}-\left(G_{Q_R(0)\setminus A}(x,y)\right)_{x,y\in Q_r(0)}\right|_{\infty}\le C_\dd\exp\left(-\frac{R-r}{C_\dd\lmac}\right)\frac{1+\I_{\dd=4}|\log\e|^{5/4}}{\e^{1/2}}\]
and the right-hand side tends to 0 as $R\to\infty$. Moreover, we have for $A\in\tilde\Omega_{0,k}$ that 
\[\left(G_{\Lambda\setminus A}(x,y)\right)_{x,y\in Q_r(0)}\in [-X_{\dd,\e,k,r},X_{\dd,\e,k,r}]^{Q_r(0)\times Q_r(0)}\,.\]
Thus, the uniform continuity of $\E(f)$ implies that there is a function $\omega_{\dd,\e,f,k,r}(R)$ (independent of $\Lambda$) with $\lim_{R\to\infty}\omega_{\dd,\e,f,k,r}(R)=0$ such that for all $A\in\Omega\cap\tilde\Omega_{0,k}$
\[\left|\E_{\Lambda\setminus A}(f)-\E_{Q_R(0)\setminus A}(f)\right|\le \omega_{\dd,\e,f,k,r}(R)\,.\]
This implies for the first summand in \eqref{e:thermolimit2} that
\begin{equation}\label{e:thermolimit4}
\begin{aligned}
&\left|\sum_{\substack{A\subset\Lambda\\A\in\Omega\cap\tilde\Omega_{0,k}}}\E_{\Lambda\setminus A}(f)\zeta^\e_\Lambda(A)-\sum_{A\subset\Lambda}\E_{Q_R(0)\setminus A}(f)\zeta^\e_\Lambda(A)\right|\\
&\quad\le\left|\sum_{\substack{A\subset\Lambda\\A\in\Omega\cap\tilde\Omega_{0,k}}}\E_{\Lambda\setminus A}(f)\zeta^\e_\Lambda(A)-\sum_{\substack{A\subset\Lambda\\A\in\Omega\cap\tilde\Omega_{0,k}}}\E_{Q_R(0)\setminus A}(f)\zeta^\e_\Lambda(A)\right|+\left|\sum_{\substack{A\subset\Lambda\\A\notin\Omega\cap\tilde\Omega_{0,k}}}\PP_{Q_R(0)\setminus A}(f)\zeta^\e_\Lambda(A)\right|\\
&\quad\le\sum_{\substack{A\subset\Lambda\\A\in\Omega\cap\tilde\Omega_{0,k}}}\omega_{\dd,\e,f,k,r}(R)\zeta^\e_\Lambda(A)+\|f\|_{L^\infty(\Z^\dd)}\left(\zeta^\e_\Lambda(\Omega^c)+\zeta^\e_\Lambda(\tilde\Omega_{0,k}^c)\right)\\
&\quad\le \omega_{\dd,\e,f,k,r}(R)+\left(C_\dd r^\dd\exp\left(-\frac{R-r}{C_\dd\lmac}\right)+\frac{1}{2^k}\right)\|f\|_{L^\infty(\Z^\dd)}
\end{aligned}
\end{equation}
where we have estimated the error term the same way as in \eqref{e:thermolimit3}.

We also know that
\begin{equation}\label{e:thermolimit5}
\begin{aligned}
\sum_{A\subset\Lambda}\E_{Q_R(0)\setminus A}(f)\zeta^\e_\Lambda(A)&=\sum_{A'\subset Q_R(0)}\E_{Q_R(0)\setminus A}(f)\sum_{A''\subset\Lambda\setminus Q_R(0)}\zeta^\e_\Lambda(A'\cup A'')\\
&=\sum_{A'\subset Q_R(0)}\E_{Q_R(0)\setminus A'}(f)\zeta^\e_\Lambda(\A\cap Q_R(0)=A')\,.
\end{aligned}
\end{equation}
Putting \eqref{e:thermolimit2}, \eqref{e:thermolimit3}, \eqref{e:thermolimit4}, \eqref{e:thermolimit5} together, we find
\begin{equation}\label{e:thermolimit6}
\begin{aligned}
&\left|\E^\e_\Lambda(f)-\sum_{A'\subset Q_R(0)}\E_{Q_R(0)\setminus A'}(f)\zeta^\e_\Lambda(\A\cap Q_R(0)=A')\right|\\
&\quad\le\omega_{\dd,\e,f,k,r}(R)+2\left(C_\dd r^\dd\exp\left(-\frac{R-r}{C_\dd\lmac}\right)+\frac{1}{2^k}\right)\|f\|_{L^\infty(\Z^\dd)}\,.
\end{aligned}
\end{equation}
We now want to take the limits $\Lambda\nearrow\Z^\dd$, $R\to\infty$, $k\to\infty$ in that order. For that purpose, note that the weak convergence of $\zeta^\e_\Lambda$ to $\zeta^\e$ implies that $\lim_{\Lambda\nearrow\Z^\dd}\zeta^\e_\Lambda(\A\cap Q_R(0)=A')=\zeta^\e(\A\cap Q_R(0)=A')$, and so \eqref{e:thermolimit6} implies
\[\limsup_{k\to\infty}\limsup_{R\to\infty}\limsup_{\Lambda\nearrow\Z^\dd}\left|\E^\e_\Lambda(f)-\sum_{A'\subset Q_R(0)}\E_{Q_R(0)\setminus A'}(f)\zeta^\e(\A\cap Q_R(0)=A')\right|=0\,.\]
From this we see that \[\lim_{\Lambda\nearrow\Z^\dd}\E^\e_\Lambda(f)=\lim_{R\to\infty}\sum_{A'\subset Q_R(0)}\E_{Q_R(0)\setminus A'}(f)\zeta^\e(\A\cap Q_R(0)=A')\]
and that in particular both limits exist. This is what we wanted to show.

\end{proof}

\nocite{Bolthausen2001,Bolthausen2017,Velenik2006,Sakagawa2018, Deuschel2000,Ioffe2000,Bolthausen1999,Hoefer2018,Widman1971, Cioranescu1997,Marchenko2006,Sakagawa2012,Niethammer2006}

\paragraph{Acknowledgements}
The author wants to thank Simon Buchholz, Richard M. Höfer, Stefan Müller and Juan J.L. Velázquez for inspiring discussions, Yvan Yelenik for some comments on Remark \ref{r:gap_IV}, and Stefan Müller also for suggesting various improvements to the present paper.

The author was supported by the DFG (German Research Foundation) via the Hausdorff Center for Mathematics (GZ 2047/1, Projekt-ID 3906$ $85813) and via the SFB 1060 "The Mathematics of Emergent Effects" (Projekt-ID 21150$ $4053), and also by the German National Academic Foundation.

A preliminary version of this article will appear in the author's PhD thesis \cite{Schweiger2020b}.

\bibliographystyle{alpha_edited2}
\bibliography{PinningMembrane}

\begin{thebibliography}{DMRR92}

\bibitem[AS64]{Abramowitz1964}
M. Abramowitz and I.~A. Stegun.
\newblock {\em Handbook of mathematical functions with formulas, graphs, and
  mathematical tables}, volume~55 of {\em National Bureau of Standards Applied
  Mathematics Series}.
\newblock U.S. Government Printing Office, Washington, D.C., 1964.

\bibitem[BB01]{Bolthausen1999}
E. Bolthausen and D. Brydges.
\newblock Localization and decay of correlations for a pinned lattice free
  field in dimension two.
\newblock In {\em State of the art in probability and statistics ({L}eiden,
  1999)}, volume~36 of {\em IMS Lecture Notes Monogr. Ser.}, pages 134--149.
  Inst. Math. Statist., Beachwood, OH, 2001.

\bibitem[BCK16]{Bolthausen2016}
E. {Bolthausen}, A. {Cipriani}, and N. {Kurt}.
\newblock {Fast decay of covariances under $\delta-$pinning in the critical and
  supercritical membrane model}, 2016.
\newblock arXiv:1601.01513.

\bibitem[BCK17]{Bolthausen2017}
E. Bolthausen, A. Cipriani, and N. Kurt.
\newblock Exponential decay of covariances for the supercritical membrane
  model.
\newblock {\em Comm. Math. Phys.}, 353(3):1217--1240, 2017.

\bibitem[BDKS19]{Buchholz2019}
S. Buchholz, J.-D. Deuschel, N. Kurt, and F. Schweiger.
\newblock Probability to be positive for the membrane model in dimensions 2 and
  3.
\newblock {\em Electron. Commun. Probab.}, 24:Paper No. 44, 14 pp., 2019.

\bibitem[BDZ00]{Bolthausen2000}
E. Bolthausen, J.~D. Deuschel, and O. Zeitouni.
\newblock Absence of a wetting transition for a pinned harmonic crystal in
  dimensions three and larger.
\newblock {\em J. Math. Phys.}, 41(3):1211--1223, 2000.

\bibitem[BV01]{Bolthausen2001}
E. Bolthausen and Y. Velenik.
\newblock Critical behavior of the massless free field at the depinning
  transition.
\newblock {\em Comm. Math. Phys.}, 223(1):161--203, 2001.

\bibitem[CCH16]{Chiarini2016}
A. Chiarini, A. Cipriani, and R.~S. Hazra.
\newblock Extremes of some {G}aussian random interfaces.
\newblock {\em J. Stat. Phys.}, 165(3):521--544, 2016.

\bibitem[CD08]{Caravenna2008}
F. Caravenna and J.-D. Deuschel.
\newblock Pinning and wetting transition for {$(1+1)$}-dimensional fields with
  {L}aplacian interaction.
\newblock {\em Ann. Probab.}, 36(6):2388--2433, 2008.

\bibitem[CD09]{Caravenna2009}
F. Caravenna and J.-D. Deuschel.
\newblock Scaling limits of {$(1+1)$}-dimensional pinning models with
  {L}aplacian interaction.
\newblock {\em Ann. Probab.}, 37(3):903--945, 2009.

\bibitem[CDH19]{Cipriani2019}
A. Cipriani, B. Dan, and R.~S. Hazra.
\newblock The scaling limit of the membrane model.
\newblock {\em Ann. Probab.}, 47(6):3963--4001, 2019.

\bibitem[CM97]{Cioranescu1997}
D. Cioranescu and F. Murat.
\newblock A strange term coming from nowhere.
\newblock In {\em Topics in the mathematical modelling of composite materials},
  volume~31 of {\em Progr. Nonlinear Differential Equations Appl.}, pages
  45--93. Birkh\"{a}user Boston, Boston, MA, 1997.

\bibitem[CV00]{Caputo2000}
P. Caputo and Y. Velenik.
\newblock A note on wetting transition for gradient fields.
\newblock {\em Stochastic Process. Appl.}, 87(1):107--113, 2000.

\bibitem[CZ13]{Cazacu2013}
C. Cazacu and E. Zuazua.
\newblock Improved multipolar {H}ardy inequalities.
\newblock In {\em Studies in phase space analysis with applications to {PDE}s},
  volume~84 of {\em Progr. Nonlinear Differential Equations Appl.}, pages
  35--52. Birkh\"{a}user/Springer, New York, 2013.

\bibitem[DMRR92]{Dunlop1992}
F. Dunlop, J. Magnen, V. Rivasseau, and P. Roche.
\newblock Pinning of an interface by a weak potential.
\newblock {\em J. Statist. Phys.}, 66(1-2):71--98, 1992.

\bibitem[DV00]{Deuschel2000}
J.-D. Deuschel and Y. Velenik.
\newblock Non-{G}aussian surface pinned by a weak potential.
\newblock {\em Probab. Theory Related Fields}, 116(3):359--377, 2000.

\bibitem[Fun05]{Funaki2005}
T. Funaki.
\newblock Stochastic interface models.
\newblock In {\em Lectures on probability theory and statistics}, volume 1869
  of {\em Lecture Notes in Math.}, pages 103--274. Springer, Berlin, 2005.

\bibitem[GM12]{Giaquinta2012}
M. Giaquinta and L. Martinazzi.
\newblock {\em An introduction to the regularity theory for elliptic systems,
  harmonic maps and minimal graphs}, volume~11 of {\em Appunti. Scuola Normale
  Superiore di Pisa (Nuova Serie)}.
\newblock Edizioni della Normale, Pisa, second edition, 2012.

\bibitem[HV18]{Hoefer2018}
R.~M. H\"{o}fer and J.~J.~L. Vel\'{a}zquez.
\newblock The method of reflections, homogenization and screening for {P}oisson
  and {S}tokes equations in perforated domains.
\newblock {\em Arch. Ration. Mech. Anal.}, 227(3):1165--1221, 2018.

\bibitem[IV00]{Ioffe2000}
D. Ioffe and Y. Velenik.
\newblock A note on the decay of correlations under {$\delta$}-pinning.
\newblock {\em Probab. Theory Related Fields}, 116(3):379--389, 2000.

\bibitem[Kur07]{Kurt2007}
N. Kurt.
\newblock Entropic repulsion for a class of {G}aussian interface models in high
  dimensions.
\newblock {\em Stochastic Process. Appl.}, 117(1):23--34, 2007.

\bibitem[Kur09]{Kurt2009}
N. Kurt.
\newblock Maximum and entropic repulsion for a {G}aussian membrane model in the
  critical dimension.
\newblock {\em Ann. Probab.}, 37(2):687--725, 2009.

\bibitem[LM17]{Latala2017}
R. Lata{\l}a and D. Matlak.
\newblock Royen's proof of the {G}aussian correlation inequality.
\newblock In {\em Geometric aspects of functional analysis}, volume 2169 of
  {\em Lecture Notes in Math.}, pages 265--275. Springer, Cham, 2017.

\bibitem[MK06]{Marchenko2006}
V.~A. Marchenko and E.~Y. Khruslov.
\newblock {\em Homogenization of partial differential equations}, volume~46 of
  {\em Progress in Mathematical Physics}.
\newblock Birkh\"{a}user Boston, Inc., Boston, MA, 2006.

\bibitem[MS19]{Mueller2019}
S. M\"{u}ller and F. Schweiger.
\newblock Estimates for the {G}reen's function of the discrete bilaplacian in
  dimensions 2 and 3.
\newblock {\em Vietnam J. Math.}, 47(1):133--181, 2019.

\bibitem[NV06]{Niethammer2006}
B. Niethammer and J.~J.~L. Vel\'{a}zquez.
\newblock Screening in interacting particle systems.
\newblock {\em Arch. Ration. Mech. Anal.}, 180(3):493--506, 2006.

\bibitem[Roy14]{Royen2014}
T. Royen.
\newblock A simple proof of the {G}aussian correlation conjecture extended to
  some multivariate gamma distributions.
\newblock {\em Far East J. Theor. Stat.}, 48(2):139--145, 2014.

\bibitem[Sak03]{Sakagawa2003}
H. Sakagawa.
\newblock Entropic repulsion for a {G}aussian lattice field with certain finite
  range interaction.
\newblock {\em J. Math. Phys.}, 44(7):2939--2951, 2003.

\bibitem[Sak12]{Sakagawa2012}
H. Sakagawa.
\newblock On the free energy of a {G}aussian membrane model with external
  potentials.
\newblock {\em J. Stat. Phys.}, 147(1):18--34, 2012.

\bibitem[Sak18]{Sakagawa2018}
H. Sakagawa.
\newblock Localization of a {G}aussian membrane model with weak pinning
  potentials.
\newblock {\em ALEA Lat. Am. J. Probab. Math. Stat.}, 15(2):1123--1140, 2018.

\bibitem[Sch20a]{Schweiger2020}
F. Schweiger.
\newblock The maximum of the four-dimensional membrane model.
\newblock {\em Ann. Probab.}, 48(2):714--741, 2020.

\bibitem[Sch20b]{Schweiger2020b}
F. Schweiger.
\newblock {\em On the membrane model and the Green's function of the discrete
  Bilaplacian}.
\newblock PhD thesis, Rheinische Friedrich-Wilhelms-Universit{\"a}t Bonn, 2020.

\bibitem[Vel06]{Velenik2006}
Y. Velenik.
\newblock Localization and delocalization of random interfaces.
\newblock {\em Probab. Surv.}, 3:112--169, 2006.

\bibitem[Wid71]{Widman1971}
K.-O. Widman.
\newblock H\"{o}lder continuity of solutions of elliptic systems.
\newblock {\em Manuscripta Math.}, 5:299--308, 1971.

\end{thebibliography}

\end{document}